\documentclass[reqno, 11pt]{amsart}

\usepackage[T1]{fontenc} 
\usepackage[latin1]{inputenc} 
\usepackage{mathptmx} 
\usepackage{lmodern} 
\usepackage{amsthm} 
\usepackage{amsmath} 
\usepackage{amsfonts} 
\usepackage{amssymb} 
\usepackage{mathrsfs}
\usepackage{bbm}
\usepackage[left=2.25cm, right=2.25cm, top=2.5cm, bottom=2.5cm]{geometry} 
\usepackage[ngerman,german,english]{babel} 
\usepackage{paralist} 
\usepackage{graphicx} 
\usepackage[usenames,dvipsnames]{color}
\usepackage{etoolbox}
\usepackage{caption}
\usepackage{subcaption}

\usepackage[colorlinks=true, pdfstartview=FitV, linkcolor=BrickRed,citecolor=black, urlcolor=black]{hyperref}
\hypersetup{linktocpage}

\numberwithin{equation}{section}
\theoremstyle{plain} 
\newtheorem{thm}{Theorem}[section] 
\newtheorem{cor}[thm]{Corollary}
\newtheorem{lem}[thm]{Lemma}
\newtheorem{pro}[thm]{Proposition}

\theoremstyle{definition}
\newtheorem{defi}[thm]{Definition}
\newtheorem{rem}[thm]{Remark}

\theoremstyle{remark}

\newcommand{\N}{\mathbb{N}}

\newcommand{\R}{\mathbb{R}}
\newcommand{\Z}{\mathbb{Z}}
\newcommand{\F}{\mathcal{F}}

\newcommand{\D}{\mathcal{D}}

\newcommand{\Act}{\mathcal{A}}
\newcommand{\Ham}{\mathcal{H}}

\newcommand{\SobH}{\mathscr{H}}
\renewcommand{\S}{\mathcal{S}}

\newcommand{\cX}{\mathcal{X}}
\newcommand{\cV}{\mathcal{V}}
\newcommand{\eps}{\varepsilon}

\newcommand{\bx}{\mathbf{x}}
\newcommand{\bv}{\mathbf{v}}

\DeclareMathOperator{\arcosh}{arcosh}

\newcommand\restr[2]{{
		\left.\kern-\nulldelimiterspace 
		#1 
		\vphantom{\big|} 
		\right|_{#2} 
}}

\newcommand{\norm}[2][]{\left\|#2\right\|_{#1}}

\newcommand{\skp}[2]{\left\langle #1,#2 \right\rangle}

\newcommand{\jabr}[1]{\left\langle #1 \right\rangle}

\newcommand{\Lpnorm}[2][]{\ifthenelse{\equal{#1}{}}{\norm{#2}_{L^p}}{\norm{#2}_{L^p(#1)}}}
\newcommand{\Hknorm}[2][]{\ifthenelse{\equal{#1}{}}{\norm{#2}_{H^k}}{\norm{#2}_{H^k(#1)}}}

\newcommand{\supp}{\text{\normalfont supp}}

\newcommand{\sgn}{\mathrm{sgn}}

\newcommand{\QV}[2][]{\ifthenelse{\equal{#1}{}}{\langle #1 \rangle}{\langle #1,#2 \rangle}}

\newcommand{\ind}{\mathbbm{1}}

\DeclareMathOperator{\eff}{\operatorname{eff}}

\title[Modified scattering in Vlasov-Poisson near an attractive point mass]{\vspace*{-1cm}Modified scattering dynamics in the Vlasov-Poisson equation near an attractive point mass}

\author[B. Kepka]{Bernhard Kepka}
\address{Institute of Mathematics, University of Zurich}
\email{bernhard.kepka@math.uzh.ch}
\author[K. Widmayer]{Klaus Widmayer}
\address{Faculty of Mathematics, University of Vienna \& Institute of Mathematics, University of Zurich}
\email{klaus.widmayer@univie.ac.at \& klaus.widmayer@math.uzh.ch}



\begin{document}

\begin{abstract}
We study the long-time behavior of radially symmetric solutions to the Vlasov-Poisson equation consisting of an attractive point mass and a small, suitably localized and absolutely continuous distribution of particles: if the latter is initially localized on hyperbolic trajectories for the associated Kepler problem, we obtain global in time, unique Lagrangian solutions that asymptotically undergo a modified scattering dynamics (in the sense of distributions). A key feature of this result is its low regularity regime, which does not make use of derivative control, but can be upgraded to strong solutions and strong convergence by propagation of regularity. 
\end{abstract}
\maketitle

\tableofcontents

\section{Introduction}\label{sec:Introduction}
A classical model for galactic dynamics is the Vlasov-Poisson system 
\begin{align}\label{eq:VP}\tag{VP}
	\partial_t F + \bv\cdot \nabla_{\bx} F - \nabla_{\bx} \phi \cdot \nabla_{\bv} F =0,
	\quad
	\Delta \phi(t,\bx) = 4\pi\int F(t,\bx,\bv)\, d\bv, \quad \lim_{|\bx|\to \infty} \phi(t,\bx) =0,
\end{align}
which describes the dynamics of a particle distribution function $F:\R\times\R^3\times\R^3\to\R_+$ subject to its self-generated gravitational field $\nabla_\bx\phi(t,\bx)$. We refer to \cite{BinneyTremaine2008,Rein2007} for a physical and mathematical introduction to the Vlasov-Poisson equation in galactic dynamics. In this article, we investigate the dynamics of \eqref{eq:VP} near an attractive point mass. Concretely, we consider solutions of the form
\begin{equation}\label{eq:F-shape}
	F(t,d\bx,d\bv) = m_p\delta_{(\cX(t),\cV(t))}(d\bx,d\bv) + \lambda f(t,\bx,\bv) d\bx d\bv,
\end{equation}
where $m_p>0$ and $(\cX(t),\cV(t))\in\R^3\times\R^3$ are the mass and location of the point mass, respectively, and $\lambda>0$ is the specific mass of an absolutely continuous distribution $f(t,\bx,\bv)\geq 0$. One correspondingly decomposes the gravitational potential as
\begin{align*}
	\phi = \phi_p + \phi_g, \quad \phi_p(t,\bx) = - \dfrac{m_p}{|\bx-\cX(t)|},
	\quad
	\Delta \phi_g(t,\bx) = 4\pi\int f(t,\bx,\bv)\, d\bv,
\end{align*}
and formally obtains the coupled system
\begin{equation}\label{eq:VPD}
\begin{aligned}
	\left(  \partial_t + \bv\cdot\nabla_\bx - \dfrac{m}{2} \dfrac{\bx-\cX(t)}{|\bx-\cX(t)|^3} \cdot \nabla_\bv \right) f - \lambda \nabla_\bx \phi_g \cdot \nabla_\bv f &=0,\quad \Delta \phi_g(t,\bx) = 4\pi\int f(t,\bx,\bv)\, d\bv,\\
    \frac{d\cX(t)}{dt}=\cV(t),\qquad \frac{m}{2}\frac{d\cV(t)}{dt}&=-\nabla\phi_g(\cX(t),t),
\end{aligned}    
\end{equation}
where $ m:=2m_p >0$ and again we impose that $\lim_{|\bx|\to \infty} \phi_g(t,\bx) =0$. 

Our main result -- stated first in a simplified fashion in Theorem \ref{thm:main_intro} -- captures the global in time stability of \eqref{eq:VPD} for radially symmetric, sufficiently small and \emph{suitably localized} initial perturbations $f_0(\bx,\bv)$. Here, by radial symmetry we mean the natural $\mathcal{O}(3)$-symmetry of \eqref{eq:VPD} with $\cX(0)=\cV(0)=0$, i.e.\ the invariance under transformations $(\bx,\bv) \mapsto (O\bx,O\bv)$ for $ O\in \mathcal{O}(3) $, which dynamically persists if initially imposed.

\begin{thm}\label{thm:main_intro}
Let $f_0\in L^\infty_c(\D_0)$ be radially symmetric, where $\D_0:=\{|\bv|^2 >m|\bx|^{-1}\}$. 
\begin{enumerate}[(i)]
    \item\label{it:global} (Global existence and uniqueness) Then there exists $\eps_0>0$ such that for all $0<\eps\leq \eps_0$, there is a unique global Lagrangian solution to \eqref{eq:VPD} with $f(0,\bx,\bv)=\eps f_0(\bx,\bv)$, and $f\in C_tL^\infty_{\bx,\bv}$ with $\cX(t)=\cV(t)=0$, i.e.\ the characteristics for $f$ are well-defined and $f$ is the unique weak solution given by transport of $f_0$ along them. 
    \item\label{it:aysmpt} (Asymptotic behavior) Moreover, there exists a Lagrangian map $ (\mathscr{X},\mathscr{V}): [0,\infty)\times\D_0 \to \D_0 $ and an asymptotic profile $ f_\infty\in L^\infty_c(\D_0) $ such that for all $\varphi\in C^\infty_c(\D_0)$
	\begin{equation}\label{eq:asympt_main}
		\jabr{f\left( t, \mathscr{X}(t,\bx,\bv),\mathscr{V}(t,\bx,\bv)\right),\varphi}_{L^2} \to \jabr{f_\infty(\bx,\bv),\varphi}_{L^2},\qquad t\to\infty.
	\end{equation}
    \item\label{it:strong} (Strong solutions for smoother initial data) If $f\in C^1_c(\D_0)$, then the above solution is in fact a strong solution $C_tC^1_{\bx,\bv}$, and the asymptotic convergence \eqref{eq:asympt_main} holds strongly in $L^2$.
\end{enumerate}
\end{thm}
As it turns out, the dynamics of \eqref{eq:VPD} are more adequately captured in adapted spherical resp.\ action-angle variables, and we give the corresponding, more precise version of the main statement below in Theorem \ref{thm:IntrodMainThmAAVar}. For now, let us highlight four key features of this result:
\begin{enumerate}[(1)]
    \item (Open trajectories and modified scattering) The assumption that the initial data be supported on $\D_0$ is essential for our theorem: This guarantees that the dynamics start on hyperbolic trajectories of the linearized flow, which in turn is simply the Hamiltonian flow with respect to the Keplerian Hamiltonian
    \begin{align*}
	\Ham_{lin}(\bx,\bv)= \dfrac{1}{2}|\bv|^2 - \dfrac{m}{2}\dfrac{1}{|\bx|},
    \end{align*}
    see also Section \ref{sec:intro_lin+aa}. For simplicity, we have quantified this support restriction here using a compactness assumption, but this can be relaxed and made more precise using suitable moments -- see Theorem \ref{thm:IntrodMainThmAAVar} below. Our result then shows that the resulting dynamics in the nonlinear problem, the characteristic system of which is the Hamiltonian flow of
    \begin{align*}
	\Ham(t,\bx,\bv) = \Ham_{lin}(\bx,\bv) + \lambda \phi_g(t,\bx),
    \end{align*}
    remain in the region $\D_0=\{\Ham_{lin}>0\}$. Furthermore, $\mu$ undergoes a modified scattering dynamic due to the long-range effects of the gravitational potential $\phi_g$: the asymptotic behavior \eqref{eq:asympt_main} is given by a logarithmic correction to the linearized dynamics -- we refer to Remark \ref{rem:asympt} for more details. To the best of our knowledge, Theorem \ref{thm:main_intro} is then the first global in time stability result for dynamics near an attractive point mass in the Vlasov-Poisson equations. Contrast this with the setting of elliptic, bound orbits $\Ham_{lin}(\bx,\bv)<0$, where only local in time stability is known \cite{CL2024}, see also the discussion below.
    
    \item (Weak topology) A key feature of our result is that unique, global solutions can be constructed already for bounded and localized initial data (without any derivative assumptions), and that moreover their asymptotic behavior can be isolated in the form of the weak distributional convergence in \eqref{eq:asympt_main} (which could be quantified more precisely). Our functional setting thus allows to isolate the asymptotic behavior of patch type solutions, for example. Overall, this reinforces the perspective of \eqref{eq:VP} as a transport equation for measures. We expect the relevant arguments to apply more broadly, in particular also to the Vlasov-Poisson equations near vacuum. We also remark that propagation of regularity holds in more generality, for higher derivatives than stated in \eqref{it:strong}. 

    \item (Radial symmetry and angular momentum) The assumption of radial symmetry about the point mass reduces the degrees of freedom, and in particular ensures that the point mass remains at rest at its initial location. With our choice of coordinates, this is simply the origin, and leaves as dynamic variables $(r,v,\ell)\in\R_+\times\R\times\R$, where $\ell$ denotes the angular momentum -- see Section \ref{sec:intro_symm} below for more details. Since (as in the classical Kepler problem) angular momentum is conserved, $\ell$ largely plays the role of a parameter, and the dynamics essentially reduce to those on a $1+1$-dimensional phase space. However, the precise support properties with respect to $\ell$ still play an important role in the quantitative understanding of the gravitational force field -- see Remark \ref{rem:lweights} for more on this. 

    \item (Methodology) The approach of the present paper builds on the ``method of asymptotic actions'' developed in the context of the repulsive setting \cite{PW2020,PWY2022}, see also the discussion just below. Whereas in \cite{PW2020} radial symmetry with vanishing angular momentum was considered, here we extend the methodology to the case of general angular momenta (but do not treat the case without symmetry assumptions as in \cite{PWY2022}). This shows that the Hamiltonian or symplectic tools developed in the aforementioned can be adapted and extended to a setting with three degrees of freedom.
\end{enumerate}

\paragraph{\bf Context.}
As a fundamental model for collisionless dynamics in galaxies and in plasmas, the Vlasov-Poisson system \eqref{eq:VP} for particle distribution functions of the form \eqref{eq:F-shape} with $m_p=0$, $\lambda\in\R$, has been widely studied, and the corresponding literature is too vast to be surveyed here adequately. We highlight instead some more directly relevant aspects and limit our attention to the setting of three space and three velocity dimensions $(\bx,\bv)\in\R^3\times\R^3$. Here, classical works have established the global well-posedness of sufficiently smooth and localized solutions \cite{BD1985,GI2020,LP1991,Pfa1992,Sch1991}. Under milder assumptions, weak solutions have been shown to exist globally e.g.\ in \cite{DPL1988}, and their Lagrangian nature (i.e.\ they are given as transport of the initial data along a well-defined characteristic flow) has been established in \cite{ACF17}, while criteria for uniqueness can be found e.g.\ in \cite{Loeper2006UniquenessVP,Mio2016}. However, there are very few results concerning long-time dynamics. Recent works \cite{CK2016,FOPW2021,IPWW2020,Pan2020} have started to address aspects of this, and provide a precise description of asymptotic behavior in particular near vacuum. (See also \cite{Big2022,Bret2025,PBA2023} for related results on the Vlasov-Maxwell equations.)

The presence of a point charge or mass, i.e.\ $m_p\neq 0$ in \eqref{eq:F-shape}, introduces a singular force field and severe analytical challenges, rendering many classical results inapplicable. Nevertheless, in the repulsive setting ($m_p<0$), parallels to the classical theory have been developed, including the global well-posed of strong solutions under suitable support restrictions \cite{MMP2011} (see also \cite{CM2010}), the existence of global weak solutions \cite{DMS2015,LZ2017,LZ2018} and global Lagrangian solutions \cite{CLS2018}. Moreover, the global stability of a point charge has been investigated in \cite{PW2020,PWY2022}, identifying the asymptotic as a modified scattering dynamic.

In the present setting of an attractive point mass, i.e.\ $m_p>0$, strong well-posedness is not known (not even locally in time), but global weak solutions have been constructed \cite{CMMP2012,CZW2015}. Theorem \ref{thm:main_intro} establishes the global nonlinear stability of an attractive point mass under a certain class of radially symmetric perturbations.
Our support restriction to (linearly) hyperbolic trajectories guarantees that the main mechanism of stability is dispersion, which in particular implies the decay of the gravitational field by virtue of the spatial dilution of the particle distribution. This shows clear parallels with the dynamics near a repulsive point charge, where all dynamics are (linearly) hyperbolic, and global in time stability was shown in \cite{PW2020,PWY2022}, even without radial symmetry assumptions. The ``method of asymptotic actions'' developed in those works is also the driving force behind the present article. In contrast, for bound orbits different dynamics are to be expected, and first important results \cite{CL2024} in this direction point to parallels with Landau damping on the torus: decay of the gravitational force field hereby arises from mixing effects, which rely on regularity (rather than spatial dilution). This has been established for the linearized dynamics, and has been shown to imply an extended time scale of existence in the radially symmetric case.

The setting of \eqref{eq:VP} near a point mass may be regarded as an idealization for the dynamics outside a spherically symmetric and highly localized equilibrium of \eqref{eq:VP}. Classical examples of the latter are polytropes, which are known to be orbitally stable, see for instance \cite{Guo1999StableSteadyStates,Lemou2012OrbitalStability,Rein2002StabilityGerneralPerturbations,Rein2007}. Furthermore, recent works have made substantial progress on the linear stability of these (and other) configurations, both in terms of spectral analysis \cite{HRSS2023,HRS2021} and quantitative decay rates for the gravitational field \cite{HRSS24,HS25}.

Finally, we remark that the question of stability of particular equilibria of \eqref{eq:VP} has also seen much progress in the setting of spatially homogeneous backgrounds, which are of natural interest for multi-species plasmas. Despite differing underlying mechanism, such stability results are typically referred to as Landau damping: in confined or screened cases these are comparatively well-understood \cite{BMM2018,FR2016,HNR2019,IPWW24,MV2011} and have parallels to inviscid damping in the $2d$ Euler equations (see e.g.\ \cite{IJ2019} for the stability of a point vortex), whereas only few results exist in the unconfined setting \cite{BMM2020,HNR2020,IPWW2023,IPWW2022} (see also \cite{T2024} for an interesting model problem). We further highlight the related works \cite{AW2021,HW2022} on the interaction of point charges with a homogeneous background.

\subsection{The setting of radial symmetry}\label{sec:intro_symm}
As is well known, the system \eqref{eq:VP} is radially symmetric in the sense that it is invariant under phase space rotations $(\bx,\bv)\mapsto (O\bx,O\bv)$, $O\in \mathcal{O}(3)$, and thus it is natural to look for solutions that share this feature. In the context of \eqref{eq:VPD}, if we choose coordinates such that $\cX(0)=\cV(0)=0$ and $f_0$ is radially symmetric, then so is the solution $f$ and the point mass is stationary at the origin (see also Remark \ref{rem:lweights}). In this setting, due to the reduced degrees of freedom  the particle distribution function can be expressed in terms of three scalar variables: the spatial distance to the center $r\in\R_+$, the velocity $v\in\R$ in the radial direction and the angular momentum $\ell\geq 0$, i.e.\footnote{Here we have chosen to work with non-negative densities $\mu^2$ in an $L^2$ framework rather than a general non-negative function $f$ in $L^1$ -- see also previous works \cite{FOPW2021,IPWW2020,PW2020,PWY2022} for more on this.}
\begin{align*}
	f(t,\bx,\bv) &= \mu^2(t,r(\bx),v(\bx,\bv),\ell(\bx,\bv)), \quad r(\bx) =|\bx|,\quad v(\bx,\bv):=\bv\cdot \dfrac{\bx}{|\bx|},\quad  \ell(\bx,\bv)=|\bx \wedge \bv|^2,
\end{align*}
while the Liouville measure in phase space transforms as
\begin{align*}
	f(t,\bx,\bv) \, d\bx d\bv &= 4\pi^2\mu^2(t,r,v,\ell) \, drdvd\ell.
\end{align*}
We have chosen here to incorporate the non-negativity of $f$ by writing $f$ in terms of $ \mu^2 $. Slightly abusing notation, we shall henceforth label the radial velocity with $v$. The equation for $ \mu(t,r,v,\ell):\R\times \R_+\times\R\times\R\to\R $ is then given by (see Appendix \ref{sec:AppendixDerivRadSymEq} for a full derivation)
\begin{align}\label{eq:Sec1:NonDynPhysVar}
	\left( \partial_t + v\partial_r+ \dfrac{\ell}{r^3}\partial_v - \dfrac{m}{2} \dfrac{1}{r^2}\partial_v \right)\mu- \lambda\partial_r\psi \, \partial_v\mu =0,
\end{align}
where the gravitational potential is recovered from $\mu$ as
\begin{align*}
	\psi(t,r) =-\int_{s=0}^\infty \int_{v\in\R}\int_{\ell=0}^\infty \,  \dfrac{\mu^2(t,s,v,\ell)}{\max(r,s)}ds dv d\ell.
\end{align*}
Thus, in particular the force is given by
\begin{align}\label{eq:Sec1:GravField}
	\F(t,r) := -\partial_r \psi(t,r) = -\dfrac{1}{r^2}\int_{s=0}^r \int_{v\in\R}\int_{\ell=0}^\infty \, \mu^2(t,s,v,\ell)ds dvd\ell.
\end{align}
The characteristic system of \eqref{eq:Sec1:NonDynPhysVar} is the Hamiltonian flow with respect to the Hamiltonian 
\begin{align*}
	\Ham(t,r,v,\ell) = \Ham_{lin}(r,v,\ell) + \lambda \psi(t,r),\quad \Ham_{lin}(r,v,\ell) =\dfrac{1}{2}v^2 + \dfrac{\ell}{2r^2} - \dfrac{m}{2r},
\end{align*}
so that \eqref{eq:Sec1:NonDynPhysVar} can be recast as
\begin{align}\label{eq:Sec1:NonDyn}
		\partial_t \mu + \{\mu,\Ham\} &= \partial_t \mu +  \{\mu,\Ham_{lin}\} + \lambda \{\mu,\psi\} =0,
\end{align}
with Poisson bracket given by
\begin{equation}\label{eq:Sec1:PoissonBracket}
  \{f,g\} = \partial_rf \partial_vg-\partial_vf \partial_rg.  
\end{equation}
As is well known, the angular momentum $ \ell $ is conserved in any central field. In particular, $ \ell $ is only a parameter in the dynamics and hence the Poisson bracket above is acting only on the variables $ (r,v) $.

\begin{rem}\label{rem:lweights}
    While weaker interpretations may be possible, in this article we will focus on dynamics for which the characteristic systems of \eqref{eq:VPD} are well-defined. In particular, for the point mass motion this requires that the gravitational field at the point mass vanishes, i.e.\ that $\partial_r\psi(t,0)=0$. This amounts to a vanishing condition of $\mu$ as $\ell \searrow 0$: we quantify this using moments in $\ell^{-1}$ on $\mu$, see Theorem \ref{thm:IntrodMainThmAAVar}.
\end{rem}

\begin{rem}\label{rem:asympt}
    The assumptions of Theorem \ref{thm:main_intro} translate in a straightforward fashion to the variables $(r,v,\ell)$. As sketched in \eqref{eq:asympt_main}, asymptotically $\mu$ converges along modified trajectories $(\mathscr{R},\mathscr{V})(t,r,v,\ell)$ to a final state $\mu_\infty$, where the Lagrangian map  $ t\mapsto (\mathscr{R},\mathscr{V})(t,r,v,\ell) $ can be expressed via an asymptotic gravitational field $ \F_\infty:\R_+\to (-\infty,0] $, depending only on the final state $\mu_\infty$. More precisely, as $ t\to \infty $ it is of the form
    \begin{equation}\label{eq:modscat_rvl}
	\begin{aligned}
		\mathscr{R}(r,v,\ell,t) &= t\sqrt{v^2+\dfrac{\ell}{r^2}-\dfrac{m}{r}} + \dfrac{m}{2}\left( v^2+\dfrac{\ell}{r^2}-\dfrac{m}{r}\right) ^{-1} \ln t -\lambda \F_\infty\left( \sqrt{v^2+\dfrac{\ell}{r^2}-\dfrac{m}{r}} \right) \ln t +\mathcal{O}(1),
		\\
		\mathscr{V}(r,v,\ell,t) &= \sqrt{v^2+\dfrac{\ell}{r^2}-\dfrac{m}{r}} + \dfrac{m}{2}\left( v^2+\dfrac{\ell}{r^2}-\dfrac{m}{r}\right) ^{-1} \dfrac{1}{t} + \mathcal{O}\left(\dfrac{1}{t^2} \right).
	\end{aligned}
    \end{equation}
	Note that the first term is the particle energy $ \Ham_{lin}(r,v,\ell) $ along the trajectory of the linearized dynamics. The second term is due to the linearized dynamics, whereas the third term is a nonlinear long-range correction due to the critical nature of Newton interactions in three dimensions. Note that the force $ \F_\infty $ is negative, i.e.\ it is directed towards the centre. This is due to the spherical symmetry and the nature of the Newtonian interaction (more precisely, at radius $ r>0 $ only mass at radius $ s\leq r $ is giving a contribution to the induced force field). As a consequence, the nonlinear correction in $ \mathscr{R} $ above gives a contribution which slows down particles due to the (attractive) interaction with the gas. Note that, formally setting $m=0$ in \eqref{eq:modscat_rvl}, we recover the asymptotic behavior near vacuum \cite{FOPW2021,IPWW2020}.
\end{rem}

\subsection{Overview and ideas of proof -- the method of asymptotic actions}
Our overall approach to establishing Theorem \ref{thm:main_intro} is guided by the Hamiltonian structure of the equations as exploited through the method of asymptotic actions developed in the prior works \cite{PW2020,PWY2022}. A first important step is the analysis of the linearized flow and the development of suitable asymptotic action-angle variables (see Section \ref{sec:LinearizedDyn} for full details). These are symplectic and allow to explicitly integrate the linear flow. Due to the symplectic structure and the favorable expression of the gravitational potential in terms of a phase space integral, this allows to reduce the nonlinear dynamics to a purely nonlinear equation. The asymptotic behavior can then be understood through a leading order asymptotic shear equation. 

In the context of attractive interactions, a key challenge is the lack of a global smooth choice of action-angle variables. As already discussed, in this article we focus on the case of open, hyperbolic trajectories, which shares parallels with the repulsive setting, and for which a smooth choice of action-angle variables is possible. However, since this is a condition on the possible linearized dynamics, one needs to make sure that the nonlinear dynamics do not leave the corresponding realm. We also remark that the setting of radial symmetry considered here goes beyond that of \cite{PW2020}, in that general angular momenta $\ell> 0$ are considered. While this alters the overall phase space structure, since $\ell$ is conserved along the nonlinear evolution it mostly plays the role of a parameter, and we can regard the problem as being set along one-parameter family of $1+1$-dimensional phase spaces in $(r,v)$.

\subsubsection{Linearized dynamics and action-angle variables}\label{sec:intro_lin+aa}
We observe that the linearized dynamics of \eqref{eq:Sec1:NonDynPhysVar} respectively \eqref{eq:Sec1:NonDyn} are given by the singular transport equation
\begin{align}\label{eq:Sec1:LinEq}
	\left( \partial_t + v\partial_r+ \dfrac{\ell}{r^3}\partial_v - \dfrac{m}{2} \dfrac{1}{r^2}\partial_v \right)\mu=0\quad\Leftrightarrow\quad \partial_t \mu +  \{\mu,\Ham_{lin}\} = 0.
\end{align}
The associated characteristic ODEs (see also \eqref{eq:Sec2:CharSys}) are simply those of the Kepler problem, i.e.\ of the classical, Newtonian two-body problem with attractive interactions. As is well-known, this is a completely integrable system, and both $\ell\geq 0$ and $\Ham_{lin}(r,v,\ell)$ are conserved along its trajectories. Moreover, $\Ham_{lin}$ can be used to distinguish different types of trajectories: when $\Ham_{lin}>0$, these are hyperbolas, while for $\Ham_{lin}<0$ one finds ellipses, and the separatrix $\Ham_{lin}=0$ consists of parabolas. We refer to \cite[Section 15]{LandauI1976} for more on the Kepler problem.

In the present article, our focus is on \emph{hyperbolic trajectories}, and we will thus restrict ourselves to studying solutions to \eqref{eq:Sec1:LinEq} with support on $\Ham_{lin}>0$. (By the aforementioned conservation laws, this support assumption is dynamically preserved if initially assumed.) The dynamics then naturally foliates along the conserved angular momentum $\ell\geq 0$, which plays the role of a parameter. Defining the action variable $a>0$ as
\begin{equation}
    a=a(r,v,\ell):=\sqrt{2\Ham_{lin}(r,v,\ell)}=\sqrt{v^2 + \dfrac{\ell}{r^2} - \dfrac{m}{r}}>0,
\end{equation}
in Proposition \ref{pro:ActionAngleVariables} we explicitly construct an associated angle variable $\theta=\theta(r,v,\ell)$ such that for each $\ell$ the change of variables $(r,v)\mapsto (\theta,a)$ has the following properties:
\begin{enumerate}[(i)]
	\item It is symplectic, i.e.\ preserves the Hamiltonian structure, in particular the Poisson bracket \eqref{eq:Sec1:PoissonBracket}.
	\item It satisfies the action-angle property, i.e.\ under the characteristics of \eqref{eq:Sec1:LinEq} we have $\dot{a}=0$ and $\dot{\theta}=a$,
	\item\label{it:asymptotic_actions} The actions $a$ parametrize the asymptotic velocities along trajectories as $ t\to \infty $.
\end{enumerate}
As discussed in \cite[Section 1.2]{PWY2022}, this last property is crucial for the study of the long-time behavior of the nonlinear system, as it implies that trajectories with different values of $a$ separate linearly in time.

These ``asymptotic'' action-angle variables allow us to desingularize and explicitly solve the underlying linear problem \eqref{eq:Sec1:LinEq}: denoting by  $(\theta,a)\mapsto (R(\theta,a,\ell),V(\theta,a,\ell))$ the (symplectic, for each $\ell$) inverse of the above change to action-angle variables, the equations for $\bar{\mu}(t,\theta,a,\ell):=\mu(t,R(\theta,a,\ell),V(\theta,a,\ell)) $ become
\begin{align*}
	\partial_t \bar{\mu} - a\partial_\theta \bar{\mu}=0,\quad \bar{\mu}\mid_{t=0}=\bar{\mu}_0,
\end{align*} 
the solution of which is given by
\begin{equation}
    \bar{\mu}(t,\theta,a,\ell)=\bar{\mu}_0(\theta+ta,a,\ell).
\end{equation}

\subsubsection{Nonlinear Dynamics}
As we just saw, one can stabilize the linearized dynamics by defining
\begin{align}\label{eq:Sec1:StabLinFlow}
	\begin{split}
		\gamma(t,\theta,a,\ell) :=&\bar{\mu}(t,\theta+ta,a,\ell) = \mu\left( t,\tilde{R}(\theta,a,\ell),\tilde{V}(\theta,a,\ell)\right) ,
		\\
		(\tilde{R}(\theta,a,\ell),\tilde{V}(\theta,a,\ell)) =& 	(R(\theta+at,a,\ell),V(\theta+at,a,\ell)).
	\end{split}
\end{align}
Since also $ (\theta,a)\mapsto (R(\theta+at,a,\ell),V(\theta+at,a,\ell)) $ is symplectic, the non-linear dynamics \eqref{eq:Sec1:NonDyn} is given by the purely nonlinear equation
\begin{align}\label{eq:Sec1:VPinActionAngle}
	\begin{split}
		\partial_t \gamma + \lambda \{\gamma,\tilde{\Psi}\} &=0,\qquad \tilde{\Psi}(t,\theta,a,\ell) := \Psi\left(t,\tilde{R}(\theta,a,\ell)\right),
		\\
		\Psi(t,r) =& - \int_{\R}\int_0^\infty\int_0^\infty  \,  \dfrac{\gamma(t,\theta,a,\ell)^2}{\max(R(\theta+at,a),r)} \, d\theta da d\ell, \quad 
	\end{split}
\end{align}
where the Poisson brackets only include the variables $ (\theta,a) $, while $ \ell $ appears as a parameter. Here we have used that the gravitational potential as well as the field are given as phase space integrals, that can naturally be expressed in terms of the unknown $\gamma$ in the action-angle variables $ (\theta,a,\ell)$ (since this change of variables is symplectic, the associated Jacobian is one).

While the restriction to hyperbolic trajectories $ \{a>0\} $ is a simple matter for the linearized dynamics, in order to justify the reformulation \eqref{eq:Sec1:VPinActionAngle} one needs to ascertain that also the nonlinear dynamics remain in the regime of hyperbolic trajectories, provided the initial distribution is. This is due to the perturbative framework and the Hamiltonian structure, and is made more precise in Lemma \ref{lem:NonlinCharSys}. There we show that for any $\delta\in (0,1)$ and any sufficiently small and localized initial distribution $ \gamma_0 $ supported inside the set
\begin{equation}\label{eq:def_Ddelta}
    \D(\delta)=\left\lbrace (\theta,a,\ell) \in \R\times\R_+^2 \, : \, a\jabr{\ell}^{\frac12}\geq \delta\right\rbrace,
\end{equation}
the corresponding solution $ \gamma $ remains supported inside $\D(\delta/2)$. In particular, this also includes an additional stabilizing effect of high angular momenta, for which the support of $\gamma_0$ can be closer to the separatrix $\{a=0\}$ of the linearized dynamics. A posteriori, this proves that the action-angle variables can be used for all times, and the equation \eqref{eq:Sec1:VPinActionAngle} in the variables $ (\theta,a,\ell) $ remains consistent with the one in the variables $ (r,v,\ell) $.

In order to obtain dynamical control over $\gamma$ and its uniqueness, it is essential to understand the decay and regularity properties of the gravitational field
\begin{equation}
   \F(t,r) = -\partial_r \Psi(t,r) = -\dfrac{1}{r^2}\int_{\R}\int_0^\infty\int_0^\infty  \,  \ind_{\{R(\theta+ta,a,\ell)\leq r\}}\gamma(t,\theta,a,\ell)^2  \, d\theta da d\ell, 
\end{equation}
and its derivative -- see Section \ref{sec:GravField}.  In particular, in order to be able to work with minimal requirements on the function space topology for the unknowns, a careful analysis of the interplay between the localization of $\gamma$ and the resulting decay, regularity and vanishing conditions on $\F$ is needed (Section \ref{subsec:EstGravFieldLInf}). We track the former through moment bounds on $\gamma$ in Lebesgue spaces, which in turn need to be propagated in time. As already commented on in Remark \ref{rem:lweights}, in these arguments the angular momentum plays a subtle role by translating vanishing conditions for $\gamma$ as $\ell\searrow 0$ into vanishing orders of $\F(t,r)$ as $r\searrow 0$. Inspired by the asymptotic behavior of solutions to the characteristic system \eqref{eq:Sec1:LinEq} and using similar tools, we also derive sharp bounds for the leading order ``effective'' gravitational field (Section \ref{subsec:EstEffGravField}), which will be instrumental in deriving the asymptotic behavior, as well as the gravitational force field in action-angle variables as it occurs in \eqref{eq:Sec1:VPinActionAngle}, i.e.\ 
\begin{equation}
    \partial_a\tilde{\Psi}(t,\tilde{R})=\F(t,\tilde{R})  \partial_a \tilde{R},   
\quad  - \partial_\theta\tilde{\Psi}(t,\tilde{R})=- \F(t,\tilde{R})\partial_\theta \tilde{R},
\end{equation}
see Section \ref{subsec:EstEffGravFieldAA}.

Using an Eulerian bootstrap, we combine the results of Section \ref{sec:GravField} with the propagation of moments in order to obtain well-defined global solutions. Hereby, building on the fact that they commute with the linear equation $\partial_t\gamma=0$, moments in $(\theta,a)$ can naturally be kept under control for \eqref{eq:Sec1:VPinActionAngle},\footnote{contrast this with moments in $(r,v)$ for \eqref{eq:Sec1:NonDynPhysVar}} while moments in $\ell$ simply commute with the equation due to the radial symmetry assumption. The key difficulty then lies in establishing that the resulting Lagrangian solutions are in fact unique (see Theorem \ref{thm:WellPosedLagrangianSol}): This relies on a detailed study of the nonlinear characteristics and Lipschitz bounds for $\F$ as established in Lemma \ref{lem:EstimatesDerivGravFieldLInf}. It is then a straightforward matter to propagate higher regularity -- we do this in Theorem \ref{thm:WellPosednessRegularSol} for initial data with one derivative in $L^2$.

\subsubsection{Asymptotic behavior}
A key towards understanding the long-time dynamics is the asymptotic behavior of the radial component of the action-angle variables along the flow $ (\theta,a)\mapsto (\theta+at,a) $ of the linear dynamics, namely that (thanks to the asymptotic action property \eqref{it:asymptotic_actions})
\begin{align}\label{eq:Sec1:RAsympt}
	\tilde{R}(\theta,a,\ell) = R(\theta+at,a,\ell) = at + o(t), \quad t\to \infty,
\end{align}
see Lemma \ref{lem:EstDynRadBulk} for the precise statement. Plugging this into the formula for $\tilde{\Psi}$ suggests that to leading order the dynamics in \eqref{eq:Sec1:VPinActionAngle} are driven by an effective potential
\begin{equation}
    \tilde{\Psi}_{\eff}(t,a):=-\frac{1}{t}\int_{\R}\int_0^\infty\int_0^\infty  \,  \dfrac{\gamma(t,\theta,\alpha,\ell)^2}{\max(\alpha,a)} \, d\theta d\alpha d\ell.
\end{equation}
In fact, one can show that for the solutions of Theorem \ref{thm:WellPosedLagrangianSol} (and thus also of Theorem \ref{thm:WellPosednessRegularSol})
\begin{equation}
   \tilde{\Psi}_{\eff}(t,a)\to-\frac{1}{t}\tilde{\Psi}_\infty(a),\quad t\to\infty, 
\end{equation}
so the characteristics for the asymptotic behavior are
\begin{equation}
    \dot\theta=-\lambda t^{-1}\F_\infty(a),\quad \dot{a}=0,\qquad \F_\infty(a)=\partial_a\tilde{\Psi}_\infty(a),
\end{equation}
which leads to the modified scattering behavior 
\begin{equation}\label{eq:mod_scat_formal}
   \lim_{t\to\infty}\gamma\left( t, \theta-\lambda\ln(1+t)\F_\infty(a),a,\ell \right)=\gamma_\infty(\theta,a). 
\end{equation}
These heuristic arguments can be made precise for the unique Lagrangian solutions of Theorem \ref{thm:WellPosedLagrangianSol} by working with the nonlinear characteristics, and thereby one obtains the convergence \eqref{eq:mod_scat_formal} in the sense of distributions (see Theorem \ref{thm:ModifiedScatteringLagrangianSol}). For strong solutions as in Theorem \ref{thm:WellPosednessRegularSol}, a simpler Eulerian approach working with the asymptotic shear equation
\begin{equation}
    \partial_t \gamma +\dfrac{\lambda}{t} \partial_{a}\tilde{\Psi}_\infty \, \partial_{\theta}\gamma =o\left(t^{-1}\right),
\end{equation}
shows that the convergence \eqref{eq:mod_scat_formal} holds strongly in $L^2$ (see Theorem \ref{thm:ModifiedScattering}).

\subsection{Main result}
A more precise version of our main result Theorem \ref{thm:main_intro} can then be summarized as follows (see Section \ref{sec:notation} for the relevant notation):
\begin{thm}\label{thm:IntrodMainThmAAVar}
	There exist a sufficiently small constant $ c>0 $ such that the following holds. Let $ \delta\in (0,1) $ and  $\gamma_0\in L^2\cap L^\infty$ with $ \supp \gamma_0\subset \D(\delta) $.
    \begin{enumerate}[(i)]
    \item\label{it:IntMainThm1} (Theorem \ref{thm:WellPosedLagrangianSol}) If $ \eps\leq c\delta^2 $ and
    \begin{align*}
    	\norm[L^\infty]{\jabr{\ell}^{32} \ell^{-1} (a+a^{-1})^{32} \gamma_0}+\norm[L^\infty]{\jabr{\ell}^{16} \ell^{-1} \jabr{\theta}^{32} \gamma_0} \leq \eps,
    \end{align*}
	then the equation \eqref{eq:Sec1:VPinActionAngle} has a unique, global in time Lagrangian solution $ \gamma\in C([0,\infty),L^2\cap L^\infty) $ with initial datum $ \gamma_0 $. Moreover this solution satisfies for all $t\geq 0$ that $ \supp \gamma(t)\subset \D(\delta/2) $ and
	\begin{align*}
		\norm[L^2]{(a+a^{-1})^{4}\gamma(t)} +\norm[L^\infty]{\jabr{\ell}^{32} \ell^{-1} (a+a^{-1})^{32} \gamma(t)}&\lesssim \eps,\quad\norm[L^\infty]{\jabr{\ell}^{16} \ell^{-1} \jabr{\theta}^{32} \gamma(t)}\lesssim \eps\ln^{32}\jabr{t}.
	\end{align*}
	
    \item\label{it:IntMainThm2} (Theorem \ref{thm:ModifiedScatteringLagrangianSol}) If in addition $ \eps\leq c\delta^3 $ then the Lagrangian solution in \eqref{it:IntMainThm1} disperses according to a modified scattering dynamics: there exists $ \gamma_\infty\in L^2\cap L^\infty $ such that
    \begin{align*}
    	\lim_{t\to \infty} \hat{\gamma}(t) = \gamma_\infty, \quad \hat{\gamma}(t,\theta,a,\ell):=\gamma\left( t, \theta-\lambda\ln(1+t)\F_\infty(a),a,\ell \right),
    \end{align*}
    in the sense of distributions, where
    \begin{align*}
    	\F_\infty(a) = -\dfrac{1}{a^2}\int_{\R}\int_{0}^\infty\int_{0}^\infty\ind_{\{\alpha\leq a\}} \, \gamma_\infty^2(\theta,\alpha, \ell) \, d\theta d\alpha d\ell.
    \end{align*}
    \item\label{it:IntMainThm3} (Theorems \ref{thm:WellPosednessRegularSol} and \ref{thm:ModifiedScattering}) If moreover $ \eps\leq c\delta^3 $ and $ \gamma_0 $ satisfies
    \begin{align*}
    	&\norm[L^2]{\omega^{-1}\chi\partial_\theta\gamma_0} + \norm[L^2]{\omega\chi\partial_a\gamma_0}+
    	\norm[L^2]{\omega^{-1}\jabr{\theta}^{2}\partial_\theta\gamma_0} + \norm[L^2]{\omega\jabr{\theta}^{2}\partial_a\gamma_0}\leq \eps,
        \\
        &\chi=\jabr{\ell}^4\jabr{\ell^{-1}}^2\jabr{a}^2, \quad \omega=\left(\frac{a}{a+\jabr{l}^{\frac12}}\right)^{\frac12},
    \end{align*}
 	then the solution in \eqref{it:IntMainThm2} is strong $ \gamma\in C^1((0,\infty);L^2)\cap C([0,\infty);\SobH^1) $ and satisfies for all $t\geq0$
	\begin{align*}
		\norm[L^2]{\omega^{-1}\chi\partial_\theta\gamma(t)}&\lesssim \varepsilon, \quad \norm[L^2]{\omega\chi\partial_a\gamma(t)}\lesssim \varepsilon \ln^{33}\jabr{t}
		\\
		\norm[L^2]{\omega^{-1}\jabr{\theta}^{2}\partial_\theta\gamma(t)} &\lesssim \eps\ln^2\jabr{t}, \quad \norm[L^2]{\omega\jabr{\theta}^{2}\partial_a\gamma(t)}\lesssim \eps\ln^{35}\jabr{t}.
	\end{align*}
    Furthermore, the convergence in \eqref{it:IntMainThm2} is strong: we have
 	\begin{align*}
 		\norm[L^2]{\hat{\gamma}(t)-\gamma_\infty} \lesssim \dfrac{\varepsilon^3}{\delta^4}\dfrac{\ln^{35}\jabr{t}}{(1+t)^{1/2}}.
 	\end{align*}
    \end{enumerate}
\end{thm}

\paragraph{\bf Organization of the paper.} In the following subsection we recall some useful definitions and notations used in our study. In Section \ref{sec:LinearizedDyn} we study the linearized dynamics induced by the Kepler problem and define appropriate action-angle variables. Furthermore, we establish several estimates of relevant quantities in  action-angle variables. In Section \ref{sec:GravField} we prove bounds on the gravitational field in both physical and action-angle variables. In Section \ref{sec:nonlinear} we study the nonlinear problem. We first prove well-posedness of Lagrangian and strong solutions in Subsection \ref{subsec:WellPosedness}. Then, we show that Lagrangian respectively strong solutions admit a modified scattering dynamics in Subsection \ref{subsec:LongtimeBehavior}. Finally, in Appendix \ref{sec:AppendixDerivRadSymEq} we give a derivation of the Vlasov-Poisson equation \eqref{eq:Sec1:NonDynPhysVar} for radially symmetric data.

\subsection{Notations}\label{sec:notation}
We use the notation $ \R_+=(0,\infty) $, excluding zero. Furthermore, we define the Japanese brackets $ \jabr{x}= (2+|x|^2)^{1/2} $ for any $ x\in \R $. In particular, $ \ln\jabr{x}\geq \ln2>0 $.

For our estimates we will use the abbreviation: for two real-valued functions $ f,\, g $ (depending on some parameters) we write $ f\lesssim g $ if there is a positive constant $ C>0 $ such that $ f\leq Cg $. The constant is independent of the parameters. For our purposes this in particular includes time $ t $, the angle variable $ \theta $, the action variable $ a $ and the angular momentum $ \ell $. Similarly, we define $ f\gtrsim g $. Furthermore, we write $ f\approx g $ if both $ f\lesssim g $ and $ f\gtrsim g $ holds. Observe that with this notation $ \jabr{x}\approx 1+|x| $ and $ \ln\jabr{x}\approx \ln(2+|x|) $.

Moreover, we denote by $ L^2(\Omega;\R) $ the standard $ L^2 $-space of real-valued functions defined on some set $ \Omega\subset \R^d $, $ d\in \N $. We will drop both the domain $ \Omega  $ and codomain $ \R $ if there is no ambiguity. The corresponding scalar product is denoted by $ \skp{f}{g}_{L^2} $ for $ f,\, g\in L^2 $. 

Furthermore, for real-valued functions $ f=f(\theta,a,\ell) $, where $ (\theta,a,\ell)\in \R\times\R_+^2 $ are the action-angle variables, see Section \ref{sec:LinearizedDyn}, we define the following Sobolev space $ \SobH^1 $. We say $ f\in \SobH^1 $ if $ f\in L^2(\R\times\R_+^2) $ and the weak derivatives $ \partial_{\theta}f, \, \partial_{a}f \in L^2(\R\times\R_+^2) $ exists. A norm is defined via
\begin{align*}
	\norm[\SobH^1]{f}^2 = \norm[L^2]{f}^2+\norm[L^2]{\partial_{\theta}f}^2+\norm[L^2]{\partial_{a}f}^2.
\end{align*}
Moreover, we denote by $ \mathcal{D}'(\Omega) $ the space of distributions on some domain $ \Omega $. 

Furthermore, we write $ C([0,\infty);L^2) $ and $ C([0,\infty);\SobH^1) $ for those functions spaces containing functions $ t\mapsto f(t,\cdot) $ continuous (with respect to the norm) into $ L^2 $ and $ \SobH^1 $, respectively, for all $ t\geq0 $. Similarly, we define $ C([0,\infty);\mathcal{D}') $ only requiring $ t\mapsto f(t,\cdot) $ to be continuous into $ \mathcal{D}' $ with respect to the weak topology (in the sense of distribution). In addition, we define $ C^1((0,\infty);L^2) $ when $ t\mapsto f(t,\cdot)\in L^2 $ is Fr\'echet-differentiable at each point $ t\in (0,\infty) $ and its Fr\'echet derivative $ \partial_tf\in C((0,\infty);L^2) $.

Concerning Poisson brackets we will use the notation
\begin{align*}
	\{f,g\}=\partial_rf\partial_vg-\partial_vf\partial_rg,
\end{align*}
when $ f,\, g $ are depending on the (physical) variables $ (r,v)\in \R_+\times\R $, possible also depending on the momentum variable $ \ell\in \R_+ $. On the other hand when $ f,\, g $ are functions of the action-angle variables $ (\theta,a)\in \R\times \R_+ $, possible also depending on the momentum variable $ \ell\in \R_+ $, we write
\begin{align*}
	\{f,g\}=\partial_\theta f\partial_ag-\partial_af\partial_\theta g.
\end{align*}


\section{Linearized dynamics}\label{sec:LinearizedDyn}
We study the linearized dynamics, which is described by the equation
\begin{align*}
	\partial_t \mu + \{\mu,\Ham_{lin}\} = \left( \partial_t + v\partial_r+ \dfrac{\ell}{r^3}\partial_v - \dfrac{m}{2} \dfrac{1}{r^2}\partial_v \right)\mu=0.
\end{align*}

\subsection{Characteristic system and action-angle variables}
The linearized equation can be solved by the method of characteristics
\begin{align}\label{eq:Sec2:CharSys}
	\dot{r} = v, \quad \dot{v} = \dfrac{\ell}{r^3}-\dfrac{m}{2}\dfrac{1}{r^2}, \quad \dot{\ell}=0.
\end{align}
These ODEs appear as the Hamiltonian dynamics to the Hamiltonian
\begin{align*}
	\Ham_{lin}(r,v,\ell) = \dfrac{1}{2}\left( v^2+\dfrac{\ell}{r^2}-\dfrac{m}{r} \right) 
\end{align*}
We want to define the action $ a $ such that the asymptotic velocity $ t\to \infty $ is given by $ a $. Since we are only interested in unbounded trajectories, i.e.\ as $ t\to \infty $ we have $ r(t)\to \infty $, we see from the conservation of energy that $ \Ham_{lin}(r,v,\ell)\geq0 $. The limit case $ \Ham_{lin}(r,v,\ell)=0 $ corresponds to the separatrix between bounded and unbounded trajectories. We study in detail only hyperbolic trajectories, i.e.\ those for which $ \Ham_{lin}(r,v,\ell)>0 $.

The action $ a>0 $ can now be defined via
\begin{align*}
	a^2 = v^2-\dfrac{m}{r}+\dfrac{\ell}{r^2}
\end{align*} 
such that for $ t\to \infty $ it yields the asymptotic velocity. We then obtain
\begin{align}\label{eq:Sec2:velocity}
	v = \pm a\sqrt{1+\dfrac{m}{a^2r}-\dfrac{\ell}{a^2r^2}} = \pm \dfrac{a}{r}\sqrt{r^2+\dfrac{m}{a^2}r-\dfrac{\ell}{a^2}}.
\end{align}
We now summarize some properties of the trajectory $ t\to (r(t),v(t)) $ for $ a>0 $. See also Figure \ref{fig:RadTrajectory} and Figure \ref{fig:VelTrajectory} for plots of the radial function $ r(t) $ and the velocity $ v(t) $.
\begin{lem}[Hyperbolic trajectories]\label{lem:HypTraj}
	Consider a solution $ t\to (r(t),v(t)) $ to \eqref{eq:Sec2:CharSys} with $ a,\, \ell >0 $. 
	\begin{enumerate}[(i)]
		\item\label{it:lem:HypTraj1} There is exactly one point of closest approach, i.e.\ $ v(t)=0 $ at one time $ t=t_0 $ with
		\begin{align*}
			r(t_0) = r_0(a,\ell) := \dfrac{m}{2a^2}\left( \sqrt{1+\dfrac{4a^2\ell}{m^2}}-1\right) = \dfrac{2\ell}{m} \dfrac{1}{1+\sqrt{1+\dfrac{4a^2\ell}{m^2}}}.
		\end{align*}
		The trajectory is symmetric with respect to $ t=t_0 $, that is
		\begin{align*}
			r(t_0+t) = r(t_0-t), \quad v(t_0+t) = -v(t-t_0).
		\end{align*}
		Furthermore, the function $ a\mapsto r_0(a,\ell) : \R_+\to (0,2\ell/m) $ is decreasing.
		
		\item\label{it:lem:HypTraj2} For $ t\to \pm\infty $ we have $ v(t) \to \pm a $.
		
		\item\label{it:lem:HypTraj3} The function $ t\to v(t_0+t) $ is increasing-decreasing attaining its maximum with value
		\begin{align*}
			v_{peak}(a,\ell) = \sqrt{a^2+\dfrac{m^2}{4\ell}}
		\end{align*}
		while $ r $ attains the value $ r_{peak}(\ell) = 2\ell/m $.
		
		\item\label{it:lem:HypTraj4} A parametric form of the solution is given by
		\begin{align*}
			\begin{cases}
				r = p(\cosh\xi - \kappa)
				\\
				a(t-t_0) = p(\sinh\xi -\kappa \xi)
			\end{cases}, \quad \kappa(a,\ell):= \left(1+\dfrac{4a^2\ell}{m^2} \right)^{-1/2}, \quad  p(a,\ell) := \dfrac{m}{2a^2\kappa(a,\ell)}. 
		\end{align*}
		Here, the parameter is given by $ \xi\in \R $.
		
		\item\label{it:lem:HypTraj5} Finally, an explicit formula is given by
		\begin{align*}
			r(t+t_0) = p H_\kappa\left(\dfrac{a|t-t_0| }{p}\right)-p\kappa, \quad r_0(a,\ell)=p(a,\ell)(1-\kappa(a,\ell)).
		\end{align*}
		where 
		\begin{align*}
			&G_\kappa : [1,\infty) \to [0,\infty) : x\mapsto \sqrt{x^2-1} - \kappa \arcosh(x)= \sqrt{x^2-1} -\kappa\ln\left(x+\sqrt{x^2-1}\right),
			\\
			&H_\kappa: [0,\infty) \to [1,\infty) : H_\kappa(x) = G_\kappa^{-1}(x).
		\end{align*}
	\end{enumerate}
\end{lem}

\begin{rem}
	Note that the trajectories as written in \eqref{it:lem:HypTraj4} are hyperbolas, the geometry of which is determined by the following dimensionless parameters: $ p(a,\ell) $ is the \textit{semi-major axis}, $ e(a,\ell)=1/\kappa(a,\ell) $ is the \textit{eccentricity} and $ p(1/\kappa^2-1)=2\ell/m\kappa $ is the \textit{parameter}. We refer to \cite[Section 15]{LandauI1976}.
\end{rem}

\begin{rem}[Rescaling the angular momentum]\label{rem:Scaling}
	Let us note that the differential equations \eqref{eq:Sec2:CharSys} can be reduced to the case $ \ell=1 $ via the following scaling argument: Denoting by $ (r^\ell,v^\ell) $ the solution to \eqref{eq:Sec2:CharSys} with angular momentum $ \ell>0 $, we observe that
    \begin{align*}
		r^\ell(t) = \ell \, r^1\left( \dfrac{t}{\ell^{3/2}} \right), \quad v^\ell(t) = \dfrac{1}{\sqrt{\ell}} \, v^1\left( \dfrac{t}{\ell^{3/2}} \right).
	\end{align*} 
	This implies for the energy resp.\ actions that
	\begin{align*}
		a^2\ell = \ell\, \Ham_{lin}\left( r^\ell(t),v^\ell(t),\ell \right) = \Ham_{lin}\left( r^1\left( \dfrac{t}{\ell^{3/2}}\right),v^1\left( \dfrac{t}{\ell^{3/2}}\right),1\right). 
	\end{align*}
	The natural scaling for the actions $ a $ is $ 1/\sqrt{\ell} $: observe that
	\begin{align*}
		r_0\left( a, \ell\right) &= \ell \, r_0\left( \sqrt{\ell}a,1\right), \quad r_{peak}\left( a, \ell\right) = \ell \, r_{peak}\left( \sqrt{\ell}a,1\right), \quad v_{peak}\left(a, \ell \right) = \dfrac{1}{\sqrt{\ell}}v_{peak}\left(\sqrt{\ell}a,1\right),
		\\
		\kappa(a,\ell) &= \kappa\left(\sqrt{\ell}a,\ell\right), \quad p(a,\ell) = \ell\, p\left(\sqrt{\ell}a,\ell\right).
	\end{align*}
\end{rem}

\begin{figure}[!ht]
	\centering
	\begin{minipage}{.5\textwidth}
		\centering
		\includegraphics[width=\linewidth]{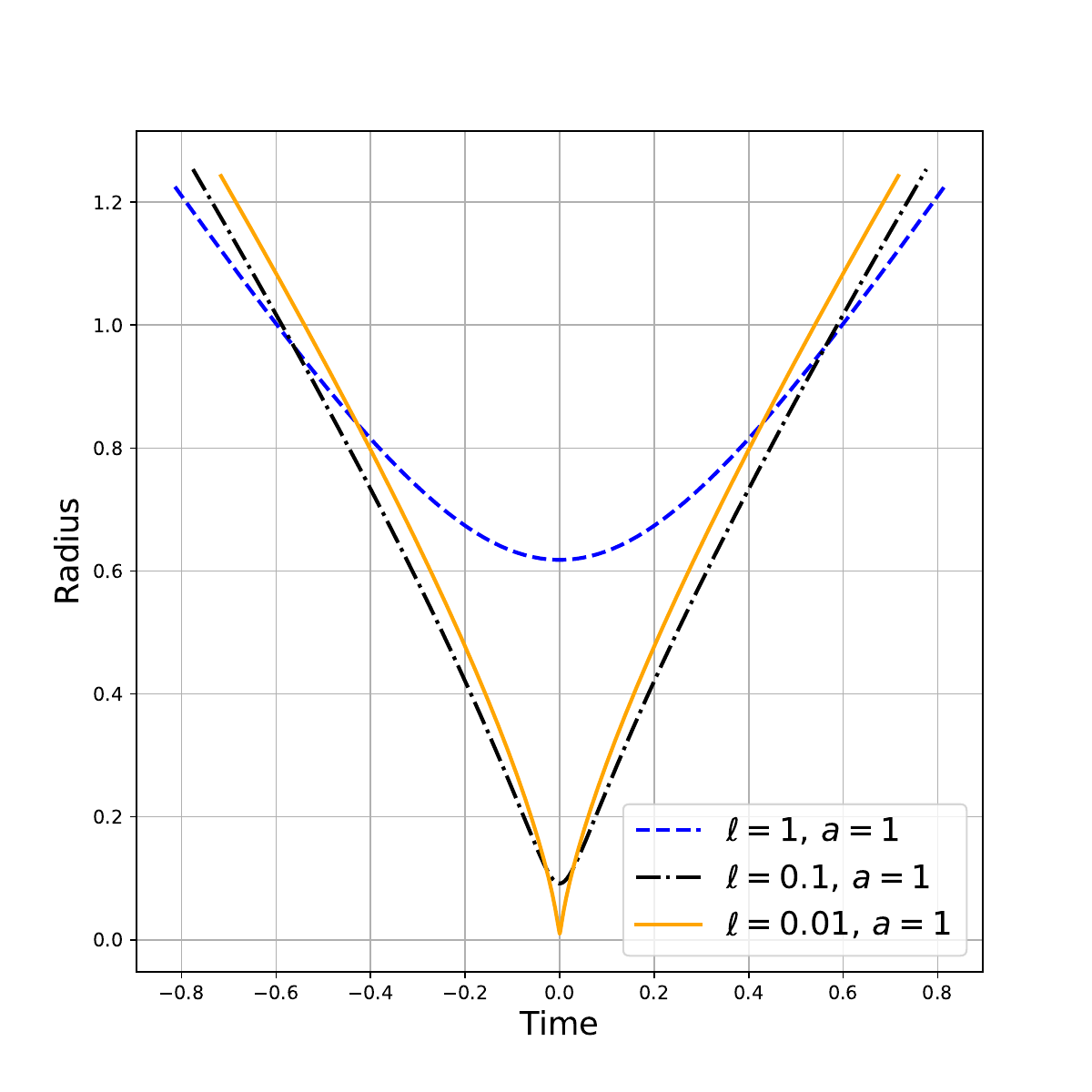}
	\end{minipage}%
	\begin{minipage}{.5\textwidth}
		\centering
		\includegraphics[width=\linewidth]{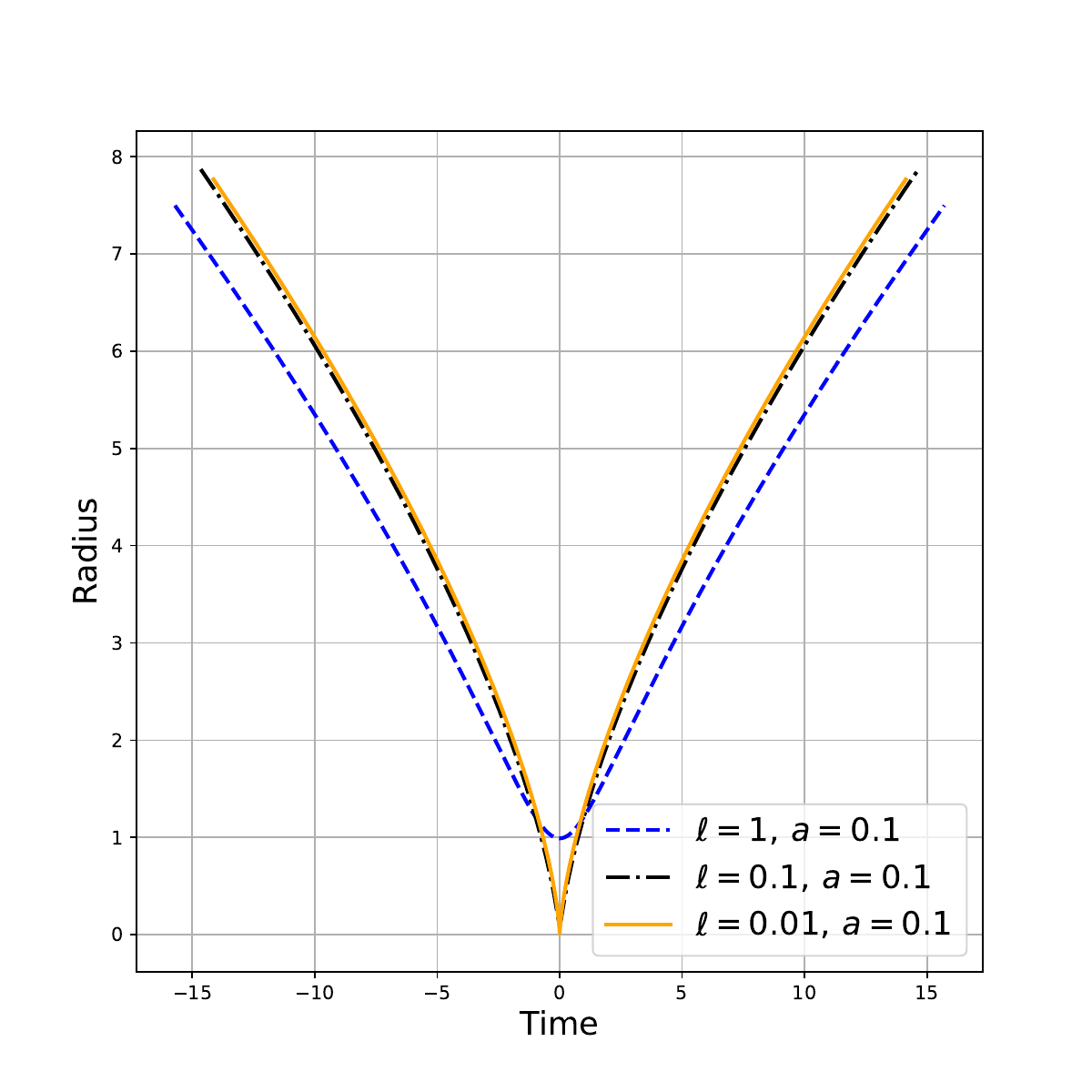}
	\end{minipage}
	\caption{Plots of the radial function $ t\mapsto r(t) $ of the solution to \eqref{eq:Sec2:CharSys} for $ t_0=0 $ and different values of actions $ a $ and angular momentum $ \ell $. Observe that for small values of $ \ell $ the function behaves more singular at time $ t=0 $. Compare this with Remark \ref{rem:ZeroEll}.}
	\label{fig:RadTrajectory}
\end{figure}

\begin{figure}[!ht]
	\centering
	\begin{minipage}{.5\textwidth}
		\centering
		\includegraphics[width=\linewidth]{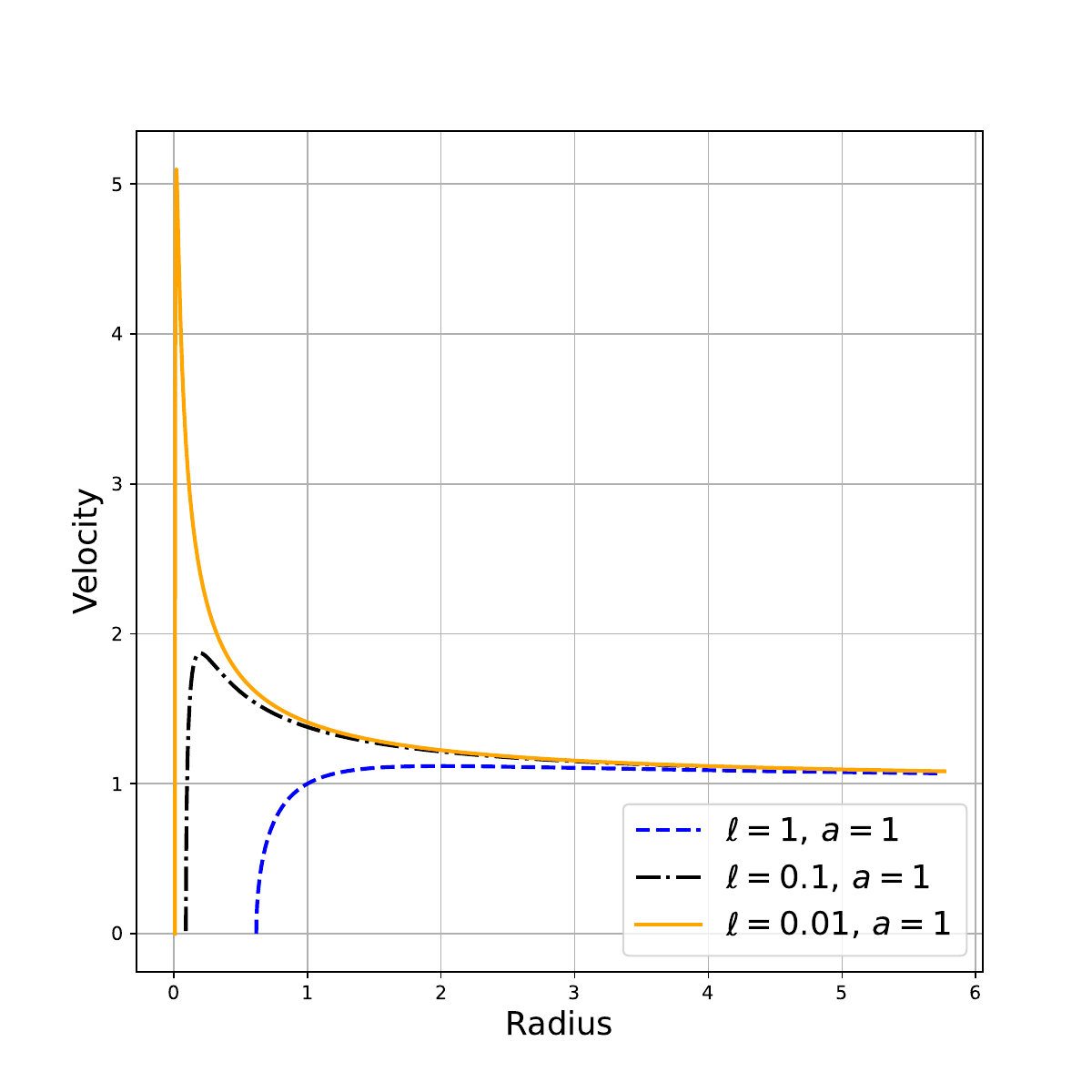}
	\end{minipage}%
	\begin{minipage}{.5\textwidth}
		\centering
		\includegraphics[width=\linewidth]{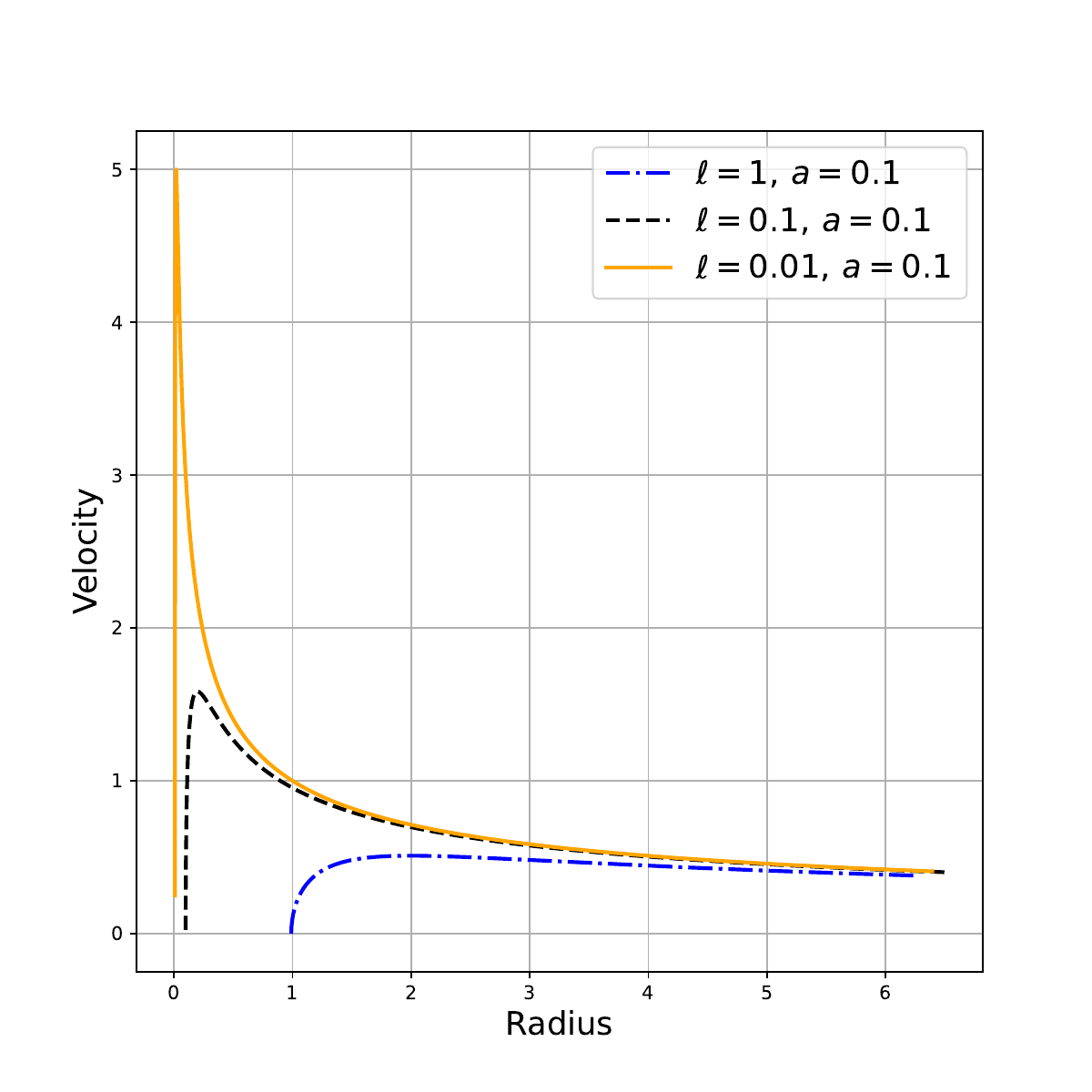}
	\end{minipage}
	\caption{Plots of the velocity $ t\mapsto v(t) $ of the solution to \eqref{eq:Sec2:CharSys} for $ t_0=0 $ and different values of actions $ a $ and angular momentum $ \ell $. Note that the velocity approaches the asymptotic velocity $ a $ when $ t\to \infty $. Furthermore, observe that for small values of $ \ell $ the velocity becomes singular close to $ t=0 $. Compare this with Remark \ref{rem:ZeroEll}.}
	\label{fig:VelTrajectory}
\end{figure}

\begin{proof}[Proof of Lemma \ref{lem:HypTraj}]
	First of all, from \eqref{eq:Sec2:velocity} we see that 
	\begin{align*}
		v = \pm \dfrac{a}{r}\sqrt{(r-r_0(a,\ell))(r+r_-(a,\ell))},
	\end{align*}
	where the roots of the polynomial $ r^2+rm/a^2-\ell/a^2 $ are given by $ r_0(a,\ell)>0 $ and
	\begin{align*}
		r_-(a,\ell) := \dfrac{m}{2a^2}\left( \sqrt{1+\dfrac{4a^2\ell}{m^2}}+1\right).
	\end{align*}
	Hence, necessarily we have $ r(t)\geq r_0(a,\ell) $. 
	
	Consequently, if $ v(t)<0 $ at some point $ t\in \R $, the function is strictly decreasing. One can see that it has to reach the value $ r_0(a,\ell) $ at some finite time. At this point of closest approach, say for $ t=t_0 $ the acceleration is given by
	\begin{align*}
		\dot{v}(t_0) = \left( \dfrac{\ell}{r^3}-\dfrac{m}{2}\dfrac{1}{r^2} \right) \mid_{r=r_0} = \dfrac{1}{r_0^3}\left( \ell-\dfrac{mr_0}{2} \right) >0,
	\end{align*}
	since $ r_0(a,\ell)<2\ell/m $. Hence, $ v(t)>0 $ for all $ t>t_0 $. One can argue similarly when $ v(t)<0 $ at some point $ t\in \R $ and looking at the dynamics backwards in time. This shows that there is exactly one point of closest approach at some time $ t=t_0 $. The fact that the trajectories are symmetric follows from the symmetry of the equations and the uniqueness of the solution. The fact that $ a\mapsto r_0(a,\ell) $ is decreasing follows from the formula. This proves \eqref{it:lem:HypTraj1}.
	
	Point \eqref{it:lem:HypTraj1} shows that for $ t> t_0 $ the function $ r(t) $ is strictly increasing as $ v(t)>0 $. Thus, $ r(t)\to \infty $ and by \eqref{eq:Sec2:velocity} we have $ v(t) \to a $. Together with the symmetry of the trajectory this shows \eqref{it:lem:HypTraj2}.
	
	For \eqref{it:lem:HypTraj3} we observe that
	\begin{align*}
		r\mapsto \dfrac{a}{r}\sqrt{r^2+\dfrac{m}{a^2}r-\dfrac{\ell}{a^2}} = \dfrac{a}{r} \sqrt{(r-r_0(a,\ell))(r+r_-(a,\ell))}
	\end{align*}
	attains its maximum for $ r = 2r_0r_-/(r_--r_0) = 2\ell/m= r_{peak}(\ell) $ with value $ v_{peak}(a,\ell) $.
	
	For the parametric form in \eqref{it:lem:HypTraj4} we compute the integral for $ t\geq t_0 $, following from \eqref{eq:Sec2:velocity},
	\begin{align*}
		a(t-t_0) = \int_{r_0}^{r(t)} \dfrac{\rho d\rho}{\sqrt{\rho^2+\dfrac{m}{a^2}\rho-\dfrac{\ell}{a^2}}} = \int_{r_0}^{r(t)} \dfrac{\rho d\rho}{\sqrt{\left( \rho+\dfrac{m}{2a^2} \right)^2-p^2 }}.
	\end{align*}
	We now use the substitution $ \rho = p\cosh\xi-\frac{m}{2a^2}= p(\cosh\xi-\kappa) $ yielding
	\begin{align*}
		a(t-t_0) = p \int_0^{\xi(t)}\dfrac{(\cosh\xi-\kappa)\sinh\xi\, d\xi}{\sqrt{\cosh^2\xi-1}} = p(\sinh\xi(t)-\kappa\xi(t)).
	\end{align*}
	Note that we used $ r_0=p(1-\kappa) $ yielding $ \xi_0=\xi(t_0)=0 $. This yields the parametric from.
	
	Finally, for the explicit form we use
	\begin{align*}
		\xi = \arcosh\left( \dfrac{r}{p}+\kappa \right)
	\end{align*}
	yielding
	\begin{align*}
		a|t-t_0| = pG_\kappa\left( \dfrac{r}{p}+\kappa \right).
	\end{align*}
    This concludes the proof.
\end{proof}

\begin{rem}[Trajectories with $ \ell=0 $]\label{rem:ZeroEll}
	In case $ \ell=0 $ the point of closest approach becomes $ r_0(a,0)=0 $. Furthermore, the velocity becomes infinite there, since $ r_{peak}(a,0)=0 $ and by Remark \ref{rem:Scaling} we have
	\begin{align*}
		v_{peak} (a,\ell) = \dfrac{1}{\sqrt{\ell}} v_{peak}\left(\sqrt{\ell}a,1\right) \to \infty
	\end{align*}
	as $ \ell\to 0 $, see also Figure \ref{fig:TrajectoryL0}. This singular behavior is reflected in action-angle variables in terms of singular derivatives as $ \ell \to 0 $. More precisely, $ H_\kappa $ is not differentiable at zero for $ \kappa(a,0)=1 $. 
\end{rem}

\begin{figure}[!ht]
	\centering
	\begin{minipage}{.5\textwidth}
		\centering
		\includegraphics[width=\linewidth]{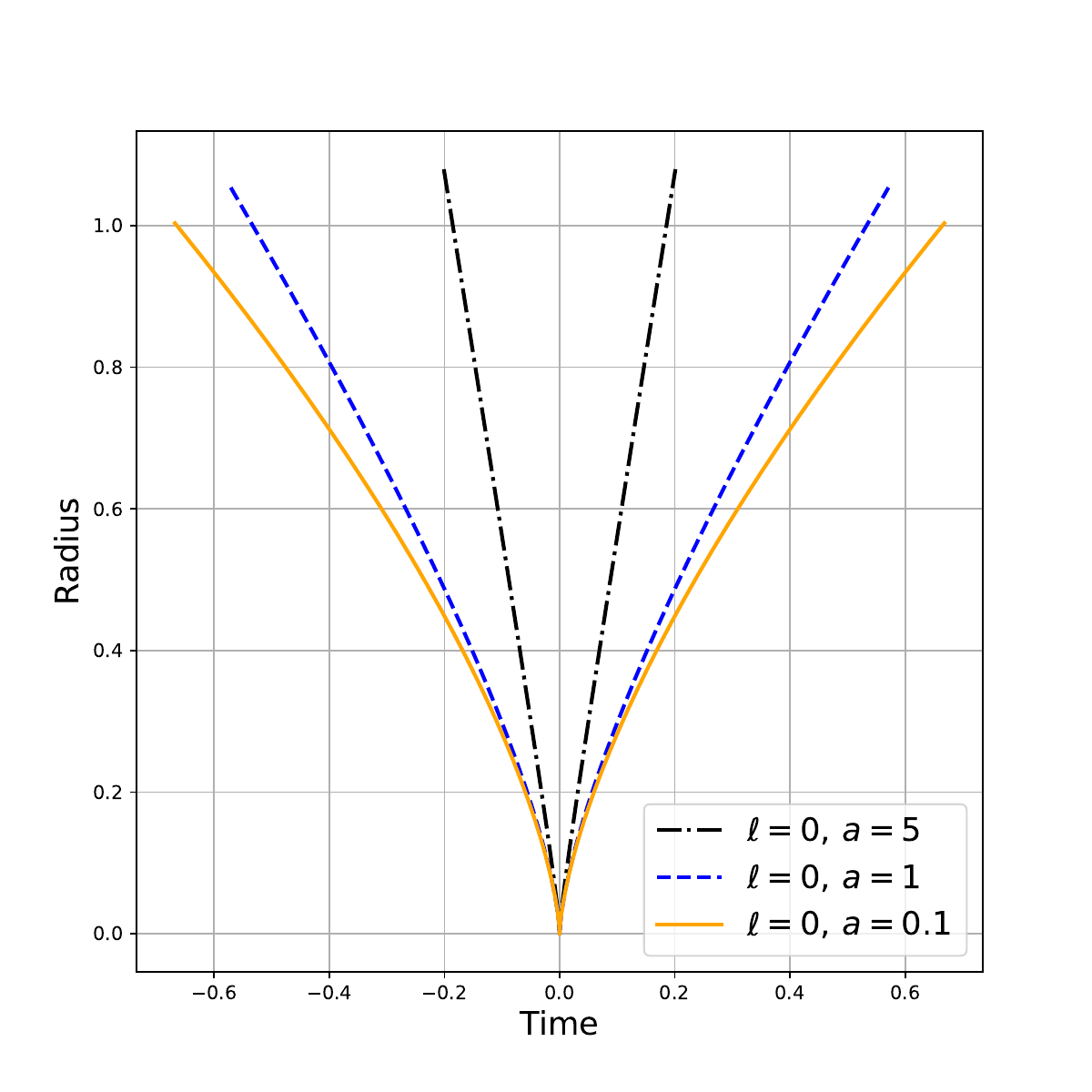}
	\end{minipage}%
	\begin{minipage}{.5\textwidth}
		\centering
		\includegraphics[width=\linewidth]{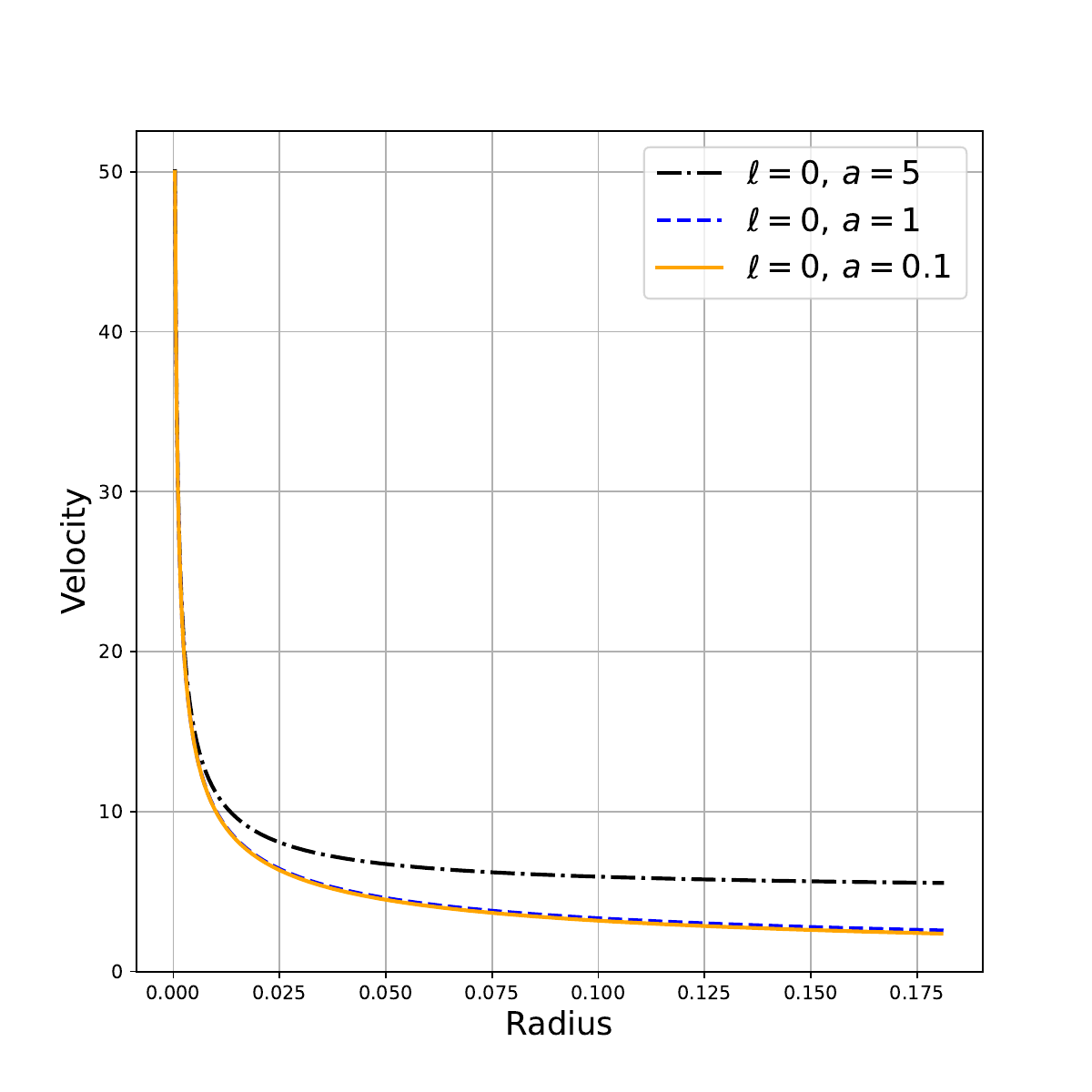}
	\end{minipage}
	\caption{Plot of the radial function $ t\mapsto r(t) $ (left) and velocity $ t\mapsto v(t) $ (right) of the solution to \eqref{eq:Sec2:CharSys} when $ t_0 = \ell=0 $ and different values of actions $ a $. At time $ t=0 $ the radial function has a singular behavior. Accordingly, the velocity goes to infinity as predicted in Remark \ref{rem:ZeroEll}.}
	\label{fig:TrajectoryL0}
\end{figure}

\begin{rem}[Parabolic trajectories]
	In the case that $ a=0 $ and hence $ \kappa(a,\ell)=1 $, $ p(0,\ell)= \infty $ the trajectories become parabolas, in comparison with the hyperbolic form (see point \eqref{it:lem:HypTraj4} in Lemma \ref{lem:HypTraj}). Then we get the parametric form (by replacing $ \xi $ by $ a\xi $ and letting $ a\to0 $ in the parametric form for $ a>0 $)
	\begin{align*}
		\begin{cases}
			r = \dfrac{m}{4}\xi^2 + \dfrac{\ell}{m}
			\\
			t-t_0 = \dfrac{m}{12}\xi^3+\dfrac{\ell}{m}\xi
		\end{cases}, \quad |t-t_0| = \dfrac{2}{\sqrt{m}}\left[ \dfrac{1}{3}\left( r(t)-\dfrac{\ell}{m} \right)^{3/2}+\dfrac{\ell}{m}\left( r(t)-\dfrac{\ell}{m} \right)^{1/2}  \right] .
	\end{align*}
\end{rem}

\paragraph{\bf Action-angle variables.} Using the explicit formulas we can define the action-angle variables as follows.

\begin{pro}[Action-angle variables]\label{pro:ActionAngleVariables}
	For any $ \ell>0 $ define 
	\begin{align*}
		\mathscr{D}_\ell := \left\lbrace (r,v)\in(0,\infty)\times\R \, :\, v^2-\dfrac{m}{r}+\dfrac{\ell}{r^2}>0 \right\rbrace.
	\end{align*}
	Then, for any $ \ell>0 $ the functions
	\begin{align}\label{eq:Sec2:PropAAVar1}
		\begin{split}
			(R(\cdot,\cdot,\ell),V(\cdot,\cdot,\ell)) &: \R \times (0,\infty)\to \mathscr{D}_\ell\subset (0,\infty)\times\R :
			\\
			R(\theta,a,\ell) &= p(a,\ell) H_\kappa\left(\dfrac{|\theta| }{p(a,\ell) }\right)-p(a,\ell) \kappa(a,\ell),
			\\
			V(\theta,a,\ell) &= \sgn(\theta) \, a \, \sqrt{1+\dfrac{m}{a^2R(\theta,a,\ell)}-\dfrac{\ell}{a^2R(\theta,a,\ell)^2}}
			\\
			&= \sgn(\theta) \, a \dfrac{p(a,\ell)}{R(\theta,a,\ell)} \, \sqrt{\left( \dfrac{R(\theta,a,\ell)}{p(a,\ell)} +\kappa(a,\ell) \right)^2-1 }
		\end{split}
	\end{align}
	define a canonical diffeomorphism. The inverse is given by
	\begin{align}\label{eq:Sec2:PropAAVar2}
		\begin{split}
			(\Theta(\cdot,\cdot,\ell),\Act(\cdot,\cdot,\ell)) &: \mathscr{D}_\ell \to \R \times (0,\infty) :
			\\
			\Act(r,v,\ell) &= \sqrt{v^2-\dfrac{m}{r}+\dfrac{\ell}{r^2}},
			\\
			\Theta(r,v,\ell) &= \sgn(v) \, p(a,\ell) \, G_\kappa\left( \dfrac{r}{p(a,\ell)}+\kappa(a,\ell)  \right) \mid_{a=\Act(r,v,\ell)}
		\end{split}
	\end{align}
	Finally, the solutions to \eqref{eq:Sec2:CharSys} with positive energy have the form
	\begin{align}\label{eq:Sec2:SolActAng}
		t\mapsto  \left( R(\theta+ta,a,\ell) , V(\theta+ta,a,\ell)\right) 
	\end{align}
	with parameters $ (\theta,a)\in \R\times (0,\infty) $ and $ \ell>0 $.
\end{pro}
\begin{rem}\label{rem:Scaling2}
	By the scaling properties discussed in Remark \ref{rem:Scaling}, the functions $R,V$ in \eqref{eq:Sec2:PropAAVar1} satisfy
	\begin{align*}
		R(\theta,a,\ell) = \ell \, R\left( \dfrac{\theta}{\ell}, \sqrt{\ell} a,1 \right), \quad V(\theta,a,\ell) = \dfrac{1}{\sqrt{\ell}} V\left( \dfrac{\theta}{\ell}, \sqrt{\ell} a,1 \right).
	\end{align*}
	In particular, we note that $ R/p $ is invariant under this scaling:
	\begin{align*}
		\dfrac{R(\theta,a,\ell)}{p(a,\ell)} = \dfrac{R\left( \theta/\ell,\sqrt{\ell}a,1\right)}{p\left(\sqrt{\ell} a,1\right)}.
	\end{align*}
    For the bounds on $R,V$ in the below sections, instead of frequent rescalings we have chosen to rely on this quantity whenever possible -- compare e.g.\ Lemmas \ref{lem:EstimatesDerivRadFunct} and \ref{lem:EstimatesSecondDerivRadFunct}.
\end{rem}

\begin{proof}[Proof of Proposition \ref{pro:ActionAngleVariables}]
	By the formulas in Lemma \ref{lem:HypTraj} the defined functions are inverse to each other. Observe that by the definition of $ \kappa, p $ in Lemma \ref{lem:HypTraj} \eqref{it:lem:HypTraj4} we have
	\begin{align*}
		a\sqrt{1+\dfrac{m}{a^2r}-\dfrac{\ell}{a^2r^2}} = a\dfrac{p}{r}\sqrt{\left( \dfrac{r}{p}+\kappa \right)^2-1 }
	\end{align*}
	yielding the formula for $ V(\theta,a,\ell) $. 
    
    Furthermore, as is shown in Lemma \ref{lem:HFunction} below the function $ H_\kappa $ and hence also $ R $ is smooth. In Lemma \ref{lem:DerivRadFunct} \eqref{it:lem:DerivRadFunct2} below we obtain $ V = \partial_{\theta}R/a $, so that $V$ is smooth too. Below we show that $ (\theta,a)\mapsto (R(\theta,a,\ell),V(\theta,a,\ell)) $ is a canonical change of variables, and by the inverse function theorem (or direct computation) it is a diffeomorphism.
    
    To show that it is canonical one can use the fact that they are locally given through a generating function $ S^\pm_\ell(r,a) $ with
	\begin{align*}
		\dfrac{\partial S^\pm_\ell}{\partial r} = V(r,a) = \pm a\sqrt{1+\dfrac{m}{a^2r}-\dfrac{\ell}{a^2r^2}}, \quad \dfrac{\partial S^\pm_\ell}{\partial a} = \Theta(r,v).
	\end{align*}
	Let us define $ \sigma = \sgn(\theta)=\sgn(v)\in \{-1,+1\} $ and
	\begin{align}\label{eq:Sec2:GenFct}
		S^\sigma_\ell(r,a) :=& \sigma\int_{r_0}^{r} a\sqrt{1+\dfrac{m}{a^2\rho}-\dfrac{\ell}{a^2\rho^2}}\, d\rho,
	\end{align}
	Observe that by the computation done in the proof of Lemma \ref{lem:HypTraj}
	\begin{align*}
		\dfrac{\partial S^\sigma_\ell(r,a)}{\partial a} =& \sigma\int_{r_0}^{r} \dfrac{1}{\sqrt{1+\dfrac{m}{a^2\rho}-\dfrac{\ell}{a^2\rho^2}}}\, d\rho = \sigma pG_\kappa\left( \dfrac{r}{p}+\kappa \right).
	\end{align*}
	Hence, $ \Theta(r,v) $ coincides with $ \partial S^\sigma_\ell/\partial a $. 
	
	Finally, the fact that \eqref{eq:Sec2:SolActAng} gives a parametrization of all the solutions with positive energy follows from the formulas in \eqref{it:lem:HypTraj5} in Lemma \ref{lem:HypTraj}.
\end{proof}

\begin{rem}
	We remark that the term $ \sigma=\sgn \theta= \sgn v $ in \eqref{eq:Sec2:GenFct} is not a function of $ (r,a) $ alone. In particular, in order to obtain a globally defined generating function a patching of the two functions $ S_\ell^\pm $ is required (see also \cite[Section 2.2]{PW2020} for a related discussion).
\end{rem}

\subsection{Preliminary estimates on the action-angle coordinates}
In this section we provide estimates on the change of variables introduced in the previous section. We then derive estimates for the kinematic variables $ \tilde{R} $. Here, we define for any function $ g(t,\theta,a,\ell) $ 
\begin{align*}
	\tilde{g}(t,\theta,a,\ell) := g(t,\theta+at,a,\ell).
\end{align*}
We will usually write $ \tilde{R}(\theta,a) $ and suppress the dependence on $ (t,\ell) $. 

Concerning estimates on $ \tilde{R} $ we distinguish those holding for all $ (t,\theta,a,\ell) $ and those improved estimates when $ (t,\theta,a,\ell) $ is restricted to the bulk $ \mathcal{B} $, that is
\begin{align}\label{eq:Sec2:DefBulk}
	\mathcal{B} := \left\lbrace |\theta|\leq at/2 \right\rbrace.
\end{align}

\medskip
\paragraph{\bf Estimates on radial function.} The following lemma will be useful for our further estimates.
\begin{lem}\label{lem:EstimatesG}
	It holds for all $ y\geq1 $
	\begin{align*}
		\dfrac{G_\kappa(y)}{y-\kappa} &\lesssim 1,
		\\
		 G_{\kappa}(y) &= y -\kappa \ln y + \mathcal{O}(1), \quad y\to \infty.
	\end{align*}
\end{lem}
\begin{proof}
	Observe first that for $ y=1+h $ we have
	\begin{align*}
		G_\kappa(y) &= \sqrt{y^2-1}-\kappa\ln\left(y+\sqrt{y^2-1}\right) = \sqrt{h(2+h)}-\kappa\ln\left(1+h+\sqrt{h(2+h)}\right)
		\\
		&= (1-\kappa)\sqrt{h(2+h)} +  \mathcal{O}(h).
	\end{align*}
	We then have
	\begin{align*}
		\dfrac{G_\kappa(y)}{y-\kappa} =\dfrac{\sqrt{y^2-1}-\kappa\ln\left(y+\sqrt{y^2-1}\right)}{y-\kappa}.
	\end{align*}
	For $ y=1+h $, $ h\in (0,1) $ we get 
	\begin{align*}
		\dfrac{G_\kappa(y)}{y-\kappa} = \dfrac{(1-\kappa)\sqrt{h(2+h)} + \mathcal{O}(h)}{1-\kappa+h} \leq \sqrt{h(2+h)} + \mathcal{O}(1) \lesssim 1.
	\end{align*}
	On the other hand, we have for $ y\geq 2 $
	\begin{align*}
		\dfrac{G_\kappa(y)}{y-\kappa} \lesssim \dfrac{\sqrt{y^2-1}+\kappa\ln\left(y+\sqrt{y^2-1}\right)}{y} \lesssim 1.
	\end{align*}
	The other estimate follows from 
	\begin{align*}
		G_\kappa(y) = y \sqrt{1-\dfrac{1}{y^2}} -\kappa \ln y - \kappa \ln\left( 1+\sqrt{1-\dfrac{1}{y^2}} \right)  = y-\kappa \ln y +  \mathcal{O}(1).
	\end{align*}
\end{proof}

In the following we collect some properties of the function $ H_\kappa $, see also Figure \ref{fig:HFunct}.
\begin{lem}\label{lem:HFunction}
	For any $ \kappa\in(0,1) $ the function $ H_\kappa $ is smooth and
	\begin{align}
		H_\kappa(x) &= 1+\dfrac{x^2}{2(1-\kappa)^2}+o\left( \dfrac{x^2}{(1-\kappa)^2}\right), \quad x\to 0, \label{eq:Sec2:HFunction1}
		\\
		H_\kappa(x) &= x + \kappa \ln x + \kappa^2\dfrac{\ln x}{x} + \mathcal{O}(1), \quad x\to \infty. \label{eq:Sec2:HFunction2}
	\end{align}
\end{lem}
\begin{proof}
	We first prove \eqref{eq:Sec2:HFunction1}. To this end, we observe that for any $ \kappa\in (0,1) $ we have $ H_\kappa(0)=1 $ and $H_\kappa$ satisfies
	\begin{align*}
		H_\kappa' = \dfrac{1}{G_\kappa'\circ H_\kappa} = \dfrac{\sqrt{H_\kappa^2-1}}{H_\kappa-\kappa}.
	\end{align*}
	In particular, the ansatz
	\begin{align*}
		H_\kappa(x) = 1 + (1-\kappa) \left( \eta_\kappa\left( \dfrac{x}{(1-\kappa)^{3/2}} \right)  \right)^2
	\end{align*}
	yields the differential equation
	\begin{align*}
		\eta'_\kappa = \dfrac{1}{2}\dfrac{\sqrt{2+(1-\kappa)\eta_\kappa^2}}{1+\eta_\kappa^2}, \quad \eta_\kappa(0)=0.
	\end{align*} 
	One can readily see that $ \eta_{\kappa} $ is smooth and 
	\begin{align*}
		\eta_{\kappa}(x)=\frac{1}{\sqrt{2}}x+o(x), \quad x\to 0.
	\end{align*}
	This yields  \eqref{eq:Sec2:HFunction1}. In order to show  \eqref{eq:Sec2:HFunction2} one can use Lemma \ref{lem:EstimatesG} to obtain that
	\begin{align*}
		x = H_\kappa(x) -\kappa\ln H_\kappa(x) + \mathcal{O}(1), \quad x\to \infty.
	\end{align*}
\end{proof}

\begin{lem}\label{lem:HFunctEquiv}
	Let $ C_0\geq1 $. There exists a constant $ C_1=C_1(C_0)\geq1 $ such that for all $ x_1,\, x_2>0 $ and $ \kappa_1,\, \kappa_2\in (0,1) $ satisfying
	\begin{align}\label{eq:assume-kappa}
		\dfrac{x_1}{x_2},\, \dfrac{\kappa_1}{\kappa_2},\, \dfrac{1-\kappa_1}{1-\kappa_2}\in \left[ \dfrac{1}{C_0},C_0 \right], 
	\end{align}
	we have
	\begin{align}\label{eq:Hkappa-compare}
		\dfrac{1}{C_1}\leq \dfrac{H_{\kappa_1}(x_1)-1}{H_{\kappa_2}(x_2)-1}\leq C_1.
	\end{align}
\end{lem}
\begin{proof}
	As in Lemma \ref{lem:HFunction} we write the ansatz
	\begin{align}\label{eq:Sec2:ProofHFunctEquiv}
		H_\kappa(x) = 1 + (1-\kappa) \left( \eta_\kappa\left( \dfrac{x}{(1-\kappa)^{3/2}} \right)  \right)^2,
	\end{align}
	where $ \eta_{\kappa} $ satisfies
	\begin{align}\label{eq:kappa_ODE}
		\eta'_\kappa = \dfrac{1}{2}\dfrac{\sqrt{2+(1-\kappa)\eta_\kappa^2}}{1+\eta_\kappa^2}=:E_\kappa(\eta_\kappa), \quad \eta_\kappa(0)=0.
	\end{align}
	Note that $ E'_\kappa\leq0 $, and $\partial_\kappa E_\kappa\leq 0$. Based on this we divide the proof into three steps.

    \textit{Step 1.} Let $\eta_\kappa$ solve \eqref{eq:kappa_ODE}. We claim that for all $ x\geq 0 $ and $ C\geq 1 $ there holds
	\begin{equation}\label{eq:etakappa_compare}
		\eta_\kappa(x)\leq \eta_\kappa(Cx) \leq C\eta_\kappa(x).
	\end{equation}
	The first inequality follows from the fact that $ \eta_\kappa $ is increasing. For the second one, observe that $ \tilde{\eta}_\kappa(x) = \eta_\kappa(Cx)/C $ satisfies
    \begin{align*}
		\tilde{\eta}_\kappa' = E_\kappa(C_1\tilde{\eta}_\kappa)\leq E_\kappa(\tilde{\eta}_\kappa),
	\end{align*}
    since $ E_\kappa $ is decreasing. Since $\tilde{\eta}_\kappa(0)=\eta_\kappa(0)=0$, a comparison argument gives the claim.

    \textit{Step 2.} Let $\kappa_1\leq \kappa_2$ as in \eqref{eq:assume-kappa}. We claim that
    for all $ x\geq0 $
	\begin{equation}\label{eq:kappa12_compare}
		\eta_{\kappa_2}(x)\leq \eta_{\kappa_1}(x) \leq \sqrt{C_0}\eta_{\kappa_2}(x).
	\end{equation}
    A comparison argument gives the first inequality, since $ \kappa\mapsto E_\kappa $ is decreasing. Similarly, for the second one we use that by assumption $ (1-\kappa_1)\leq C_0(1-\kappa_2) $ and thus $ E_{\kappa_1} \leq \sqrt{C_0} E_{\kappa_2} $, so that another comparison argument gives
	\begin{align*}
		\eta_{\kappa_1}(x)\leq \eta_{\kappa_2}\left( \sqrt{C_0}x\right) \leq \sqrt{C_0}\eta_{\kappa_2}(x),
	\end{align*}
	where in the last step we have used \eqref{eq:etakappa_compare}.

	\textit{Step 3.} We deduce \eqref{eq:Hkappa-compare} by invoking \eqref{eq:Sec2:ProofHFunctEquiv} together with the bounds \eqref{eq:etakappa_compare} and \eqref{eq:kappa12_compare}.
\end{proof}

\begin{figure}
	\centering
	\begin{minipage}{.5\textwidth}
		\centering
		\includegraphics[width=\linewidth]{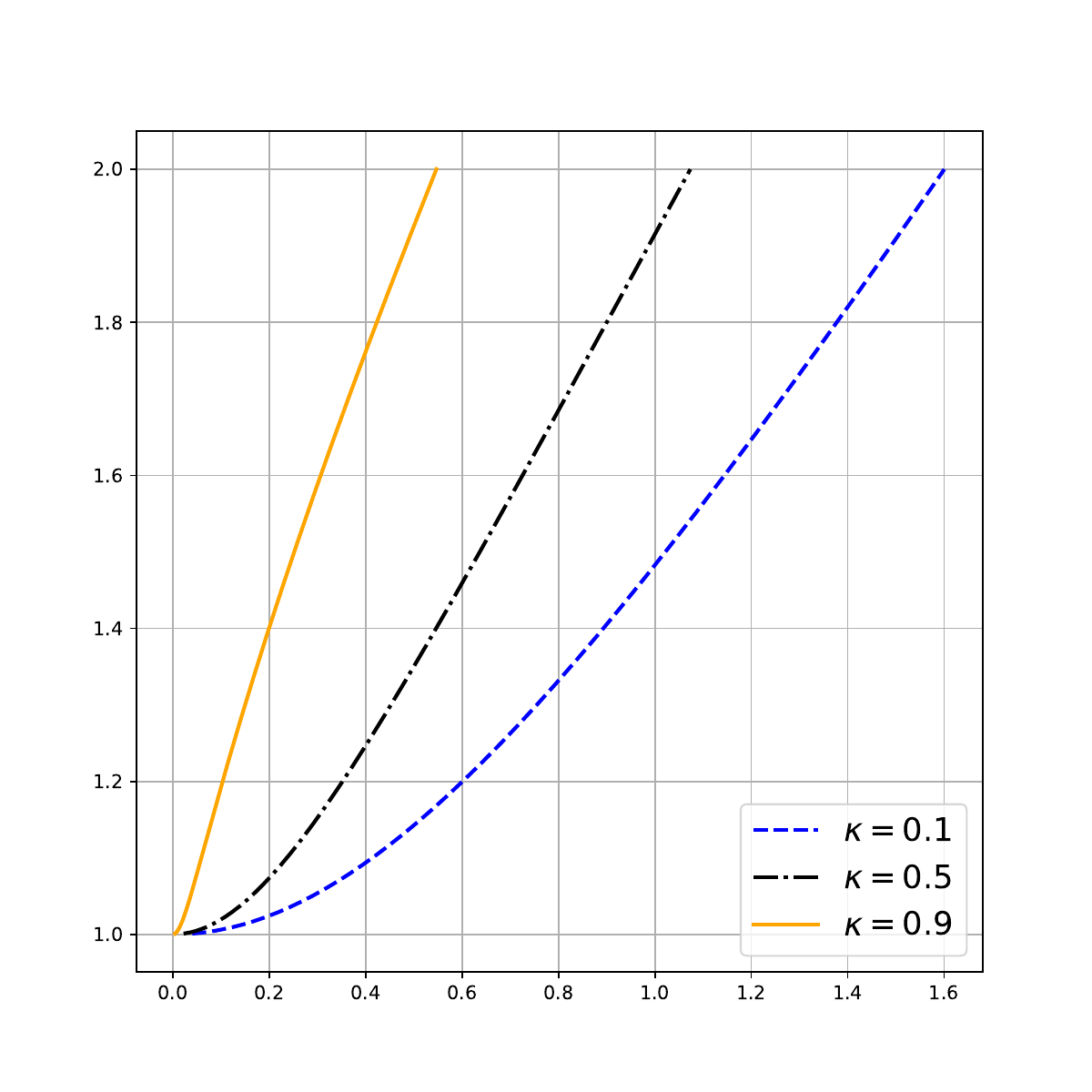}
	\end{minipage}%
	\begin{minipage}{.5\textwidth}
		\centering
		\includegraphics[width=\linewidth]{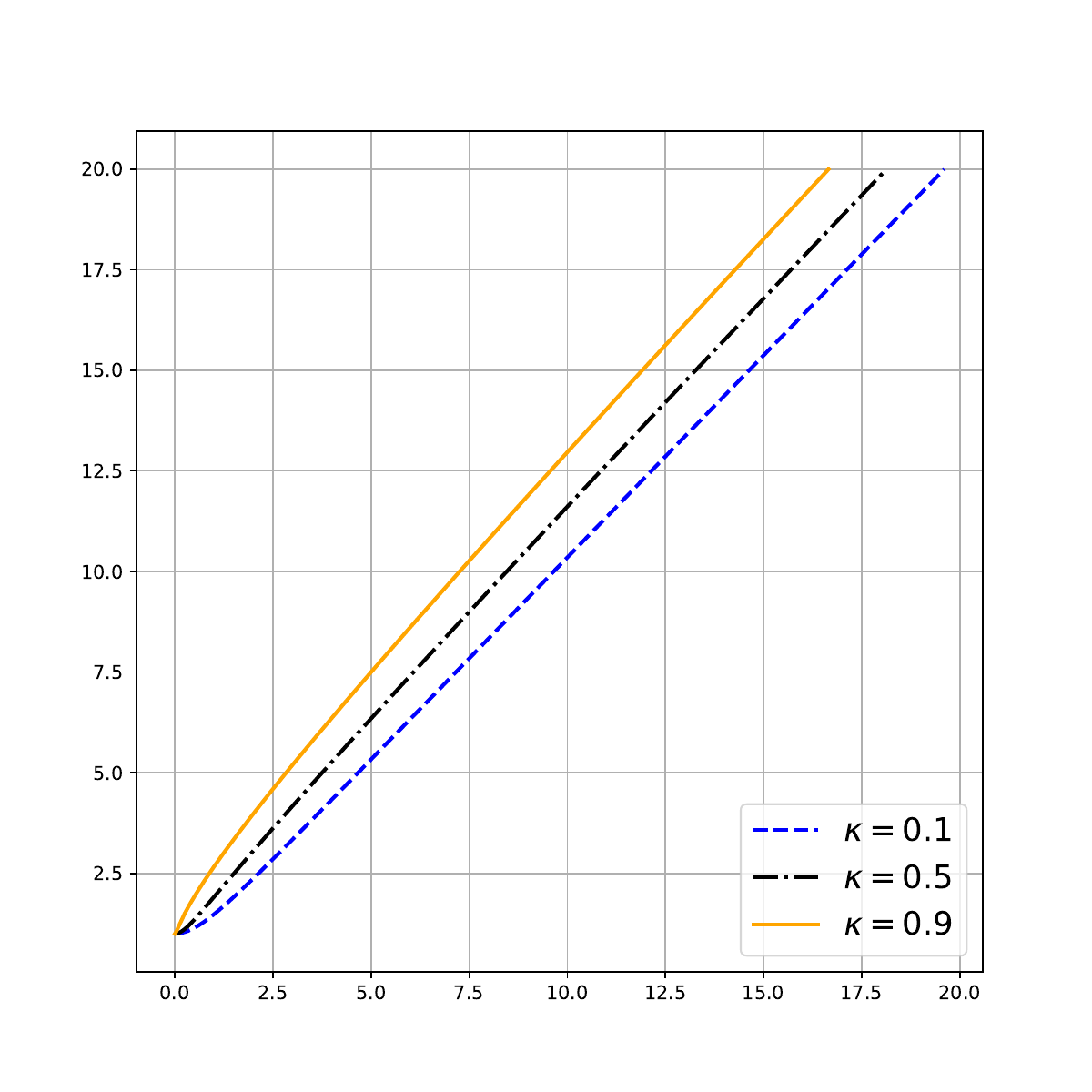}
	\end{minipage}
	\caption{Plots of the function $ H_\kappa(x) $ for small and large values of $ x $ in the cases $ \kappa=0.1, \, 0.5, \, 0.9 $.}
	\label{fig:HFunct}
\end{figure}

\begin{lem}\label{lem:EstimatesRadFunct}
	The following estimates hold.
	\begin{enumerate}[(i)]
		\item\label{it:lem:EstimatesRadFunct1} It holds 
		\begin{align*}
			r_0(a,\ell)\approx \dfrac{\ell}{\jabr{a\sqrt{\ell}}}, \quad
			\kappa(a,\ell) \approx \jabr{a\sqrt{\ell}}^{-1}, \quad
			p(a,\ell) \approx \dfrac{1}{a^2}\jabr{a\sqrt{\ell}}.
		\end{align*}
		as well as $ \kappa\in (0,1) $ and
		\begin{align*}
			\dfrac{1}{1-\kappa} \approx \jabr{\dfrac{1}{a^2\ell}}, \quad a\ell \kappa^2 \approx \min\{a\ell, \dfrac{1}{a}\}.
		\end{align*}
		\item\label{it:lem:EstimatesRadFunct2} We have for all $ (\theta,a,\ell)\in \R\times \R_+\times\R_+ $
		\begin{align*}
			R&\geq r_0=p(1-\kappa), \quad \dfrac{R}{p}+\kappa \approx \jabr{\dfrac{R}{p}},
			\\
			\left| \dfrac{R}{p}+\kappa-\dfrac{|\theta|}{p} \right| & \lesssim \ln\jabr{\dfrac{R}{p}}, \quad
			|\theta| \lesssim  R \lesssim p+|\theta|.
		\end{align*}
	\end{enumerate}
\end{lem}
\begin{proof}
	For part \eqref{it:lem:EstimatesRadFunct1} we have
	\begin{align*}
		r_0(a,\ell) &= \dfrac{2\ell}{m} \dfrac{1}{1+\sqrt{1+\dfrac{4a^2\ell}{m^2}}}\approx \dfrac{\ell}{\jabr{a^2\ell}^{1/2}}\approx \dfrac{\ell}{\jabr{a\sqrt{\ell}}},
		\\
		\kappa &= \left( 1+\dfrac{4a^2\ell}{m^2}\right)^{-1/2}\approx \jabr{a^2\ell}^{-1/2}\approx \jabr{a\sqrt{\ell}}^{-1},
		\\
		p(a,\ell) &= \dfrac{m}{2a^2}\left( 1+\dfrac{4a^2\ell}{m^2} \right)^{1/2} \approx\dfrac{1}{a^2}\jabr{a^2\ell}^{1/2}\approx \dfrac{1}{a^2}\jabr{a\sqrt{\ell}}.
	\end{align*}
	By definition we have $ \kappa \in (0,1) $. Furthermore, we have
	\begin{align*}
		\dfrac{1}{1-\kappa} = (1+\kappa)\dfrac{1}{1-\kappa^2} = (1+\kappa) \dfrac{m^2}{4a^2\ell} \left( 1+ \dfrac{4a^2\ell}{m^2}\right)  = (1+\kappa) \left( 1+ \dfrac{m^2}{4a^2\ell}\right) \approx 1+\dfrac{1}{a^2\ell}.
	\end{align*}
	In addition, we have for $ a\ell\leq 1/a $ respectively $ a\ell\geq 1/a $ 
	\begin{align*}
		a\ell \kappa^2 = \dfrac{a\ell}{1+\dfrac{4a^2\ell}{m^2}} \approx  a\ell, \quad a\ell \kappa^2 \approx \dfrac{1}{a},
	\end{align*}
	which yields the assertion.
	
	Concerning part \eqref{it:lem:EstimatesRadFunct2} we first observe that $ R(\theta,a)\geq r_0(a,\ell) = p(1-\kappa) $ by definition. In addition, we have for $ \kappa\geq1/2 $ and $ \kappa \leq 1/2 $, respectively,
	\begin{align*}
		\dfrac{R}{p}+\kappa \gtrsim\jabr{\dfrac{R}{p}} ,\quad \dfrac{R}{p}+\kappa  \geq \dfrac{1}{2}\dfrac{R}{p} + \dfrac{1-\kappa}{2} \gtrsim\jabr{\dfrac{R}{p}}. 
	\end{align*}
	Furthermore, we have for $ x=|\theta|/p\geq0 $, $ y= H_\kappa(x)\geq1 $
	\begin{align*}
		\left| \dfrac{R}{p}+\kappa-\dfrac{|\theta|}{p} \right| &= \left| H_\kappa(x) -x \right| = \left| y -G_\kappa(y) \right| = \left| \sqrt{y^2-1}-y-\kappa \ln \left( y+ \sqrt{y^2-1} \right)  \right| 
		\\
		&=\left| y\left( \sqrt{1-\dfrac{1}{y^2}}-1\right) -\kappa \ln \left( y \left(  1+ \sqrt{1-\dfrac{1}{y^2}} \right) \right)  \right| \lesssim \dfrac{1}{y} + \kappa \ln (1+y) 
		\\
		&\leq \left( \dfrac{R}{p}+\kappa\right)^{-1} + \kappa \ln\jabr{\dfrac{R}{p}} \lesssim \ln\jabr{\dfrac{R}{p}},
	\end{align*}
	where we used $ y= R/p+\kappa $. 
	
	In addition, due to Lemma \ref{lem:EstimatesG} it holds
	\begin{align*}
		\dfrac{|\theta|}{R} = \dfrac{|\theta|/p}{R/p} = \dfrac{x}{H_\kappa(x)-\kappa} = \dfrac{G_\kappa(y)}{y-\kappa} \lesssim 1.
	\end{align*}
	Furthermore, we have by Lemma \ref{lem:EstimatesG} 
	\begin{align*}
		y \lesssim 1+G_\kappa(y)  \implies H_\kappa(x) \lesssim 1+x
	\end{align*}	
	and hence
	\begin{align*}
		\dfrac{R}{p}+\kappa = H_\kappa\left( \dfrac{|\theta|}{p} \right) \lesssim 1+\dfrac{|\theta|}{p}
	\end{align*}
	which yields $ R \lesssim p+|\theta| $. This concludes the proof.
\end{proof}

We start with formulas for first order derivatives.
\begin{lem}\label{lem:DerivRadFunct}
	The following equalities hold:
	\begin{enumerate}[(i)]
		\item\label{it:lem:DerivRadFunct1} $ \partial_a p = -p(1+\kappa^2)/a $, $ \partial_a\kappa=-4a\ell\kappa^3/m^2 $, $ \partial_a(p\kappa)=-m/a^3 $.
		
		\item\label{it:lem:DerivRadFunct2} We have
		\begin{align*}
			\partial_\theta R &=\dfrac{1}{a}V =  \sgn\theta\,  \dfrac{p}{R}\sqrt{\left( \dfrac{R}{p}+\kappa \right)^2-1},
			\\
			\partial_a R &= \partial_a p \left[ \dfrac{R}{p}+\kappa -\dfrac{\theta}{p}\partial_\theta R \right] + p\partial_a\kappa \arcosh\left( \dfrac{R}{p}+\kappa \right) \, \sgn\theta \, \partial_\theta R  +\dfrac{m}{a^3}.
		\end{align*}		
	\end{enumerate}
\end{lem}
\begin{proof}
	The first statement can be readily checked. For the first equality in \eqref{it:lem:DerivRadFunct2} we use the fact that 
	\begin{align}
		\partial_\theta  R(\theta,a) = \dfrac{1}{a}\dfrac{d}{dt}\mid_{t=0} R(\theta+at,a) = \dfrac{1}{a}V(\theta,a) = \sgn\theta\,  \dfrac{p}{R}\sqrt{\left( \dfrac{R}{p}+\kappa \right)^2-1 }.
	\end{align}
	due to \eqref{eq:Sec2:SolActAng} and the definition of $ \kappa, p $. The formula for $ \partial_a R $ follows from differentiating
	\begin{align*}
		R(\theta,a) = pH_\kappa\left( \dfrac{|\theta|}{p} \right) -p\kappa
	\end{align*}
	and using (see Lemma \ref{lem:HypTraj} \eqref{it:lem:HypTraj5} for the definition of $ G_\kappa $)
	\begin{align*}
		\partial_\kappa H_\kappa(x) &= -\dfrac{\partial_\kappa G_\kappa(H_\kappa(x))}{G_\kappa'(H_\kappa(x))} = \arcosh(H_\kappa(x)) H_\kappa'(x), \quad x = |\theta|/p
		\\
		H_\kappa'\left( \dfrac{|\theta|}{p} \right) &= \sgn\theta\, \partial_\theta R(\theta,a), \quad H_\kappa\left( \dfrac{|\theta|}{p} \right) = \dfrac{R(\theta,a)}{p}+\kappa.
	\end{align*}
	This yields the assertion.
\end{proof}

\begin{lem}\label{lem:EstimatesDerivRadFunct}
	We have the following statements.
	\begin{enumerate}[(i)]
		\item\label{it:lem:EstimatesDerivRadFunct1} The following estimates hold:
		\begin{align*}
			|\partial_ap|\approx \dfrac{1}{a^3}\jabr{a\sqrt{\ell}}, \quad |\partial_a\kappa| \approx  \dfrac{1}{a}\jabr{a\sqrt{\ell}}^{-1} \min\{ 1,a^2\ell \}. 
		\end{align*}
	
		\item\label{it:lem:EstimatesDerivRadFunct2} It holds for all $ (\theta,a,\ell)\in \R\times \R_+\times\R_+ $
		\begin{align*}
			\left| \sgn \theta \, \partial_\theta R - 1 \right| &\lesssim \dfrac{p}{R},
			\\
			\left| \partial_\theta R \right| &\lesssim \jabr{\dfrac{p}{R}}^{1/2} \lesssim \jabr{\dfrac{1}{a\sqrt{\ell}}},
			\\
			\left| \partial_aR \right| &\lesssim \dfrac{p}{a} \ln\jabr{\dfrac{R}{p}} \approx \dfrac{1}{a^3}\jabr{a\sqrt{\ell}}\ln\jabr{\dfrac{R}{p}}.
		\end{align*}	
	\end{enumerate}
\end{lem}
\begin{proof}
	To prove \eqref{it:lem:EstimatesDerivRadFunct1} we use Lemma \ref{lem:EstimatesRadFunct} \eqref{it:lem:EstimatesRadFunct1} and Lemma \ref{lem:DerivRadFunct} \eqref{it:lem:DerivRadFunct1} to get
	\begin{align*}
		|\partial_ap| &\approx \dfrac{p}{a} \approx \dfrac{1}{a^3}\jabr{a\sqrt{\ell}},
		\\
		|\partial_a\kappa| &\approx a\ell\kappa^3 \approx \kappa \min\{ a\ell,\dfrac{1}{a}\}\approx \dfrac{1}{a}\jabr{a\sqrt{\ell}}^{-1} \min\{ 1, a^2\ell \}. 
	\end{align*}
	
	For \eqref{it:lem:EstimatesDerivRadFunct2} we first observe with Lemma \ref{lem:DerivRadFunct} \eqref{it:lem:DerivRadFunct2}
	\begin{align*}
		\left| \sgn \theta \, \partial_\theta R - 1 \right|  &= \left| \dfrac{p}{R}\sqrt{\left(\dfrac{R}{p}+\kappa \right)^2-1} - 1 \right| = \left| \sqrt{1+\dfrac{2\kappa p}{R}-\dfrac{(1-\kappa^2)p^2}{R^2}} - 1 \right|
		\\
		&\leq \left| \dfrac{2\kappa p}{R}-\dfrac{(1-\kappa^2)p^2}{R^2} \right| \leq  \dfrac{p}{R}\left| 2\kappa+(1+\kappa) \right| \leq \dfrac{4p}{R} ,
	\end{align*}
	where we used $ |\sqrt{1+x}-1|\leq |x| $ for $ x\geq-1 $. Furthermore, we have with Lemma \ref{lem:DerivRadFunct} \eqref{it:lem:DerivRadFunct2}
	\begin{align*}
		\left| \partial_\theta R \right| &= \dfrac{1}{a} \left| V(\theta,a) \right| = \sqrt{1 + \dfrac{m}{a^2R}-\dfrac{\ell}{a^2R^2}} \leq \sqrt{1 + \dfrac{1}{a^2R}\left| m-\dfrac{\ell}{R} \right|} \lesssim \sqrt{1 + \dfrac{\jabr{a\sqrt{\ell}}}{a^2R}} \lesssim \jabr{\dfrac{p}{R}}^{1/2}.
	\end{align*}
	In the above estimates we used the fact that $ \ell / R \leq \ell/r_0 \lesssim \jabr{a\sqrt{\ell}} $, $ p\approx\jabr{a\sqrt{\ell}}/a^2 $ according to Lemma \ref{lem:EstimatesRadFunct} \eqref{it:lem:EstimatesRadFunct1}. In addition, we have with $ R\geq r_0=p(1-\kappa) $ and \ref{lem:EstimatesRadFunct} \eqref{it:lem:EstimatesRadFunct1}
	\begin{align*}
		\left| \partial_\theta R\right| \lesssim \jabr{\dfrac{p}{R}}^{1/2} \lesssim \jabr{\dfrac{1}{\sqrt{1-\kappa}}}\approx \jabr{\dfrac{1}{a\sqrt{\ell}}}.
	\end{align*}
	
	The last assertion follows using Lemma \ref{lem:DerivRadFunct} and the previous estimates. Indeed, we have
	\begin{align*}
		|\partial_a R| & \lesssim |\partial_a p| \left| \dfrac{R}{p}+\kappa -\dfrac{\theta}{p}\partial_\theta R \right| + p|\partial_a\kappa| \ln\jabr{\dfrac{R}{p}} \left(  \right| \sgn\theta \, \partial_\theta R- 1\left| +1  \right) + \dfrac{1}{a^3}.
	\end{align*}
	We observe that with Lemma \ref{lem:EstimatesRadFunct} \eqref{it:lem:EstimatesRadFunct2}
	\begin{align}\label{eq:Sec2:ProofEstDerivRadFunct}
		\left| \dfrac{R}{p}+\kappa -\dfrac{\theta}{p}\partial_\theta R \right|  &\leq 	\left| \dfrac{R}{p}+\kappa -\dfrac{|\theta|}{p} \right| + \dfrac{|\theta|}{p}	\left| 1 -\sgn \theta \partial_\theta R \right| \lesssim \ln\jabr{\dfrac{R}{p}} + \dfrac{|\theta|}{p}\dfrac{p}{R} \lesssim \ln\jabr{\dfrac{R}{p}}.
	\end{align}
	Furthermore, we have 
	\begin{align*}
		\left| \sgn\theta \, \partial_\theta R- 1\right| \lesssim \jabr{\dfrac{1}{a\sqrt{\ell}}}.
	\end{align*}
	Hence, we obtain with Lemma \ref{lem:DerivRadFunct} \eqref{it:lem:DerivRadFunct1} and the previous estimate in \eqref{it:lem:EstimatesDerivRadFunct1}
	\begin{align*}
		|\partial_a R| &\lesssim \left[ |\partial_ap| + p|\partial_a\kappa| \jabr{\dfrac{1}{a\sqrt{\ell}}} \right] \ln\jabr{\dfrac{R}{p}} + \dfrac{1}{a^3} 
		\\
		&\lesssim \left[\dfrac{1}{a^3}\jabr{a\sqrt{\ell}} + \dfrac{1}{a^3}\min\{ 1, a^2\ell \}  \jabr{\dfrac{1}{a\sqrt{\ell}}} \right] \ln\jabr{\dfrac{R}{p}} + \dfrac{1}{a^3} 
		\\
		&\lesssim \dfrac{1}{a^3}\jabr{a\sqrt{\ell}} \ln\jabr{\dfrac{R}{p}} \approx \dfrac{p}{a}\ln\jabr{\dfrac{R}{p}}.
	\end{align*}
	In the last step we used $ p\approx \jabr{a\sqrt{\ell}}/a^2 $ according to \ref{lem:EstimatesRadFunct} \eqref{it:lem:EstimatesRadFunct1}. This concludes the proof.
\end{proof}
In the following lemmas we treat second order derivatives.
 \begin{lem}\label{lem:SecondDerivRadFunct}
 	The following equalities hold.
 	\begin{enumerate}[(i)]
 		\item\label{it:lem:SecondDerivRadFunct1} We have
 		\begin{align*}
 			\partial_a^2 p &= \dfrac{p(1+\kappa^2)}{a^2} + \dfrac{p(1+\kappa^2)^2}{a^2} + \dfrac{8p\ell\kappa^4}{m^2},
 			\\
 			\partial_a^2 \kappa &= -\dfrac{4\ell\kappa^3}{m^2} + \dfrac{48a^2\kappa^5\ell^2}{m^4}.
 		\end{align*}
 		\item\label{it:lem:SecondDerivRadFunct2} We have
 		\begin{align*}
 			\partial_\theta^2 R &= \dfrac{1}{a^2R^2}\left( \dfrac{\ell}{R}-\dfrac{m}{2} \right),
 			\\
 			\partial^2_{\theta a} R &= \dfrac{p^2}{R^2}\left[ (1-\kappa^2)\dfrac{p}{R}-\kappa \right] \left[ -\theta\dfrac{\partial_ap}{p^2} + \partial_a\kappa \arcosh\left( \dfrac{R}{p}+\kappa \right) \sgn\theta \right] +\dfrac{p}{R}\partial_a\kappa \partial_\theta R,
 			\\
 			\partial_a^2 R &= \partial^2_ap \left[  \dfrac{R}{p}+\kappa -\dfrac{\theta}{p}\partial_\theta R \right] + \left[ \dfrac{p^2}{R}(\partial_a\kappa)^2 + 2\partial_ap\partial_a\kappa +p\partial^2_a\kappa \right]  \arcosh\left( \dfrac{R}{p}+\kappa \right)\sgn\theta \partial_\theta R
 			\\
 			&\quad +\partial_{\theta a}^2R \left[ -\dfrac{\partial_ap}{p}\theta +p\partial_a\kappa \arcosh\left( \dfrac{R}{p}+\kappa \right)\sgn\theta  \right] - \theta\dfrac{\partial_ap\partial_a\kappa}{R} \partial_{\theta} R - \dfrac{3m}{a^4}.
 		\end{align*}
 	\end{enumerate}
 \end{lem}
\begin{proof}
	For assertion \eqref{it:lem:SecondDerivRadFunct1} we use Lemma \ref{lem:DerivRadFunct} yielding
	\begin{align*}
		\partial_a^2p &= \partial_a\left[ -\dfrac{p(1+\kappa^2)}{a} \right] = \dfrac{p(1+\kappa^2)}{a^2} + \dfrac{p(1+\kappa^2)^2}{a^2} + \dfrac{8p\ell\kappa^4}{m^2},
		\\
		\partial_a^2\kappa &= \partial_a\left[ -\dfrac{4a\ell\kappa^3}{m^2} \right] = -\dfrac{4\ell\kappa^3}{m^2} + \dfrac{48a^2\kappa^5\ell^2}{m^4}.
	\end{align*}
	Concerning \eqref{it:lem:SecondDerivRadFunct2} we have with \eqref{eq:Sec2:SolActAng} Proposition \ref{pro:ActionAngleVariables} and \eqref{eq:Sec2:CharSys}
	\begin{align*}
		\partial_\theta^2R = \dfrac{1}{a^2}\dfrac{d^2}{dt^2}\mid_{t=0} R(\theta +at,a) = \dfrac{1}{a^2}\left[ \dfrac{\ell}{R^3}-\dfrac{m}{2}\dfrac{1}{R^2} \right] = \dfrac{1}{a^2R^2}\left( \dfrac{\ell}{R}-\dfrac{m}{2} \right).
	\end{align*}
	Furthermore, we have with Lemma \ref{lem:DerivRadFunct} \eqref{it:lem:DerivRadFunct2}
	\begin{align*}
		\partial^2_{\theta a} R &= \partial_a \left[ \sgn\theta\,  \dfrac{p}{R}\sqrt{\left( \dfrac{R}{p}+\kappa \right)^2-1 }\right]  
		\\
		&= \partial_a\left( \dfrac{p}{R}\right) \dfrac{R}{p}\partial_\theta R  + \sgn\theta\, \dfrac{p}{R}\left[ \left( \dfrac{R}{p}+\kappa \right)^2-1  \right]^{-1/2} \left( \dfrac{R}{p}+\kappa \right) \partial_a\left( \dfrac{R}{p}+\kappa \right) 
		\\
		&= -\partial_a\left( \dfrac{R}{p}\right) \dfrac{p}{R}\partial_\theta R+ \dfrac{1}{\partial_\theta R} \dfrac{p^2}{R^2}\left( \dfrac{R}{p}+\kappa \right) \partial_a\left( \dfrac{R}{p}+\kappa \right)
		\\
		&= -\partial_a\left( \dfrac{R}{p}+\kappa\right) \dfrac{p}{R}\partial_\theta R + \dfrac{p}{R}\partial_a\kappa \partial_\theta R + \dfrac{1}{\partial_\theta R} \dfrac{p^2}{R^2}\left( \dfrac{R}{p}+\kappa \right) \partial_a\left( \dfrac{R}{p}+\kappa \right)
		\\
		&= \dfrac{p}{R} \partial_a\left( \dfrac{R}{p}+\kappa\right) \dfrac{1}{\partial_\theta R} \left[ -|\partial_{\theta}R|+1 + \dfrac{\kappa p}{R} \right] +\dfrac{p}{R}\partial_a\kappa \partial_\theta R.
	\end{align*}
	Note that by Lemma \ref{lem:DerivRadFunct} \eqref{it:lem:DerivRadFunct2}
	\begin{align*}
		-|\partial_{\theta}R|+1 + \dfrac{\kappa p}{R} = (1-\kappa^2)\dfrac{p^2}{R^2}-\dfrac{\kappa p}{R} = \dfrac{p}{R}\left[ (1-\kappa^2)\dfrac{p}{R}-\kappa \right] .
	\end{align*}
	Furthermore, we have using Lemma \ref{lem:DerivRadFunct} \eqref{it:lem:DerivRadFunct2}
	\begin{align}\label{eq:Sec2:ProofSecDeriv}
		\partial_a\left( \dfrac{R}{p}+\kappa \right) = \dfrac{\partial_aR}{p} - \dfrac{R\partial_ap}{p^2} + \partial_{a}\kappa = -\theta\dfrac{\partial_ap}{p^2}\partial_\theta R + \partial_a\kappa \arcosh\left( \dfrac{R}{p}+\kappa \right) \sgn\theta\, \partial_\theta R.
	\end{align}
	We hence obtain
	\begin{align*}
		\partial^2_{\theta a} R = \dfrac{p^2}{R^2}\left[ (1-\kappa^2)\dfrac{p}{R}-\kappa \right] \left[ -\theta\dfrac{\partial_ap}{p^2} + \partial_a\kappa \arcosh\left( \dfrac{R}{p}+\kappa \right) \sgn\theta \right] +\dfrac{p}{R}\partial_a\kappa \partial_\theta R.
	\end{align*}
	Finally, we have with Lemma \ref{lem:DerivRadFunct} \eqref{it:lem:DerivRadFunct2}
	\begin{align*}
		\partial^2_a R &= \partial_a\left[ \partial_a p \left[ \dfrac{R}{p}+\kappa -\dfrac{\theta}{p}\partial_\theta R \right] + p\partial_a\kappa \arcosh\left( \dfrac{R}{p}+\kappa \right) \, \sgn\theta \, \partial_\theta R  +\dfrac{m}{a^3} \right]
		\\
		&= \partial_a^2 p \left[ \dfrac{R}{p}+\kappa -\dfrac{\theta}{p}\partial_\theta R \right] + \partial_ap \left[ \partial_a\left( \dfrac{R}{p}+\kappa \right) +\theta\dfrac{\partial_ap}{p^2}\partial_\theta R \right] - \dfrac{\partial_ap}{p} \theta \partial^2_{\theta a} R
		\\
		&\quad +\left( \partial_ap\partial_a\kappa+p\partial_a^2\kappa \right) \arcosh\left( \dfrac{R}{p}+\kappa \right) \, \sgn\theta \, \partial_\theta R
		\\
		&\quad +p\partial_a \kappa\left[  \left( \left( \dfrac{R}{p}+\kappa \right)^2-1 \right)^{-1/2}  \, \sgn\theta \, \partial_\theta R \partial_a\left( \dfrac{R}{p}+\kappa \right) + \arcosh\left( \dfrac{R}{p}+\kappa \right) \, \sgn\theta \partial^2_{\theta a}R \right]  - \dfrac{3m}{a^4}
	\end{align*}
	Using \eqref{eq:Sec2:ProofSecDeriv} we get
	\begin{align*}
		\partial^2_a R &= \partial_a^2 p \left[ \dfrac{R}{p}+\kappa -\dfrac{\theta}{p}\partial_\theta R \right] + \partial_ap \partial_a\kappa \arcosh\left( \dfrac{R}{p}+\kappa \right) \, \sgn\theta\, \partial_\theta R - \dfrac{\partial_ap}{p} \theta \partial^2_{\theta a} R
		\\
		&\quad +\left[\partial_ap\partial_a\kappa+p\partial_a^2\kappa \right] \arcosh\left( \dfrac{R}{p}+\kappa \right) \, \sgn\theta \, \partial_\theta R - \dfrac{3m}{a^4}
		\\
		&\quad +p\partial_a \kappa\left[  \dfrac{p}{R} \left(-\theta\dfrac{\partial_ap}{p^2}\partial_\theta R + \partial_a\kappa \arcosh\left( \dfrac{R}{p}+\kappa \right) \sgn\theta\, \partial_\theta R \right) + \arcosh\left( \dfrac{R}{p}+\kappa \right) \, \sgn\theta \partial^2_{\theta a}R \right].
	\end{align*}
	Rearranging terms concludes the proof.
\end{proof}

\begin{lem}\label{lem:EstimatesSecondDerivRadFunct}
	The following estimates hold:
	\begin{enumerate}[(i)]
		\item\label{it:lem:EstimatesSecondDerivRadFunct1} We have
		\begin{align*}
			|\partial^2_ap| \approx \dfrac{1}{a^4}\jabr{a\sqrt{\ell}}, \quad |\partial^2_a\kappa| \lesssim \dfrac{1}{a^2}\jabr{a\sqrt{\ell}}^{-1} \min\{ 1, a^2\ell \}. 
		\end{align*}
	
		\item\label{it:lem:EstimatesSecondDerivRadFunct2} It holds for all $ (\theta,a,\ell)\in \R\times \R_+\times\R_+ $
		\begin{align*}
			|\partial_{\theta}^2R| &\lesssim \dfrac{p}{R^2}\lesssim \dfrac{1}{a^2\ell^2}\jabr{a\sqrt{\ell}}^3,
			\\
			|\partial_{\theta a}^2R| &\lesssim \dfrac{1}{a}\dfrac{p}{R} \lesssim \dfrac{1}{a} \jabr{\dfrac{1}{a\sqrt{\ell}}}^2,
			\\
			|\partial_{a}^2R| &\lesssim \dfrac{p}{a^2} \ln \jabr{\dfrac{R}{p}}\approx  \dfrac{1}{a^4} \jabr{a\sqrt{\ell}}\ln \jabr{\dfrac{R}{p}}.
		\end{align*}
	\end{enumerate}
\end{lem}
\begin{proof}
	For \eqref{it:lem:EstimatesSecondDerivRadFunct1} we use Lemma \ref{lem:EstimatesRadFunct} \eqref{it:lem:EstimatesRadFunct1} Lemma \ref{lem:SecondDerivRadFunct} \eqref{it:lem:SecondDerivRadFunct1} to get
	\begin{align*}
		|\partial^2_ap| &\approx \dfrac{p}{a^2} + p\ell\kappa^4 \approx \dfrac{1}{a^4} \jabr{a\sqrt{\ell}} + \dfrac{\ell}{a^2} \jabr{a\sqrt{\ell}}^{-3} = \dfrac{1}{a^4} \jabr{a\sqrt{\ell}} \left(1+a^2\ell\jabr{a\sqrt{\ell}}^{-2} \right) \approx \dfrac{1}{a^4} \jabr{a\sqrt{\ell}},
		\\
		|\partial^2_a\kappa| &= \dfrac{4\ell\kappa^5}{m^2} \left| \kappa^{-2} - \dfrac{12a^2\ell}{m^2} \right| = \dfrac{4\ell\kappa^5}{m^2} \left| 1-\dfrac{8a^2\ell}{m^2} \right|.
	\end{align*}
	To estimate the last line we distinguish the cases $ a^2\ell\leq 1 $ and $ a^2\ell\geq1 $ to get, respectively,
	\begin{align*}
		|\partial^2_a\kappa| &\lesssim \ell \kappa^5 = \dfrac{a^2\ell \kappa^5}{a^2} \leq \dfrac{\kappa}{a^2}a^2\ell \approx \dfrac{1}{a^2}\jabr{a\sqrt{\ell}}^{-1}a^2\ell,
		\\
		|\partial^2_a\kappa| &\lesssim \ell \kappa^5 a^2\ell = \dfrac{(a^2\ell)^2}{a^2} \jabr{a\sqrt{\ell}}^{-5} \leq \dfrac{1}{a^2}\jabr{a\sqrt{\ell}}^{-1}.
 	\end{align*}
 	This yields the assertion.
 	
 	For \eqref{it:lem:EstimatesSecondDerivRadFunct2} we use Lemma \ref{lem:SecondDerivRadFunct} \eqref{it:lem:SecondDerivRadFunct2}. We obtain with $ r_0<2\ell/m $, see Lemma \ref{lem:HypTraj} \eqref{it:lem:HypTraj1},
 	\begin{align*}
 		|\partial_{\theta}^2 R| = \dfrac{\ell}{a^2R^2} \left| \dfrac{1}{R} - \dfrac{m}{2\ell} \right| \lesssim \dfrac{\ell}{a^2R^2} \dfrac{1}{r_0} \lesssim \dfrac{1}{a^2R^2} \jabr{a\sqrt{\ell}} \approx \dfrac{p}{R^2} \lesssim \dfrac{1}{a^2\ell^2} \jabr{a\sqrt{\ell}}^3.
 	\end{align*}
 	For the last three inequalities we used Lemma \ref{lem:EstimatesRadFunct} \eqref{it:lem:EstimatesRadFunct1}.
 	
 	Furthermore, we have from Lemma \ref{lem:SecondDerivRadFunct} \eqref{it:lem:SecondDerivRadFunct2}, using also Lemma \ref{lem:EstimatesDerivRadFunct} \eqref{it:lem:EstimatesDerivRadFunct1} and \eqref{it:lem:EstimatesDerivRadFunct2}, in particular $ |\partial_ap|/p\approx 1/a $,
 	\begin{align*}
 		|\partial^2_{\theta a} R| &\lesssim \dfrac{p^2}{R^2}\left[ \dfrac{1-\kappa^2}{1-\kappa}+\kappa \right] \left[ \dfrac{1}{a}\dfrac{|\theta|}{p} + |\partial_{a}\kappa|\ln\jabr{\dfrac{R}{p}} \right] + \dfrac{p}{R}\jabr{\dfrac{p}{R}}^{1/2}|\partial_a\kappa|
 		\\
 		&\lesssim \dfrac{1}{a}\dfrac{p}{R}\dfrac{|\theta|}{R} + \dfrac{1}{a}\dfrac{p^2}{R^2}\min\{1, a^2\ell\} \ln\jabr{\dfrac{R}{p}} + \dfrac{1}{a}\dfrac{p}{R} \jabr{\dfrac{p}{R}}^{1/2}\min\{1, a^2\ell\}
 	\end{align*}
 	We now use the fact that $ x\ln\jabr{1/x}\lesssim \jabr{x} $ to get
 	\begin{align*}
 		|\partial^2_{\theta a} R| &\lesssim \dfrac{1}{a}\dfrac{p}{R} \left[ 1+\jabr{\dfrac{p}{R}}\min\{1, a^2\ell\} \right] \lesssim \dfrac{1}{a}\dfrac{p}{R} \left[ 1+\jabr{\dfrac{1}{a^2\ell}}\min\{1, a^2\ell\} \right] 
 		\\
 		&\lesssim \dfrac{1}{a}\dfrac{p}{R} \lesssim \dfrac{1}{a}\jabr{\dfrac{1}{a\sqrt{\ell}}}^2.
 	\end{align*}
 	Here, we also used $ |\theta|/R\lesssim1 $ and $ p/R\leq 1/(1-\kappa)\approx \jabr{1/a^2\ell} $ according to Lemma \ref{lem:EstimatesRadFunct} \eqref{it:lem:EstimatesRadFunct2}.
 	
 	Finally, for $ \partial^2_aR $ we use Lemma \ref{lem:SecondDerivRadFunct} \eqref{it:lem:SecondDerivRadFunct2} and \eqref{eq:Sec2:ProofEstDerivRadFunct}
 	\begin{align*}
 		|\partial^2_aR| &\lesssim |\partial_a^2p| \, \ln\jabr{\dfrac{R}{p}} + \left[ \dfrac{p(\partial_a\kappa)^2}{1-\kappa} + |\partial_ap||\partial_a\kappa| + p|\partial^2_{a}\kappa| \right] \ln\jabr{\dfrac{R}{p}} \jabr{\dfrac{p}{R}}^{1/2}
 		\\
 		&\quad + \dfrac{1}{a}\dfrac{p}{R} \left[ |\theta| \dfrac{|\partial_ap|}{p} +p|\partial_a\kappa| \ln\jabr{\dfrac{R}{p}} \right] + |\partial_ap||\partial_a\kappa| \dfrac{|\theta|}{R} \jabr{\dfrac{p}{R}}^{1/2} + \dfrac{1}{a^4}.
 	\end{align*}
 	Here, we used the estimate for $ \partial_{a \theta}^2R $ and \ref{lem:EstimatesDerivRadFunct} \eqref{it:lem:EstimatesDerivRadFunct2}. We now use \ref{lem:EstimatesDerivRadFunct} \eqref{it:lem:EstimatesDerivRadFunct1} and the estimates in \eqref{it:lem:EstimatesSecondDerivRadFunct1}, in particular $ |\partial_a^2p|\approx p/a^2 $ by Lemma \ref{lem:EstimatesRadFunct} \eqref{it:lem:EstimatesRadFunct1}, to get
 	\begin{align*}
 		|\partial^2_aR| &\lesssim \dfrac{p}{a^2} \, \ln\jabr{\dfrac{R}{p}} + \dfrac{p}{a^2} \min\{1, a^2\ell\} \ln\jabr{\dfrac{R}{p}} \jabr{\dfrac{p}{R}}^{1/2} + \dfrac{p}{a^2}\left[ 1+ \ln\jabr{\dfrac{R}{p}} \right] + \dfrac{p}{a^2}+ \dfrac{1}{a^4}
 		\\
 		&\lesssim \dfrac{p}{a^2} \, \ln\jabr{\dfrac{R}{p}}.
 	\end{align*}
 	Here, we made use of $ |\partial_a\kappa|\jabr{p/R}\lesssim 1/a $. This concludes the proof.
\end{proof}

In the following lemma we show that the radial function $ R(\theta,a,\ell) $ is decreasing in $ a $ for any fixed value of $ \theta,\, \ell $. Furthermore, we prove that comparable values of actions provide comparable values of $ \partial_aR $. This will be used later on to ensure that trajectories in a comparable set of actions, like $ [a/C_0,C_0a] $ for fixed $ a>0 $ and $ C_0\geq1 $, remain comparable.

\begin{lem}\label{lem:BoundActionDerivR}
	The function $ a\mapsto R(\theta,a,\ell) $ is strictly decreasing for all $ \theta\in \R\ $, $ \ell>0 $. Moreover, for any $ C_0\geq1 $ there is a constant $ C=C(C_0)\geq 1 $, such that for all $ (\theta,a,\ell) $
	\begin{equation}\label{eq:daR-compare}
		\max_{[a/C_0,C_0a]} |\partial_{a}R(\theta,\cdot,\ell)|\leq C \min_{[a/C_0,C_0a]} |\partial_{a}R(\theta,\cdot,\ell)|.
	\end{equation}
\end{lem}
\begin{proof}
	We prove the result in several steps.
	
	\textit{Step 1.} We first show that for given $ \ell>0 $ and $ \theta\in \R $, the map $ a\mapsto R(\theta,a,\ell) $ is strictly decreasing, by giving a convenient formula for $ \partial_{a}R $. We can assume without loss of generality that $ \theta\geq 0 $. Using the abbreviation $ h=h(\theta,a,\ell) = R/p+\kappa $, by Lemma \ref{lem:DerivRadFunct} \eqref{it:lem:DerivRadFunct2} we have
	\begin{align*}
		\partial_{a}R = \partial_{a}p\left( h - \theta \partial_{\theta}h \right) + p^2\partial_{a}\kappa \, \partial_{\theta}h \, \arcosh h + \dfrac{m}{a^3},
	\end{align*}
    and
	\begin{align*}
		\partial_{\theta}h =\frac{\partial_{\theta}R}{p} =\dfrac{1}{R}\sqrt{h^2-1} = \dfrac{\sqrt{h^2-1}}{p(h-\kappa)}.
	\end{align*}
	In addition, by the definition of $ G_\kappa $ and $ h(\theta,a,\ell)= H_\kappa(|\theta|/p) $ (see Proposition \ref{pro:ActionAngleVariables}) we have
	\begin{align*}
		\dfrac{\theta}{p} = G_\kappa(h) = \sqrt{h^2-1}-\kappa \arcosh h .
	\end{align*}
	This yields
	\begin{align*}
		\partial_{a}R 
		&= \partial_{a}p \left( h - p\partial_{\theta}h \, \sqrt{h^2-1}\right)  + \partial_{\theta}h \, \left( p\kappa\partial_{a}p + p^2\partial_{a}\kappa \right) \, \arcosh h+ \dfrac{m}{a^3}
		\\
		&= \partial_{a}p \left( h - \dfrac{h^2-1}{h-\kappa}\right) + \dfrac{\sqrt{h^2-1}}{h-\kappa} \,\partial_a(\kappa p) \,  \arcosh h+ \dfrac{m}{a^3}.
	\end{align*}
    With $\kappa p=\frac12ma^{-2}$ and Lemma \ref{lem:DerivRadFunct} \eqref{it:lem:DerivRadFunct1} it follows that
	\begin{align*}
		\partial_{a}R &= -\dfrac{m(1+\kappa^2)}{2a^3\kappa} \dfrac{1}{h-\kappa}\left( -\kappa h+1 \right) - \dfrac{m}{a^3} \dfrac{\sqrt{h^2-1}}{h-\kappa} \, \arcosh h+ \dfrac{m}{a^3}
		\\
		&=-\dfrac{m}{2a^3\kappa}\dfrac{1}{h-\kappa} \left\lbrace (1+\kappa^2)(1-\kappa h) + 2\kappa\sqrt{h^2-1}\arcosh h -2\kappa(h-\kappa) \right\rbrace
		\\
		&=-\dfrac{m}{2a^3\kappa}\dfrac{1}{h-\kappa} \left\lbrace -(3\kappa+\kappa^3) h + 2\kappa\sqrt{h^2-1}\arcosh h + 1+3\kappa^2 \right\rbrace
	\end{align*}
	Note that $ h(\theta=0,a,\ell)=1 $ so that 
	\begin{align}\label{eq:Sec2ProofActDerivFormula}
		\begin{split}
			\partial_{a}R &= -\dfrac{m}{2a^3\kappa}\dfrac{1}{h-\kappa}  \left\lbrace -\kappa(3+\kappa^2) (h-1) + 2\kappa\sqrt{h^2-1}\arcosh h + (1-\kappa)^3 \right\rbrace
			\\
			&= -\dfrac{p}{a}\dfrac{1}{h-\kappa}\left\lbrace (1-\kappa)^3 +  \kappa\left( \sqrt{h^2-1}\arcosh h - \dfrac{3+\kappa^2}{2} (h-1) \right) \right\rbrace
			\\
			&= -\dfrac{p}{a}\dfrac{(1-\kappa)^3}{h-\kappa} -\dfrac{p\kappa}{a}\dfrac{h-1}{h-\kappa} \left( 2 - \dfrac{3+\kappa^2}{2} +\sqrt{\dfrac{h+1}{h-1}}\arcosh h - 2\right)
			\\
			&= -\dfrac{p}{a}\dfrac{(1-\kappa)^3}{h-\kappa}-\dfrac{p\kappa(1-\kappa^2)}{a}\dfrac{h-1}{h-\kappa} -\dfrac{p\kappa}{a}\dfrac{h-1}{h-\kappa} \eta(h), \qquad \eta(h):= \sqrt{\dfrac{h+1}{h-1}}\arcosh h - 2.
		\end{split}
	\end{align}
	Observe that $ \eta(h(0,a,\ell))=0 $ and $ \partial_{\theta}\eta(h(\theta,a,\ell)) \geq0 $ for $ \theta\geq0 $. Hence we have $ \eta(h(\theta,a,\ell))\geq 0 $ and $ \partial_{a}R<0 $ for all $ (\theta,a,\ell)\in \R\times\R_+\times \R_+ $.
	
	\textit{Step 2.} In the following we prove \eqref{eq:daR-compare}, i.e.\ for given $C_0\geq 1$, $a>0$ and arbitrary $\theta,\ell$ we show that for any $a_1,a_2\in [a/C_0,C_0a]$
	\begin{align}\label{eq:Sec2ProofActDerivEquiv}
		\dfrac{\partial_{a}R(\theta,a_1,\ell)}{\partial_{a}R(\theta,a_2,\ell)}\approx1,
	\end{align}
	with bounds depending only on $ C_0 $. For readability we abbreviate $ p_1,\, p_2 $ as well as $ \kappa_1,\, \kappa_2 $ and $ h_1 $ and $ h_2 $ when evaluating the actions at $ a_1 $ and $ a_2 $, respectively. For the rest of the proof we shall use $\approx$ to denote bounds depending only on $C_0$ and universal constants.

    For $a_1,a_2$ we have by assumption $ a_1/a_2\approx 1 $, and hence with Lemma \ref{lem:EstimatesRadFunct} \eqref{it:lem:EstimatesRadFunct1}
	\begin{align*}
		\dfrac{p_1}{p_2}\approx 1, \quad \dfrac{\kappa_1}{\kappa_2}\approx 1,\quad \dfrac{1-\kappa_1}{1-\kappa_2}\approx 1.
	\end{align*}
In particular, for $ \theta=0 $ the estimate \eqref{eq:Sec2ProofActDerivEquiv} follows from \eqref{eq:Sec2ProofActDerivFormula} using the fact that $ h=1 $. We may thus assume that $\theta>0$.  As a consequence of Lemma \ref{lem:HFunctEquiv} we have
	\begin{align*}
		\dfrac{h_1-1}{h_2-1} = \dfrac{H_{\kappa_1}(|\theta|/p_1)-1}{H_{\kappa_2}(|\theta|/p_2)-1} \approx 1.
	\end{align*}
This also implies that
\begin{align*}
    \dfrac{h_1-\kappa_1}{h_2-\kappa_2} = \dfrac{h_1-1+1-\kappa_1}{h_2-1+1-\kappa_2} \approx 1,
\end{align*}
since $h_i-1>0$ for $\theta>0$, $1-\kappa_i>0$, $i=1,2$.

We now show that also $ \eta(h_1)\approx \eta(h_2) $. Let us assume that $ h_1\leq h_2 $ and let $ c=c(C_0)\geq1 $ such that $ 1/c\leq (h_1-1)/(h_2-1) \leq c $. We first observe that for $ h\leq 10 $ we have $ \eta(h) \approx h-1 $. In particular, we infer for $ h_2\leq 10 $
	\begin{align*}
		\eta(h_2) \approx h_2-1 \approx h_1-1 \approx \eta(h_1).
	\end{align*}    
On the other hand, for $ h\geq 10/c $ we have $ \eta(h) \approx \ln (h-1) $. Hence, for $ h_2\geq 10 $ we infer $ h_1\geq (h_2-1)/c+1\geq 10/c $ and obtain since $h_1-1\geq 9/c$
	\begin{align*}
		\eta(h_2) \approx \ln(h_2-1)\approx \ln(h_1-1) \approx \eta(h_1).
	\end{align*}
The above estimates together with \eqref{eq:Sec2ProofActDerivFormula} yield \eqref{eq:Sec2ProofActDerivEquiv}.
\end{proof}

\paragraph{\bf Estimates on dynamical variable.} We now provide estimates on $ \tilde{R} $. We split them in two cases, depending on whether $ (t,\theta,a,\ell) $ is in the bulk $ \mathcal{B} $ or not, see \eqref{eq:Sec2:DefBulk}.

\begin{lem}[Estimates in the bulk]\label{lem:EstDynRadBulk}
	For all $ (t,\theta,a,\ell) \in\mathcal{B}  $ we have
	\begin{align*}
		at \lesssim \tilde{R} \lesssim p+at.
	\end{align*}
	as well as
	\begin{align*}
		\left| \partial_\theta \tilde{R} - 1 \right|   &\lesssim \dfrac{p}{at}, \quad
		|\partial_a\tilde{R}-t| \lesssim \jabr{\dfrac{p}{a}} \ln\jabr{t},
		\\
		\left| \partial_{\theta}^2\tilde{R} \right| &\lesssim \dfrac{p}{a^2t^2}, \quad
		\left| \partial_{a \theta}^2\tilde{R} \right| \lesssim \dfrac{p}{a^2t}, \quad
		\left| \partial_{a}^2\tilde{R} \right| \lesssim \dfrac{1}{a}\jabr{\dfrac{p}{a}}\ln\jabr{t}.
	\end{align*}
\end{lem}
\begin{proof}
	The first estimate follows from $ |\theta+at|\geq at/2 $ and Lemma \ref{lem:EstimatesRadFunct} \eqref{it:lem:EstimatesRadFunct2}. We observe that as a consequence
	\begin{align*}
		\jabr{\dfrac{\tilde{R}}{p}} \approx \jabr{\dfrac{at}{p}}.
	\end{align*}
	Concerning the first derivatives we make use of
	\begin{align}\label{eq:Sec2:DerivDynRadFct}
		\partial_\theta \tilde{R} = \partial_\theta R(\theta+ta,a), \quad \partial_a \tilde{R} = t \partial_\theta R(\theta+at,a) + \partial_a R(\theta+at,a).
	\end{align}
	We hence obtain via Lemma \ref{lem:EstimatesDerivRadFunct} \eqref{it:lem:EstimatesDerivRadFunct2}
	\begin{align*}
		\left| \partial_\theta \tilde{R} - 1 \right|   \lesssim \dfrac{p}{at}.
	\end{align*}
	Here, we also used the fact that $ \sgn(\theta+at)=1 $.
	
	Furthermore, we get with the previous estimate and Lemma \ref{lem:EstimatesDerivRadFunct} \eqref{it:lem:EstimatesDerivRadFunct2}
	\begin{align*}
		|\partial_a\tilde{R}-t| \lesssim t\left| \partial_\theta \tilde{R} - 1 \right| + | \partial_a R(\theta+at,a)| \lesssim \dfrac{p}{a} \ln\jabr{\dfrac{at}{p}} \lesssim \dfrac{p}{a} \left( \ln\jabr{\dfrac{a}{p}}+\ln\jabr{t} \right)  \lesssim  \jabr{\dfrac{p}{a}}\ln\jabr{t}.
	\end{align*}
	Here, we used the inequality $ x\ln \jabr{1/x}\lesssim 1+x\approx\jabr{x} $ for all $ x>0 $. 
	
	For the second derivatives we use the formulas
	\begin{align}\label{eq:Sec2:SecDerivDynRadFct}
		\begin{split}
			\partial_{\theta}^2\tilde{R} &= \partial^2_\theta R(\theta+ta,a), \quad \partial_{a \theta}^2\tilde{R} = t \partial_{\theta}R(\theta+ta,a)+ \partial_{a \theta} R(\theta+ta,a),
			\\
			\partial_{a}^2\tilde{R} &= t^2 \partial^2_\theta R(\theta+ta,a) + 2t \partial_{a \theta} R(\theta+ta,a)+ \partial_{a}^2 R(\theta+ta,a).
		\end{split}
	\end{align}
	With Lemma \ref{lem:EstimatesSecondDerivRadFunct} we then have similar to above
	\begin{align*}
		\left| \partial_{\theta}^2\tilde{R}\right| \lesssim \dfrac{p}{(at)^2}, \quad \left| \partial_{a\theta}^2\tilde{R}\right| \lesssim \dfrac{p}{a^2t}.
	\end{align*}
	Finally, we also have
	\begin{align*}
		\left| \partial_{a}^2\tilde{R}\right| \lesssim \dfrac{p}{a^2} + \dfrac{p}{a^2}\ln\jabr{\dfrac{at}{p}} \lesssim  \dfrac{p}{a^2}\left( \ln\jabr{\dfrac{a}{p}}+\ln\jabr{t} \right) \lesssim \dfrac{1}{a}\jabr{\dfrac{p}{a}}\ln\jabr{t}.
	\end{align*}
	This concludes the proof.
\end{proof}

\begin{lem}[Estimates outside the bulk]\label{lem:EstDynRadNoBulk}
	We have for all $ (t,\theta,a,\ell)\in \R_+\times\R\times \R_+\times\R_+ $
	\begin{align*}
		|\partial_\theta \tilde{R}| &\lesssim \jabr{\dfrac{p}{\tilde{R}}}^{1/2} \lesssim \jabr{\dfrac{1}{a\sqrt{\ell}}},
		\\
		|\partial_a \tilde{R}| &\lesssim t\jabr{\dfrac{p}{\tilde{R}}}^{1/2}+ \dfrac{p}{a}\ln\jabr{\dfrac{\tilde{R}}{p}} \lesssim  t\jabr{\dfrac{1}{a\sqrt{\ell}}}  + \jabr{\dfrac{1}{a^3}}\jabr{a\sqrt{\ell}}\ln\jabr{\dfrac{|\theta|}{a}},
		\\
		|\partial_\theta^2 \tilde{R}| &\lesssim \dfrac{p}{\tilde{R}^2}\lesssim \dfrac{1}{a^2\ell^2}\jabr{a\sqrt{\ell}}^3,
		\\
		|\partial_{a\theta}^2 \tilde{R}| & \lesssim \dfrac{tp}{\tilde{R}^2}+\dfrac{p}{a\tilde{R}} \lesssim \dfrac{1}{a\ell}\jabr{a\sqrt{\ell}}+\dfrac{|\theta|}{a\ell^2}\jabr{a\sqrt{\ell}}^2,
		\\
		 |\partial_{a}^2 \tilde{R}| & \lesssim \dfrac{pt^2}{\tilde{R}^2} +\dfrac{tp}{a\tilde{R}} +\dfrac{p}{a^2}\ln\jabr{\dfrac{\tilde{R}}{p}} \lesssim \jabr{\theta}^2\jabr{a\sqrt{\ell}}^3\jabr{\dfrac{1}{a^4}}\ln\jabr{t}.
	\end{align*}
\end{lem}
\begin{proof}
	As for Lemma \ref{lem:EstDynRadBulk} we use \eqref{eq:Sec2:DerivDynRadFct}. This yields with Lemma \ref{lem:EstimatesDerivRadFunct} \eqref{it:lem:EstimatesDerivRadFunct2}
	\begin{align*}
		|\partial_\theta \tilde{R}| &= |\partial_\theta R(\theta+ta,a)| \lesssim \jabr{\dfrac{p}{\tilde{R}}}^{1/2} \lesssim \jabr{\dfrac{1}{a\sqrt{\ell}}},
		\\
		|\partial_a \tilde{R}| &\lesssim t\jabr{\dfrac{p}{\tilde{R}}}^{1/2} + \dfrac{p}{a} \ln\jabr{\dfrac{\tilde{R}}{p}}.
	\end{align*}
	Furthermore we have with Lemma \ref{lem:EstimatesRadFunct} \eqref{it:lem:EstimatesRadFunct1} and \eqref{it:lem:EstimatesRadFunct2}
	\begin{align*}
		\dfrac{p}{a} \ln\jabr{\dfrac{\tilde{R}}{p}} &\lesssim \dfrac{p}{a}\ln\jabr{\dfrac{|\theta+ta|}{p}} \lesssim \dfrac{p}{a}\left[ \ln\jabr{\dfrac{|\theta|}{p}} + \ln\jabr{\dfrac{at}{p}}\right] 
		\\
		&\lesssim t\jabr{\dfrac{1}{a\sqrt{\ell}}}+ \dfrac{p}{a}\left[ \ln\jabr{\dfrac{|\theta|}{a}} +\ln\jabr{\dfrac{a}{p}} + 1+\dfrac{at}{p} \right] 
		\\
		&\lesssim \jabr{\dfrac{p}{a}}\ln\jabr{\dfrac{|\theta|}{a}} + t.
	\end{align*} 
	We then have with Lemma \ref{lem:EstimatesDerivRadFunct} \eqref{it:lem:EstimatesDerivRadFunct2} and the previous estimate
	\begin{align*}
		|\partial_a \tilde{R}| &\lesssim t\jabr{\dfrac{1}{a\sqrt{\ell}}}+ \dfrac{p}{a} \ln\jabr{\dfrac{\tilde{R}}{p}} \lesssim  t\jabr{\dfrac{1}{a\sqrt{\ell}}}  + \jabr{\dfrac{p}{a}}\ln\jabr{\dfrac{|\theta|}{a}}
		\\
		&\lesssim t\jabr{\dfrac{1}{a\sqrt{\ell}}}  + \jabr{\dfrac{1}{a^3}}\jabr{a\sqrt{\ell}}\ln\jabr{\dfrac{|\theta|}{a}}.
	\end{align*}	
	Concerning the second order derivatives we use \eqref{eq:Sec2:SecDerivDynRadFct} and Lemma \ref{lem:EstimatesSecondDerivRadFunct}. This yields
	\begin{align*}
		|\partial_\theta^2 \tilde{R}| &\lesssim \dfrac{p}{\tilde{R}^2} \lesssim \dfrac{1}{a^2\ell^2}\jabr{a\sqrt{\ell}}^3,
		\\
		|\partial_{a\theta}^2 \tilde{R}| & \lesssim \dfrac{tp}{\tilde{R}^2}+\dfrac{p}{a\tilde{R}}. 
	\end{align*}
	To further bound the last term we make use of
	\begin{align}\label{eq:Sec2:ProofSecDerivDynRad}
		\dfrac{t}{\tilde{R}}\leq \dfrac{|\theta+at|}{a\tilde{R}} +\dfrac{|\theta|}{\tilde{R}} \lesssim \dfrac{1}{a} + \dfrac{|\theta|}{r_0(a,\ell)}\lesssim \dfrac{1}{a} + \dfrac{|\theta|}{\ell}\jabr{a\sqrt{\ell}}.
	\end{align}
	This yields,
	\begin{align*}
		|\partial_{a\theta}^2 \tilde{R}| \lesssim \dfrac{1}{1-\kappa}\dfrac{t}{\tilde{R}}+\dfrac{1}{a(1-\kappa)}\lesssim \dfrac{1}{a\ell}\jabr{a\sqrt{\ell}}+\dfrac{|\theta|}{a\ell^2}\jabr{a\sqrt{\ell}}^2.
	\end{align*}
	Finally, we also have with Lemma \ref{lem:EstimatesRadFunct} \eqref{it:lem:EstimatesRadFunct1}
	\begin{align*}
		|\partial_{a}^2 \tilde{R}| & \lesssim \dfrac{t^2p}{\tilde{R}^2}+\dfrac{tp}{a\tilde{R}}+\dfrac{p}{a^2}\ln\jabr{\dfrac{\tilde{R}}{p}} \lesssim  \dfrac{t^2p}{\tilde{R}^2}+\dfrac{tp}{a\tilde{R}}+\dfrac{1}{a}\dfrac{p}{a}\ln\jabr{\dfrac{|\theta+at|}{p}}
		\\
		&\lesssim \dfrac{t^2p}{\tilde{R}^2}+\dfrac{tp}{a\tilde{R}} + \dfrac{1}{a}\jabr{\dfrac{p}{a}}\left( \ln\jabr{\theta}+\ln\jabr{t} \right) \lesssim \dfrac{t^2p}{\tilde{R}^2} + \dfrac{1}{a}\jabr{\dfrac{p}{a}}\left( \ln\jabr{\theta}+\ln\jabr{t} \right) 
		\\
		&\lesssim p\left( \dfrac{1}{a} + \dfrac{|\theta|}{\ell}\jabr{a\sqrt{\ell}}\right) ^2+ \dfrac{1}{a}\jabr{\dfrac{p}{a}}\left( \ln\jabr{\theta}+\ln\jabr{t} \right) 
		\\
		&\lesssim \dfrac{|\theta|^2}{a^2\ell^2}\jabr{a\sqrt{\ell}}^3 + \dfrac{1}{a}\jabr{\dfrac{1}{a^3}}\jabr{a\sqrt{\ell}} \left( \ln\jabr{\theta}+\ln\jabr{t} \right)
		\\
		&\lesssim \jabr{\theta}^2\jabr{a\sqrt{\ell}}^3\jabr{\dfrac{1}{a^4}}\ln\jabr{t}.
	\end{align*}
	where we also use the estimate in \eqref{eq:Sec2:ProofSecDerivDynRad}. This concludes the proof.
\end{proof}


\section{Study of gravitational field}\label{sec:GravField}
In this section we analyze the gravitational field of a distribution along the linearized flow, i.e.\ (see \eqref{eq:Sec1:GravField} resp.\ \eqref{eq:Sec1:VPinActionAngle})
\begin{align}\label{eq:Sec3:DynGravField}
	\F(t,r) = -\dfrac{1}{r^2} \int_{\R}\int_0^\infty\int_0^\infty \ind_{\{R(\theta+ta,a)\leq r\}} \gamma(t,\theta,a,\ell)^2 \, d\theta da d\ell.
\end{align}
Motivated by the asymptotics $ R(\theta+ta,a)= at + o(t) $ as $ t\to \infty $ (see Section \ref{sec:LinearizedDyn}), we also consider the corresponding ``effective gravitational field''
\begin{align}\label{eq:Sec3:EffGravField}
	\F_{\eff}(t,r) := -\dfrac{1}{r^2} \int_{\R}\int_0^\infty\int_0^\infty \ind_{\{at\leq r\}} \gamma(t,\theta,a,\ell)^2 \, d\theta da d\ell.
\end{align}
In Section \ref{subsec:EstGravFieldLInf} we provide estimates on $ \F(t,r),\,  \partial_r \F(t,r) $ using $ L^2\cap L^\infty $-norms with weights. In Section \ref{subsec:EstEffGravField} we prove similar estimates for $ \F_{\eff}(t,r),\,  \partial_r \F_{\eff}(t,r) $. In addition, we show that indeed $ \F_{\eff}(t,r) $ provides a good approximation of the asymptotics of $ \F(t,r) $ when $ t\to \infty $. Finally, in Section \ref{subsec:EstEffGravFieldAA} we give estimates of the gravitational field $ \nabla_{(\theta,a)}\Psi(t,\tilde{R}) $ in action-angle variables.

\subsection{Estimates of the gravitational field}\label{subsec:EstGravFieldLInf}
Here we give estimates for $ \F(t,r),\,  \partial_r \F(t,r) $ for distributions $ \gamma\in L^2\cap L^\infty $ including appropriate moments.

\begin{lem}\label{lem:EstimatesGravFieldLInf}
	We have for any $ q\geq0 $
	\begin{align*}
		|\F(t,r)| \lesssim& \dfrac{1}{1+r^2+t^2}\min\left( \dfrac{r^q}{(1+t)^q},1 \right) M_q(\gamma),
		\\
		M_q(\gamma):=& \norm[L^2]{\left(a+a^{-1}\right)^{2+q}\gamma}^2 +  \norm[L^\infty]{\jabr{a}\ell^{-q/2}\gamma}^2 + \dfrac{1}{1+t} \norm[L^\infty]{\jabr{a}a^{-4-2q}\jabr{\ell}\ell^{-q/2}\jabr{\theta}^{3+q/2}\gamma}^2.
	\end{align*}
\end{lem}
\begin{proof}
	We split the proof into three cases, according to the (relative) size of $t,r$. To begin, notice that by Lemma \ref{lem:EstimatesRadFunct} \eqref{it:lem:EstimatesRadFunct1} we have
	\begin{align*}
		r_0(a,\ell)\approx \ell\jabr{a\sqrt{\ell}}^{-1} \approx \ell\ind_{\{a\sqrt{\ell}\geq1\}} \dfrac{1}{a\sqrt{\ell}} + \ell\ind_{\{a\sqrt{\ell}\leq 1\}} \geq a^{-2}\ind_{\{a\sqrt{\ell}\geq1\}}+ \ell\ind_{\{a\sqrt{\ell}\leq 1\}} = \min\left(a^{-2},\ell \right),
	\end{align*}
	and hence with Lemma \ref{lem:EstimatesRadFunct} \eqref{it:lem:EstimatesRadFunct2}
	\begin{align}\label{eq:Sec3:LowerBoundRadVar}
		\tilde{R}\gtrsim |\theta+at| + r_0 \gtrsim |\theta+at| + \min\left(a^{-2},\ell \right). 
	\end{align}

	\textit{Case 1: $ t,r\leq 1 $.} We prove $ |\mathcal{F}(t,r)|\lesssim r^q $. We hence obtain by splitting into $ a^{-2}\leq \ell $ as well as $ a^{-2}\geq \ell $ and using \eqref{eq:Sec3:LowerBoundRadVar} that
	\begin{align*}
		|\F(t,r)| &=  \dfrac{1}{r^2}\int_{\R}\int_0^\infty\int_0^\infty \ind_{\{\tilde{R}\leq r\}} \gamma^2 \, d\theta da d\ell 
		\\
		&\leq \dfrac{1}{r^2}\int_{\R}\int_0^\infty\int_0^\infty \left( \ind_{\{ |\theta+at| + a^{-2}\leq r\}}+ \ind_{\{|\theta+at| + \ell\leq r\}} \right) \gamma^2 \, d\theta da d\ell
		\\
		&\leq r^q \norm[L^2]{a^{2+q}\gamma}^2 +  \dfrac{r^q}{r^2}\int_{\R}\int_0^\infty\int_0^\infty \ind_{\{|\theta+at| + \ell\leq r\}} \ell^{-q}\gamma^2 \, d\theta da d\ell
		\\
		&\lesssim r^q\left( \norm[L^2]{a^{2+q}\gamma}^2 + \norm[L^\infty]{\jabr{a}\ell^{-q/2}\gamma}^2 \right).
	\end{align*}
	Here, we used in the last inequality that
	\begin{align*}
		\dfrac{1}{r^2}\int_{\R}\int_0^\infty\int_0^\infty \ind_{\{|\theta+at| + \ell\leq r\}} \jabr{a}^{-2} \, d\theta da d\ell = \dfrac{1}{r^2}\int_{\R}\int_0^\infty\int_0^\infty \ind_{\{|\bar{\theta}| + \ell\leq r\}} \jabr{a}^{-2} \, d\bar{\theta} da d\ell \lesssim 1.
	\end{align*}
	
	\textit{Case 2: $ t\geq r $, $t\geq 1$.} We will prove $ |\F(t,r)|\lesssim r^q/t^{2+q} $. We split the integral into the part containing the bulk and its complement
	\begin{align*}
		|\F(t,r)| = \dfrac{1}{r^2}\int_{\R}\int_0^\infty\int_0^\infty \ind_{\{\tilde{R}\leq r\}\cap \mathcal{B}} \gamma^2 \, d\theta da d\ell +\dfrac{1}{r^2}\int_{\R}\int_0^\infty\int_0^\infty \ind_{\{\tilde{R}\leq r\}\cap \mathcal{B}^c} \gamma^2 \, d\theta da d\ell =:I_1+I_2.
	\end{align*}
	We have by the definition of $ \mathcal{B} $, see \eqref{eq:Sec2:DefBulk}, and Lemma \ref{lem:EstDynRadBulk}
	\begin{align*}
		I_1(t,r) \leq \dfrac{1}{r^2}\int_{\R}\int_0^\infty\int_0^\infty \ind_{\{at\leq 2r\}} \gamma^2 \, d\theta da d\ell\lesssim \dfrac{r^q}{t^{2+q}}\norm[L^2]{a^{-1-q/2}\gamma}^2.
	\end{align*}
	Note that outside the bulk we have for any $ k\geq0 $ that
	\begin{align*}
		\ind_{\mathcal{B}^c} \lesssim \dfrac{|\theta|^k}{(at)^k}\, \ind_{\mathcal{B}^c},
	\end{align*}
	so together with \eqref{eq:Sec3:LowerBoundRadVar} we obtain that
	\begin{align*}
		I_2(t,r) \lesssim \dfrac{r^q}{t^{3+q}}\dfrac{1}{r^{2+q}} \int_{\R}\int_0^\infty\int_0^\infty \ind_{\{\min\{|\theta+at| + a^{-2},\ell\}\leq r\}} \left( \dfrac{|\theta|}{a} \right)^{3+q}  \gamma^2 \, d\theta da d\ell.
	\end{align*}
	We again distinguish the cases $ a^{-2}\leq \ell $ and $ a^{-2}\geq \ell $, yielding
	\begin{align*}
		I_2(t,r) &\lesssim \dfrac{r^q}{t^{3+q}}\left( \norm[L^2]{a^{-4-2q}\jabr{\theta}^{(3+q)/2}\gamma}^2 + \norm[L^\infty]{\jabr{a}a^{-3/2-q/2}\jabr{\theta}^{3/2+q}\ell^{-q/2}\gamma}^2\right) 
		\\
		&\lesssim  \dfrac{r^q}{t^{2+q}}\dfrac{1}{1+t}\norm[L^\infty]{\jabr{a}a^{-4-2q}\jabr{\ell}\ell^{-q/2}\jabr{\theta}^{3+q/2}\gamma}^2,
	\end{align*}
	by a similar estimate as in Case 1.
	
	\textit{Case 3: $ r\geq t $, $r\geq 1$.} Here we have that
	\begin{align*}
		|\F(t,r)| \leq \dfrac{1}{r^2}\norm[L^2]{\gamma}^2.
	\end{align*}
	This concludes the proof.
\end{proof}

\begin{lem}\label{lem:EstimatesDerivGravFieldLInf}
	For any $ q\geq0 $ we have the estimate
	\begin{align*}
		|\partial_r \F(t,r)| \lesssim& \dfrac{1}{1+r^2+t^2}\dfrac{1}{r}\min\left( \dfrac{r^q}{(1+t)^q},1 \right) M'_q(\gamma),
		\\
		M'_q(\gamma) :=&\norm[L^\infty]{\jabr{\ell}^{8+4q} \ell^{-q/2} (a+a^{-1})^{8+4q} \jabr{\theta}^{8+4q} \gamma(t)}^2.
	\end{align*}
\end{lem}
\begin{proof}
	By \eqref{eq:Sec3:DynGravField} we have
	\begin{align*}
		\partial_r \F(t,r) &= \dfrac{2}{r^3} \int_{\R}\int_0^\infty\int_0^\infty \ind_{\{R(\theta+ta,a)\leq r\}} \gamma^2 \, d\theta da d\ell - \dfrac{1}{r^2}\partial_r\int_{\R}\int_0^\infty\int_0^\infty \ind_{\{R(\theta+ta,a)\leq r\}} \gamma^2 \, d\theta da d\ell
		\\
		&=:I_1+I_2.
	\end{align*}
Note that by Lemma \ref{lem:EstimatesGravFieldLInf} we get
\begin{align*}
	|I_1(t,r)| =|\F(t,r)| \dfrac{1}{r} \lesssim \dfrac{1}{1+r^2+t^2}\dfrac{1}{r}\min\left( \dfrac{r^q}{(1+t)^q},1 \right) M_q(\gamma).
\end{align*}
This suffices, since $ M_q(\gamma)\lesssim M_q'(\gamma) $. 

We now treat the second integral $ I_2 $. Invoking the symplectic change of variables $ (s,v)\mapsto (\Theta(s,v,\ell)-t\Act(s,v,\ell),\Act(s,v,\ell)) $ gives
\begin{align*}
	I_2(t,r)&=-\dfrac{1}{r^2}\partial_r \int_0^\infty \int_{\R}\int_0^\infty  \ind_{\{s\leq r\}} \ind_{\{v^2-\frac{m}{s}+\frac{\ell}{s^2}>0\}} \gamma\left( \Theta(s,v,\ell)-t\Act(s,v,\ell),\Act(s,v,\ell),\ell \right)^2dsdvd\ell
	\\
	&=  -\dfrac{1}{r^2}\int_0^\infty \int_{\R} \ind_{\{v^2-\frac{m}{r}+\frac{\ell}{r^2}>0\}} \gamma\left(t, \Theta(r,v,\ell)-t\Act(r,v,\ell),\Act(r,v,\ell),\ell \right)^2d\ell dv.
\end{align*}
We then obtain that for $ k:=15+7q $
\begin{align*}
	|I_2(t,r)| &\lesssim I_3(t,r) \, \tilde{M}_q(t,r),
	\\
	I_3(t,r) &:= \dfrac{1}{r^2}\int_0^\infty \int_{\R} \ind_{\{v^2-\frac{m}{r}+\frac{\ell}{r^2}>0\}} \jabr{\ell}^{-k}\ell^q(\Act+\Act^{-1})^{-k}\jabr{\Theta-t\Act}^{-k}d\ell dv,
	\\
	\tilde{M}_q(t,r) &:= \sup_{v,\ell}\left\lbrace \ind_{\{v^2-\frac{m}{r}+\frac{\ell}{r^2}>0\}}\jabr{\ell}^{k}\ell^{-q}(\Act+\Act^{-1})^{k}\jabr{\Theta-t\Act}^{k} \gamma\left(t, \Theta-t\Act,\Act,\ell \right)^2 \right\rbrace.
\end{align*}
Note that
\begin{align*}
	\tilde{M}_q(t,r) &\leq \sup_{r,v,\ell}\left\lbrace \ind_{\{v^2-\frac{m}{r}+\frac{\ell}{r^2}>0\}}\jabr{\ell}^{k}\ell^{-q}(\Act+\Act^{-1})^{k}\jabr{\Theta-t\Act}^{k}\gamma\left(t, \Theta-t\Act,\Act,\ell \right)^2 \right\rbrace
	\\
	&= \sup_{\theta,a,\ell}\left\lbrace \ind_{\{a>0\}}\jabr{\ell}^{k} \ell^{-q} (a+a^{-1})^{k} \jabr{\theta}^{k} \gamma(t,\theta,a,\ell)^2 \right\rbrace\leq M'_q(\gamma).
\end{align*}
In the following we will prove that
\begin{align*}
	I_3(t,r)\lesssim \dfrac{1}{1+r^2+t^2}\dfrac{1}{r}\min\left( \dfrac{r^q}{(1+t)^q},1 \right).
\end{align*}
To this end, we split the integral into two terms $ I^{+}_3(t,r),\, I^{-}_3(t,r) $ according to $ v\geq0 $ and $ v\leq 0 $. For each integral we use the change of variables 
\begin{align*}
	a\mapsto v=\pm \sqrt{a^2+\dfrac{m}{r}-\dfrac{\ell}{r^2}},
\end{align*}
where by Lemma \ref{pro:ActionAngleVariables}
\begin{align*}
	\Theta(r,v,\ell) = \sgn(v) pG_\kappa\left( \dfrac{r}{p}+\kappa \right),
\end{align*}
so that
\begin{align*}
	I^{\pm}_3(t,r) = \dfrac{1}{r^2}\int_0^\infty \int_0^\infty \ind_{\{a^2+\frac{m}{r}-\frac{\ell}{r^2}>0\}} \, \dfrac{a\ell^q\jabr{\ell}^{-k}}{\sqrt{a^2+\frac{m}{r}-\frac{\ell}{r^2}}} (a+a^{-1})^{-k}\jabr{\pm pG_\kappa\left( \dfrac{r}{p}+\kappa \right)-at}^{-k}d\ell da.
\end{align*}
We now treat three cases.

\textit{Case 1: $ t,r\lesssim 1 $ (with an implicit constant that will be clear below).} We will prove that $ I_3^\pm(t,r)\lesssim r^{q-1} $. We have
\begin{align*}
		I^{\pm}_3(t,r) \leq \dfrac{1}{r^2}\int_0^\infty  \int_0^\infty \,  \ind_{\{a^2+\frac{m}{r}-\frac{\ell}{r^2}>0\}} \dfrac{a\jabr{\ell}^{-k}\ell^q(a+a^{-1})^{-k}}{\sqrt{a^2+\frac{m}{r}-\frac{\ell}{r^2}}}d\ell da,
\end{align*}
and the change of variables $ a\mapsto \xi=a^2 $, $ \ell\mapsto\eta=\ell/rm $ gives
\begin{align*}
	I^{\pm}_3(t,r) \lesssim \dfrac{r^q}{r}\int_0^\infty  \int_0^\infty \,  \ind_{\{\xi+\frac{m}{r}(1-\eta)>0\}}\dfrac{ \jabr{r\eta}^{-k} \eta^q \jabr{\xi}^{-k/2}}{\sqrt{\xi+\frac{m}{r}(1-\eta)}}d\eta d\xi.
\end{align*}
For this integral we split into the contribution $ \eta\leq 1 $ and $ \eta \geq 1 $. Since $k\geq 3$, the first contribution can be bounded by
\begin{align*}
	\dfrac{r^q}{r}\int_0^1 \int_0^\infty  \,  \ind_{\{\xi+\frac{m}{r}(1-\eta)>0\}} \dfrac{1}{\sqrt{\xi}}\jabr{\xi}^{-k/2}\lesssim \dfrac{r^q}{r}d\eta d\xi
\end{align*}
The second contribution is bounded by
\begin{align*}
	 \dfrac{r^q}{r}\int_0^\infty\int_0^\infty \,  \ind_{\{\xi+\frac{m}{r}(1-\eta)>0\}}\dfrac{ \eta^q \jabr{\xi}^{-k/2}}{\sqrt{\xi+\frac{m}{r}(1-\eta)}} d\eta d\xi&= \dfrac{r^q}{r}\int_0^\infty\int_0^\infty \,  \ind_{\{\zeta+\frac{m}{r}(\eta-1)>0\}}\dfrac{ \eta^q \jabr{\zeta+\frac{m}{r}(\eta-1)}^{-k/2}}{\sqrt{\zeta}} d\eta d\xi
	 \\
	 &\lesssim \dfrac{r^q}{r}\int_0^\infty  \int_0^\infty  \, \dfrac{ \eta^q \jabr{\zeta+\eta}^{-k/2}}{\sqrt{\zeta}}d\eta d\xi \lesssim  \dfrac{r^q}{r},
\end{align*}
having used the change of variables $ \xi\mapsto \zeta=\xi+\frac{m}{r}(1-\eta) $, as well as the facts that $ r\lesssim 1 $ in the second to last inequality, and $ k\geq 5+2q $ in the last inequality.

\textit{Case 2: $ r\leq t $ and $ t\gtrsim 1 $.} 
We will prove that $ I^{\pm}_3(t,r) \lesssim r^{q-1}t^{-2-q} $. We first consider the integral $ I_3^{-}(t,r) $:
\begin{align*}
	I_3^{-}(t,r) &\lesssim \dfrac{1}{r^2}\dfrac{1}{t^{2+q}} \int_0^\infty  \int_0^\infty \ind_{\{a^2+\frac{m}{r}-\frac{\ell}{r^2}>0\}} \, \dfrac{a\ell^q\jabr{\ell}^{-k}}{\sqrt{a^2+\frac{m}{r}-\frac{\ell}{r^2}}} (a+a^{-1})^{-k}a^{-2-q}d\ell da
	\\
	&\lesssim \dfrac{r^q}{r}\dfrac{1}{t^{2+q}} \int_0^\infty \int_0^\infty  \ind_{\{\xi+\frac{m}{r}(1-\eta)>0\}} \, \dfrac{\eta^q\jabr{r\eta}^{-k}}{\sqrt{\xi+\frac{m}{r}(1-\eta)}} \jabr{\xi}^{-k/2+4+2q}d\eta d\xi,
\end{align*}
thanks to the same change of variables as in Case 1. For $ r\lesssim1 $ we can now argue as in Case 1, making use of $ k\geq 14+6q $, whereas for $ r\geq 1 $ we use the change of variables $ \xi\mapsto\zeta = \xi+\frac{m}{r}(1-\eta) $ and $ m(\eta-1)/r\geq -m $ to obtain that
\begin{align*}
	I_3^{-}(t,r) &\lesssim\dfrac{r^q}{r}\dfrac{1}{t^{2+q}} \int_0^\infty \int_0^\infty \ind_{\{\zeta+\frac{m}{r}(\eta-1)>0\}} \, \dfrac{\eta^q\jabr{\eta}^{-k}}{\sqrt{\zeta}} \jabr{\zeta+\frac{m}{r}(\eta-1)}^{-k/2+4+2q}d\eta d\xi
	\\
	&\lesssim \dfrac{r^q}{r}\dfrac{1}{t^{2+q}} \int_0^\infty \int_0^\infty \ind_{\{\zeta+\frac{m}{r}(\eta-1)>0\}} \, \dfrac{\eta^q\jabr{\eta}^{-k}}{\sqrt{\zeta}} \jabr{\zeta}^{-k/2+4+2q} d\eta d\xi\lesssim \dfrac{r^q}{r}\dfrac{1}{t^{2+q}},
\end{align*}
since $ k\geq 12+4q $. Consider next the integral $ I^+_3(t,r) $, and let us introduce the notation
\begin{align*}
	\bar{\theta}(t,r,a,\ell) :=& pG_\kappa\left( \dfrac{r}{p}+\kappa \right)-at,
	\qquad D:= \left\lbrace t^{-1/5}\leq a\leq t^{1/5},\, \ell\leq t^{1/5},\, |\bar{\theta}|\leq \sqrt{t} \right\rbrace. 
\end{align*}
Note that the contribution of $ I^+_3(t,r) $ on the set $ D^c $ can be bounded by
\begin{align*}
	\dfrac{1}{r^2}\dfrac{1}{t^{2+q}} \int_0^\infty \int_0^\infty \ind_{\{a^2+\frac{m}{r}-\frac{\ell}{r^2}>0\}} \, \dfrac{a\ell^q\jabr{\ell}^{-k'}}{\sqrt{a^2+\frac{m}{r}-\frac{\ell}{r^2}}} (a+a^{-1})^{-k'} d\ell da,
\end{align*}
while the exponent $ k' $ is chosen as either $ k $ or $ k-10-5q $. This integral can be treated as for $ I_3^-(t,r) $, using $ k\geq 15+7q $.

We next study the contribution of the integral $ I^+_3(t,r) $ on the set $ D $. Let us recall that by Lemma \ref{lem:EstimatesG} we have
\begin{align*}
	\bar{\theta} &= r+kp-at-p\kappa \ln\left( \dfrac{r}{p}+\kappa \right)   +\mathcal{O}(p) = r-at - \kappa p \ln\jabr{r} - \kappa p\ln\jabr{p^{-1}} + \mathcal{O}(p) 
	\\
	&=r-at -\dfrac{m}{2a^2}\ln\jabr{r} + \mathcal{O}\left( a^{-2}+\ell\right),
\end{align*}
since $ \kappa p=m/2a^2 $ (see Lemma \ref{lem:HypTraj}) and $ p\approx \ell/\jabr{a\sqrt{\ell}}\lesssim a^{-2}+\ell $. The change of variables
\begin{align*}
	a\mapsto \zeta := r-at -\dfrac{m}{2a^2}\ln\jabr{r}
\end{align*}
is well-defined since
\begin{align}\label{eq:Sec3:ProofLInfChangeOfVar}
	\dfrac{d\zeta}{da} = -t + \dfrac{m}{a^3}\ln\jabr{r} \leq -t + mt^{3/5}\ln\jabr{t} \lesssim -t,\qquad t\gtrsim 1. 
\end{align}
Furthermore, on $ D $  we have that
\begin{align*}
	|\bar{\theta}| \lesssim \sqrt{t} \implies |\zeta| \lesssim \sqrt{t} + a^{-2}+\ell \lesssim \sqrt{t},
\end{align*}
which implies by definition of $ \zeta $ with $ a\geq t^{-1/5} $ that
\begin{align*}
	|r-at| \lesssim \sqrt{t} \implies r\gtrsim t^{4/5}.
\end{align*}
Hence we have with $ |\zeta|\lesssim \sqrt{t} $ and $t\gtrsim 1$ that
\begin{align*}
	a = \dfrac{r-\zeta-m\ln\jabr{r}/2a^2}{t} \approx \dfrac{r}{t}.
\end{align*}
Similarly, for some constant $ C $
\begin{align*}
	|\bar{\theta}| \geq \max\left( |\zeta|-C\left( a^{-2}+\ell \right) ,0 \right)
\end{align*}
as well as $ \ell\leq t^{1/5} \leq t^{4/5}\lesssim r $ and hence
\begin{align*}
	a^2+\dfrac{m}{r}-\dfrac{\ell}{r^2} =a^2 + \dfrac{m}{r}\left( 1- \dfrac{\ell}{mr} \right) \geq a^2.
\end{align*}
Hence, for the contribution of $ I^+_3(t,r) $ on the set $ D $ we obtain upon changing variables $ a\mapsto \zeta $ and using the bound \eqref{eq:Sec3:ProofLInfChangeOfVar} that
\begin{align*}
	&\dfrac{1}{r^2}\dfrac{1}{t} \int \int  \left( a(\zeta)+a(\zeta)^{-1} \right)^{-k}\jabr{\ell}^{-k+q}\max\left( |\zeta|-C\left( a(\zeta)^{-2}+\ell \right) ,1 \right)^{-k}d\zeta d\ell \lesssim \dfrac{1}{r^2}\dfrac{1}{t} \left( \dfrac{r}{t} \right)^{2+q} = \dfrac{r^q}{rt^{2+q}},
\end{align*}
since $ a\approx r/t $ and $ k\geq 4+q $.

\textit{Case 3: $ r\geq t $ and $ r\gtrsim1 $.} 
We will prove that $ I_3^\pm(t,r) \lesssim  1/r^3 $. First, we treat the integral $ I_3^{-}(t,r) $. Observe that in this case
\begin{align*}
	pG_\kappa\left( \dfrac{r}{p}+\kappa \right)+at = r+at -\dfrac{m}{2a^2} \ln\jabr{r} + \mathcal{O}(p),
\end{align*}
due to Lemma \ref{lem:EstimatesG} and arguments as in Case 2. Hence the contribution on the set $ \{a\geq 1/r^{1/4}, \ell\leq r^{1/4}\} $ can be bounded by
\begin{align*}
	\dfrac{1}{r^{2+q}}\int_0^\infty \int_0^\infty \,  \ind_{\{a^2+\frac{m}{r}-\frac{\ell}{r^2}>0\}} \dfrac{a\jabr{\ell}^{-k}\ell^q(a+a^{-1})^{-k}}{\sqrt{a^2+\frac{m}{r}-\frac{\ell}{r^2}}} d\ell da\lesssim \frac{r^q}{r r^{2+q}}=1/r^3,
\end{align*}
by arguments as in Cases 1 and 2. On the other hand, on the set $ \{a\geq 1/r^{1/4}, \ell\leq r^{1/4}\}^c $ we can use moments in $ a^{-1} $ and $ \ell $ to obtain the claimed bound. It thus remains treat the integral $ I^+_3(t,r) $. In analogy with Case 2 we let
\begin{align*}
	\bar{\theta}(t,r,a,\ell) :=& pG_\kappa\left( \dfrac{r}{p}+\kappa \right)-at,
	\qquad D':= \left\lbrace r^{-1/5}\leq a\leq r^{1/5},\, \ell\leq r^{1/5},\, |\bar{\theta}|\leq \sqrt{r} \right\rbrace. 
\end{align*}
On the set $ (D')^c $ we can argue using moments to yield the bound $ 1/r^3 $. On the set $ D' $ we have again by Lemma \ref{lem:EstimatesG}
\begin{align*}
	\bar{\theta} = r-at-\dfrac{m}{2a^2}\ln\jabr{r} + \mathcal{O}\left( a^{-2}+\ell\right),
\end{align*}
and in addition for $ r\gtrsim1 $
\begin{align*}
	|\bar{\theta}|\leq \sqrt{r} \implies at\geq r-\dfrac{m}{2a^2}\ln\jabr{r} -Cr^{2/5} \implies t\gtrsim r^{4/5}
\end{align*}
Furthermore, we observe that for $ \zeta(a) = r-at-\dfrac{m}{2a^2}\ln\jabr{r} $ it holds
\begin{align*}
	|\bar{\theta}| &\leq \sqrt{r} \implies |\zeta| \lesssim \sqrt{r}
	\\
	\dfrac{d\zeta}{da}&=-t + \dfrac{m}{a^3}\ln\jabr{r} \leq -t + 2mr^{3/5}\ln\jabr{r} \lesssim -t,
\end{align*}
where in the last step we used that $ t\gtrsim r^{4/5} $, $ r\gtrsim 1 $. Let us also note that we have with $ |\zeta|\lesssim \sqrt{r} $ that
\begin{align*}
	a = \dfrac{r-\zeta-m\ln\jabr{r}/2a^2}{t} \approx \dfrac{r}{t},
\end{align*}
and for some constant $ C $
\begin{align*}
	|\bar{\theta}| \geq \max\left( |\zeta|-C\left( a^{-2}+\ell \right) ,0 \right)
\end{align*}
as well as $ \ell\leq r^{1/5} \leq r $, and hence
\begin{align*}
	a^2+\dfrac{m}{r}-\dfrac{\ell}{r^2} \geq a^2 + \dfrac{m}{r}\left( 1- \dfrac{\ell}{mr} \right) \geq a^2.
\end{align*}
All in all, we obtain for the contribution of $ I_3^+(t,r) $ on $ D' $, after the change of variables $ a\mapsto \zeta $,
\begin{align*}
	&\dfrac{1}{r^2}\dfrac{1}{t} \int  \int  \left( a(\zeta)+a(\zeta)^{-1} \right)^{-k}\jabr{\ell}^{-k+q}\max\left( |\zeta|-C\left( a(\zeta)^{-2}+\ell \right) ,1 \right)^{-k} d\zeta d\ell\lesssim \dfrac{1}{r^2}\dfrac{1}{t}\dfrac{t}{r} = \dfrac{1}{r^3}.
\end{align*}
Here, we used $ a^{-1}\approx t/r $ and $ k\geq 3+q $. This concludes the proof.
\end{proof}

\subsection{Estimates on the effective gravitational field}\label{subsec:EstEffGravField}
In the following we provide bounds for both $ \F_{\eff}(t,r) $ and $ \partial_r \F_{\eff}(t,r) $ as given in \eqref{eq:Sec3:EffGravField}.

\begin{lem}\label{lem:EstimatesEffGravField}
	We have the estimate
	\begin{align*}
		|\F_{\eff}(t,r)| \lesssim&  \dfrac{1}{r^2 + t^2} \norm[L^2]{\left(1+a^{-1}\right)\gamma(t)}^2.
	\end{align*}
\end{lem}
\begin{proof}
	We have with \eqref{eq:Sec3:EffGravField} and $ t/r\leq 1/a $
	\begin{align*}
		(r^2+t^2)|\F_{\eff}(t,r)| \leq \left(1+\dfrac{t^2}{r^2}\right) \int_{\R}\int_0^\infty\int_0^\infty \ind_{\{at\leq r\}} \gamma^2 \, d\theta da d\ell \lesssim \norm[L^2]{\left(1+a^{-1}\right)\gamma}^2,
	\end{align*}
	which yields the claim.
\end{proof}
\begin{lem}\label{lem:EstimatesDerivGravEffField}
	We have the inequality
	\begin{align*}
		|\partial_r \F_{\eff}(t,r)| \lesssim \dfrac{1}{r^3+t^3}\left[ \norm[L^2]{\left(1+a^{-3/2}\right)\gamma(t)}^2+\norm[L^\infty]{\left( a+a^{-1}\right) \jabr{\theta}^{3/4}\jabr{\ell}^{3/4}\, \gamma(t)}^2 \right].
	\end{align*}
\end{lem}
\begin{proof}
	By direct computation we ahve
	\begin{align*}
		\partial_r\F_{\eff}(t,r) &= \dfrac{2}{r^3}\int_{\R}\int_0^\infty\int_0^\infty \ind_{\{at\leq r\}} \gamma^2 \, d\theta da d\ell - \dfrac{1}{tr^2}\int_{\R}\int_0^\infty \gamma^2\left(t,\theta,\dfrac{r}{t},\ell\right) \, d\theta d\ell
		\\
		&=:I_1+I_2.
	\end{align*}
	Using that $ t/r\leq 1/a $, estimating as in Lemma \ref{lem:EstimatesEffGravField} gives
	\begin{align*}
		\left(r^3+t^3\right)|I_1(t,r)| \lesssim \norm[L^2]{(1+a^{-3/2})\gamma}^2.
	\end{align*}
	Furthermore, we get
	\begin{align*}
		r^3|I_2(t,r)| = \dfrac{r}{t}\int_{\R}\int_0^\infty \gamma^2\left(t,\theta,\dfrac{r}{t},\ell\right) \, d\theta d\ell \lesssim \norm[L^\infty]{\jabr{\theta}^{3/2}\jabr{\ell}^{3/2} a \, \gamma^2 } \lesssim \norm[L^\infty]{\jabr{\theta}^{3/4}\jabr{\ell}^{3/4} a^{1/2} \, \gamma}^2,
	\end{align*}
    and similarly
	\begin{align*}
		t^3|I_2(t,r)| &= \dfrac{t^2}{r^2}\int_{\R}\int_0^\infty \gamma^2\left(t,\theta,\dfrac{r}{t},\ell\right) \, d\theta d\ell \lesssim \norm[L^\infty]{\jabr{\theta}^{3/2}\jabr{\ell}^{3/2}a^{-2}\, \gamma^2 } \lesssim \norm[L^\infty]{\jabr{\theta}^{3/4}\jabr{\ell}^{3/4}a^{-1}\, \gamma}^2.
	\end{align*}
\end{proof}
In the remainder of this section we prove that $ \F(t,r) $ can be approximated by $ \F_{\eff}(t,r) $ as $ t\to \infty $. More precisely, we have the following lemma.

\begin{lem}\label{lem:GravFieldvsEffField}
	For $ q>1 $ there holds that
	\begin{align*}
		(r^2+t^2)|\F(t,r)-\F_{\eff}(t,r)| \lesssim \dfrac{1}{\sqrt{1+t}} \norm[L^\infty]{\left( a+a^{-1}\right)^{4}\jabr{\theta}^{4}\jabr{\ell}^{4}\ell^{-q/2} \gamma(t)}^2.
	\end{align*}
\end{lem}
\begin{proof}
	We have the formula
	\begin{align*}
		|\F(t,r)-\F_{\eff}(t,r)| = \dfrac{1}{r^2}\int_{\R}\int_0^\infty\int_0^\infty \left| \ind_{\{\tilde{R}\leq r\}}-\ind_{\{at\leq r\}} \right|  \gamma^2 \, d\theta da d\ell.
	\end{align*}
	The integral is thus computed on the set
	\begin{align*}
		D:=\{\tilde{R}\leq r\leq at\}\cup \{at\leq r\leq\tilde{R}\} \subset \left\lbrace |r-at|\leq |\tilde{R}-at| \right\rbrace.
	\end{align*}
	With Lemma \ref{lem:EstimatesRadFunct} \eqref{it:lem:EstimatesRadFunct1} and \eqref{it:lem:EstimatesRadFunct2} we have
	\begin{align}\label{eq:Sec3:ProofApproxEffField}
		\begin{split}
			|\tilde{R}(\theta+at,a)-at|&\lesssim p\kappa +|\theta| + p\ln\jabr{\dfrac{\tilde{R}}{p}}\lesssim p\kappa + |\theta| + p\ln \jabr{\dfrac{|\theta+at|}{p}}
			\\
			&\lesssim p\kappa + |\theta| + p\ln\jabr{\theta+at}\lesssim p\kappa + |\theta| + p(1+|\theta+at|)^{1/8}
			\\
			&\lesssim p^2(1+a^{1/4})+|\theta| + t^{1/4}\lesssim a^{-4}+\ell^2 +|\theta| + t^{1/4}=:\zeta(t,\theta,a,\ell).
		\end{split}
	\end{align}
	We then observe
	\begin{align*}
		|\F(t,r)-\F_{\eff}(t,r)| \lesssim& I(t,r)\norm[L^\infty]{\left( a+a^{-1}\right)^{k/2}\jabr{\theta}^{k/2}\jabr{\ell}^{k/2}\ell^{-q/2}  \gamma}^2,
		\\
		I(t,r) :=& \dfrac{1}{r^2}\int_{\R}\int_0^\infty\int_0^\infty \left( a+a^{-1}\right)^{-k}\jabr{\theta}^{-k}\jabr{\ell}^{-k}\ell^{q} \ind_D\, d\theta da d\ell
	\end{align*}
	for some $ k $ to be specified below and $ q>1 $. We now prove that for $ k>7 $
	\begin{align*}
		(r^2+t^2)I(t,r)\lesssim \dfrac{1}{\sqrt{1+t}}.
	\end{align*}
	We have using either $ 1/r\leq 1/at $ or $ 1/r\leq 1/\tilde{R} $ on $ D $
	\begin{align*}
		(r^2+t^2)I(t,r) &\lesssim \int_{\R}\int_0^\infty\int_0^\infty \left(1+\dfrac{1}{a^2}+\dfrac{t^2}{\tilde{R}^2} \right)  \left( a+a^{-1}\right)^{-k}\jabr{\theta}^{-k}\jabr{\ell}^{-k}\ell^{q} \ind_D\, d\theta da d\ell.
	\end{align*}
	On the one hand, the contribution on the set $ \{|\theta|\geq at/2\}  $ can be bounded using moments in $ |\theta|a^{-1} $ and $ r_0 \gtrsim\min(a^{-2},\ell) $. At this point we use $ q>1 $ and $ k>4 $. On the other hand, on the set $ \{|\theta|\leq at/2\}  $ we can use  that $ \tilde{R}\gtrsim |\theta+at|\geq at/2 $, see Lemma \ref{lem:EstimatesRadFunct} \eqref{it:lem:EstimatesRadFunct2}. We are left with the integral
	\begin{align*}
		&\int_{\R}\int_0^\infty\int_0^\infty \left(1+\dfrac{1}{a^2}+\dfrac{4t^2}{t^2a^2} \right) \left( a+a^{-1}\right)^{-k}\jabr{\theta}^{-k}\jabr{\ell}^{-k}\ell^{q} \ind_D\, d\theta da d\ell
		\\
		&\qquad \lesssim \int_{\R}\int_0^\infty\int_0^\infty\left( a+a^{-1}\right)^{-k+2}\jabr{\theta}^{-k}\jabr{\ell}^{-k}\ind_{\{|r-at|\lesssim \zeta\}}\, d\theta da d\ell.
	\end{align*}
	We can further restrict to the set $ \{\zeta \leq \sqrt{1+t}\} $. The contribution of the complement can be treated using moments in $ |\theta|, \, a^{-1}, \, \ell $, which requires $ k>7 $. We are left with the integral 
	\begin{align*}
		\int_{\R}\int_0^\infty\int_0^\infty\left( a+a^{-1}\right)^{-k+2}\jabr{\theta}^{-k}\jabr{\ell}^{-k} \ind_{\{|r-at|\lesssim \sqrt{1+t}\}}\, d\theta da d\ell \lesssim \ind_{t\leq1}+\ind_{t\geq1}\dfrac{\sqrt{1+t}}{t}\lesssim \dfrac{1}{\sqrt{1+t}},
	\end{align*}
	where we used the fact that the integral with respect to $ a $ is restricted to the set $ \{|a-r/t|\leq \sqrt{1+t}/t\} $ which has size $ \sqrt{1+t}/t $. This yields the first claim.
\end{proof}

In the study of the long-time behavior the following estimate will be useful.
\begin{lem}\label{lem:EffectiveFieldAA}
	Assume that 
	\begin{align*}
		|\partial_r\F_{\eff}(t,r)|\leq \dfrac{K'}{r^3+t^3}.
	\end{align*}
	Then it holds for $ t\geq1 $
	\begin{align*}
		|\F_{\eff}(t,\tilde{R}) - \F_{\eff}(t,at)| \lesssim& \dfrac{K'}{t^3}\jabr{\ell}^{1/2}\jabr{a}\jabr{\dfrac{1}{a^2}}\jabr{\theta}\ln\jabr{t}.
	\end{align*}
\end{lem}
\begin{proof}
	Indeed, we have
	\begin{align*}
		|\F_{\eff}(t,\tilde{R}) - \F_{\eff}(t,at)| &\leq \int_0^1 |\partial_r\F_{\eff}(t,\tilde{R}+s(at-\tilde{R}))| |at-\tilde{R}| \, ds.
	\end{align*}
	By Lemma \ref{lem:EstimatesRadFunct} \eqref{it:lem:EstimatesRadFunct2} we obtain
	\begin{align*}
		|at-\tilde{R}| &\lesssim |\theta| + ||\theta+at|-\tilde{R}|\lesssim |\theta| + p\ln\jabr{\dfrac{\tilde{R}}{p}} \lesssim |\theta| + p\ln\jabr{\dfrac{|\theta+at|}{p}} \lesssim |\theta|+\jabr{p} + p\ln\jabr{\theta+at}
		\\
		&\lesssim \jabr{p}\jabr{\theta}\jabr{a}\ln\jabr{t} \lesssim \jabr{\ell}^{1/2}\jabr{a}\jabr{\dfrac{1}{a^2}}\jabr{\theta}\ln\jabr{t}.
	\end{align*}
	In the last inequality we used $ p\approx \jabr{a\sqrt{\ell}}/a^2\lesssim \jabr{\ell}^{1/2}\jabr{a^{-2}} $ by Lemma \ref{lem:EstimatesRadFunct} \eqref{it:lem:EstimatesRadFunct1}. Finally, we apply the estimate for $ \partial_r\F_{\eff} $ to obtain the claim.
\end{proof}

\subsection{Estimates of the gravitational field in action-angle variables}\label{subsec:EstEffGravFieldAA}
Using the previous lemmas we next provide estimates for the gravitational field as relevant for the equation in action-angle variables. Recall that the gravitational potential has the following form (see \eqref{eq:Sec1:VPinActionAngle}),
\begin{align*}
	\Psi(t,r) &= - \int_{\R}\int_0^\infty\int_0^\infty  \,  \dfrac{\gamma(t,\theta,a,\ell)^2}{\max(R(\theta+at,a),r)} \, d\theta da d\ell, \quad \F(t,r)=-\partial_r\Psi(t,r),
\end{align*}
and $ \tilde{\Psi} := \Psi(t,\tilde{R}) $ is the gravitational potential in action-angle coordinates. 
We prove bounds on $ \partial_{\alpha}\tilde{\Psi} $ and $ \partial_{\alpha,\beta}^2\tilde{\Psi} $, for $ \alpha,\, \beta\in \{\theta,a\} $. To this end, we will assume suitable estimates on $ \F(t,r) $ and $ \partial_r\F(t,r) $ (corresponding to assumptions on $\gamma$ as in the previous sections). We also note that in the results below we will make use of the assumption that $ a\gtrsim \delta/\jabr{\ell}^{1/2} $ for some $ \delta\in (0,1) $. As will be shown in the study of the nonlinear problem this estimate is dynamically stable, see Corollary \ref{cor:nonlin_chars}.

\begin{lem}\label{lem:EstimatesGravFieldAA}
	Let $ \delta\in (0,1) $. Assume that $ \F=\F(t,r) $ for $ t,\, r\geq 0 $ satisfies
	\begin{align*}
		|\F(t,r)| \leq \dfrac{K}{1+r^2+t^2}
	\end{align*}
	for some constant $ K>0 $. Then, we have for $ a\gtrsim \delta/\jabr{\ell}^{1/2} $ and $ t\geq0 $
	\begin{align*}
		|\partial_{\theta}\tilde{\Psi}| \lesssim \dfrac{\min\{1,a\}}{\delta^2}\dfrac{K}{(1+t)^{3/2}}, \quad |\partial_{a}\tilde{\Psi}| \lesssim \dfrac{\jabr{\ell}^{1/2}}{\delta^3}\dfrac{K}{1+t}.
	\end{align*}
\end{lem}
\begin{proof}
	We will use the formulas
	\begin{align*}
		\partial_{\theta}\tilde{\Psi} = -\F(t,\tilde{R}) \partial_{\theta}\tilde{R}, \quad \partial_{a}\tilde{\Psi} = -\F(t,\tilde{R}) \partial_{a}\tilde{R},
	\end{align*}
	and will distinguish cases according to $ \ell\leq 1 $ or $ \ell\geq 1 $.
	
	\textit{Bound for $ \partial_{\theta}\tilde{\Psi} $.} 
    We have by Lemma \ref{lem:EstDynRadNoBulk} 
	\begin{align*}
		|\partial_{\theta}\tilde{\Psi}| \lesssim \dfrac{K}{1+\tilde{R}^2+t^2}\min\left( \dfrac{\tilde{R}}{1+t},1 \right) \jabr{\dfrac{p}{\tilde{R}}}^{1/2}.
	\end{align*}
	For $ \ell \leq 1 $ we have
    \begin{align*}
        p\approx a^{-2}\jabr{a\sqrt{\ell}}\lesssim \min\{a^{-2}, a^{-1}\}\leq 1/\delta^2.
    \end{align*}
    Hence, we obtain with $ a\gtrsim \delta $ that
	\begin{align*}
		|\partial_{\theta}\tilde{\Psi}| \lesssim \ind_{p/\tilde{R}\leq1} \dfrac{K}{(1+t)^2} + \ind_{p/\tilde{R}\geq 1}  \dfrac{K}{(1+t)^2} p^{1/2}\lesssim \dfrac{1}{\delta}\dfrac{K}{(1+t)^2}\lesssim \dfrac{\min\{1, a\}}{\delta^2}\dfrac{K}{(1+t)^2} .
	\end{align*} 
	For $ \ell\geq 1 $ we observe first that since $ a\geq \delta/\jabr{\ell}^{1/2} $,
	\begin{align*}
		\jabr{\dfrac{p}{\tilde{R}}} \lesssim \jabr{\dfrac{1}{a^2\ell}}\lesssim \dfrac{1}{\delta^2}.
	\end{align*}
    In case $ a\geq 1 $, Lemma \ref{lem:EstimatesRadFunct} yields
	\begin{align*}
		|\partial_{\theta}\tilde{\Psi}| \lesssim \dfrac{1}{\delta}\dfrac{K}{(1+t)^2}.
	\end{align*}
	For $ a\leq 1 $ we use Lemma \ref{lem:EstimatesRadFunct} \eqref{it:lem:EstimatesRadFunct1}, \eqref{it:lem:EstimatesRadFunct2} to obtain
	\begin{align*}
		|\partial_{\theta}\tilde{\Psi}|\lesssim \dfrac{K}{(1+t)^{3/2}}\dfrac{1}{r_0^{1/2}}\jabr{\dfrac{p}{\tilde{R}}}^{1/2} \lesssim \dfrac{a}{\delta} \dfrac{K}{(1+t)^{3/2}}\dfrac{\jabr{a\sqrt{\ell}}^{1/2}}{a\sqrt{\ell}} \lesssim \dfrac{a}{\delta^2}\dfrac{K}{(1+t)^{3/2}}.
	\end{align*}
	
	\textit{Bound for $ \partial_{a}\tilde{\Psi} $} For $ \ell\leq 1 $ we have by Lemma \ref{lem:EstDynRadNoBulk} 
	\begin{align*}
		|\partial_{a}\tilde{\Psi}| \lesssim \dfrac{K}{1+\tilde{R}^2+t^2}\min\left( \dfrac{\tilde{R}}{1+t},1 \right) \left( t \jabr{\dfrac{p}{\tilde{R}}}^{1/2} + \dfrac{p}{a}\ln\jabr{\dfrac{\tilde{R}}{p}}\right) \lesssim \dfrac{K}{1+t} \jabr{\dfrac{1}{a}}\jabr{\dfrac{p}{\tilde{R}}}\min\left( \dfrac{\tilde{R}}{1+t},1 \right).
	\end{align*}
	Here, we used the inequality $ x\ln\jabr{1/x}\lesssim \jabr{x} $ for $ x=p/\tilde{R} $. Arguing as in Step 1 by splitting into the case $ \ell\leq 1 $ and $ \ell \geq 1 $ yields
	\begin{align*}
		|\partial_{a}\tilde{\Psi}|\lesssim \dfrac{\jabr{\ell}^{1/2}}{\delta^3}\dfrac{K}{1+t}.
	\end{align*}
\end{proof}
Next we treat the second order derivatives.

\begin{lem}\label{lem:EstimatesDerivGravFieldAA}
	Let $ \delta\in (0,1) $. Assume that $ \F=\F(t,r) $ for $ t,\, r\geq 0 $ satisfies
	\begin{align*}
		|\F(t,r)| \leq \dfrac{K}{1+r^2+t^2}\min\left( \dfrac{r^2}{(1+t)^2},1 \right), \quad 	|\partial_r\F(t,r)| \leq \dfrac{K'}{1+r^3+t^3}\min\left( \dfrac{r}{1+t},1 \right),
	\end{align*}
	for some constants $ K,\,  K'>0 $. Then, we have for $ a\gtrsim \delta/\jabr{\ell}^{1/2} $ and $ t\geq0 $ that
	\begin{align*}
		|\partial_{\theta}^2\tilde{\Psi}| \lesssim \dfrac{1}{\delta^3}\min\left( \dfrac{a}{\jabr{\ell}^{1/2}}, 1 \right) \dfrac{K+K'}{(1+t)^2},\quad |\partial_{\theta a}^2\tilde{\Psi}|\lesssim \dfrac{1}{\delta^4}\dfrac{K+K'}{(1+t)^{3/2}},
		\quad
		|\partial_{a}^2\tilde{\Psi}| \lesssim \dfrac{1}{\delta^6}\dfrac{\jabr{\ell}^{1/2}}{a}\dfrac{K+K'}{1+t}.
	\end{align*}
\end{lem}
\begin{proof}
	Note that
	\begin{align*}
		\partial_{\theta}^2\tilde{\Psi} &= -\F(t,\tilde{R})\partial_{\theta}^2\tilde{R} - \partial_r\F(t,\tilde{R}) \left( \partial_{\theta}\tilde{R}\right)^2,
		\\
		\partial_{\theta a}^2\tilde{\Psi} &= -\F(t,\tilde{R})\partial_{\theta a}^2\tilde{R} - \partial_r\F(t,\tilde{R}) \, \partial_{\theta}\tilde{R} \, \partial_{a}\tilde{R},
		\\
		\partial_{a}^2\tilde{\Psi} &= -\F(t,\tilde{R})\partial_{a}^2\tilde{R} - \partial_r\F(t,\tilde{R}) \left( \partial_{a}\tilde{R}\right)^2.
	\end{align*}
	We estimate each term separately, and again distinguish the cases $ \ell\leq 1 $ and $ \ell\geq 1 $.
	
	\textit{Bound for $ \partial_{\theta}^2\tilde{\Psi} $.} 
    We use Lemma \ref{lem:EstDynRadNoBulk} to get
	\begin{align*}
		|\F(t,\tilde{R})||\partial_{\theta}^2\tilde{R}| \lesssim \dfrac{K}{(1+t)^2}\min\left( \dfrac{\tilde{R}^2}{(1+t)^2},1 \right) \dfrac{p}{\tilde{R}^2}.
	\end{align*}
	For $ \ell \leq 1 $ we have $ p\lesssim 1/\delta^2 $ and hence
	\begin{align*}
		|\F(t,\tilde{R})||\partial_{\theta}^2\tilde{R}| \lesssim \dfrac{1}{\delta^2} \dfrac{K}{(1+t)^4}\lesssim \dfrac{1}{\delta^3}\min\left( \dfrac{a}{\jabr{\ell}^{1/2}},1 \right) \dfrac{K}{(1+t)^4}.
	\end{align*}
	For $ \ell\geq 1 $ we obtain with Lemma \ref{lem:EstimatesRadFunct} \eqref{it:lem:EstimatesRadFunct1} and \eqref{it:lem:EstimatesRadFunct2} that
    \begin{align*}
		|\F(t,\tilde{R})||\partial_{\theta}^2\tilde{R}| &\lesssim \dfrac{K}{(1+t)^2} \dfrac{p}{\tilde{R}^2}\lesssim \dfrac{1}{\delta^2}\dfrac{K}{(1+t)^2} \dfrac{1}{r_0}\lesssim \dfrac{1}{\delta^2}\dfrac{K}{(1+t)^2} \dfrac{\jabr{a\sqrt{\ell}}}{\ell}\lesssim \dfrac{1}{\delta^2}\left( \dfrac{1}{\ell}+\dfrac{a}{\sqrt{\ell}} \right) \dfrac{K}{(1+t)^2} 
		\\
		&\lesssim \dfrac{1}{\delta^3}\dfrac{a}{\sqrt{\ell}}\dfrac{K}{(1+t)^2}.
	\end{align*}
	Alternatively, we have
	\begin{align*}
		|\F(t,\tilde{R})||\partial_{\theta}^2\tilde{R}| \lesssim \dfrac{K}{(1+t)^2}\min\left( \dfrac{\tilde{R}^2}{(1+t)^2},1 \right) \dfrac{p}{\tilde{R}^2}\lesssim \dfrac{K}{(1+t)^3} \dfrac{p}{\tilde{R}}\lesssim \dfrac{1}{\delta^2}\dfrac{K}{(1+t)^3}.
	\end{align*}
	Furthermore, we have with Lemma \ref{lem:EstDynRadNoBulk}
	\begin{align*}
		|\partial_r\F(t,\tilde{R})| \left( \partial_{\theta}\tilde{R}\right)^2 \lesssim \dfrac{K'}{(1+t)^3}\min\left( \dfrac{\tilde{R}}{1+t},1 \right) \jabr{\dfrac{p}{\tilde{R}}}.
	\end{align*}
	When $ \ell\leq 1 $ we have with $ p\lesssim 1/\delta^2 $
	\begin{align*}
		|\partial_r\F(t,\tilde{R})| \left( \partial_{\theta}\tilde{R}\right)^2 \lesssim \dfrac{1}{\delta^2}\dfrac{K'}{(1+t)^4}\lesssim \dfrac{1}{\delta^3}\dfrac{a}{\jabr{\ell}^{1/2}}\dfrac{K}{(1+t)^4}.
	\end{align*}
	On the other hand, when $ \ell\geq 1 $ we get
	\begin{align*}
		|\partial_r\F(t,\tilde{R})| \left( \partial_{\theta}\tilde{R}\right)^2 \lesssim \dfrac{K'}{(1+t)^3} \jabr{\dfrac{p}{\tilde{R}}} \lesssim \dfrac{1}{\delta^2}\dfrac{K'}{(1+t)^3} \lesssim \dfrac{1}{\delta^3}\min\left( \dfrac{a}{\sqrt{\ell}},1\right) \dfrac{K'}{(1+t)^3}.
	\end{align*}
	
	\textit{Bound for $ \partial_{\theta a}^2\tilde{\Psi} $.} 
    By Lemma \ref{lem:EstDynRadNoBulk} we have for $ \ell \leq 1 $
	\begin{align*}
		|\F(t,\tilde{R})||\partial_{\theta a}^2\tilde{R}| &\lesssim \dfrac{K}{(1+t)^2}\min\left( \dfrac{\tilde{R}^2}{(1+t)^2},1 \right)\left( \dfrac{tp}{\tilde{R}^2}+\dfrac{p}{a\tilde{R}}\right) 
		\lesssim \dfrac{1}{\delta^3}\dfrac{K}{(1+t)^3},
	\end{align*}
	having used again $ p\lesssim1/\delta^2 $. On the other hand, for $ \ell\geq 1 $ we have
	\begin{align*}
		|\F(t,\tilde{R})||\partial_{\theta a}^2\tilde{R}| &\lesssim \dfrac{K}{1+\tilde{R}^2+t^2}\min\left( \dfrac{\tilde{R}^2}{(1+t)^2},1 \right)\left( \dfrac{tp}{\tilde{R}^2}+\dfrac{p}{a\tilde{R}}\right) 
		\lesssim \dfrac{K}{(1+t)^{3/2}} \left( \dfrac{1}{\delta^2} + \dfrac{1}{\delta^2}\dfrac{1}{a}\dfrac{1}{r_0^{1/2}} \right) 
		\\
		&\lesssim \dfrac{1}{\delta^2}\dfrac{K}{(1+t)^{3/2}} \dfrac{\jabr{a\sqrt{\ell}}^{1/2}}{a\sqrt{\ell}}\lesssim \dfrac{1}{\delta^3}\dfrac{K}{(1+t)^{3/2}}.
	\end{align*}
	Furthermore, for $ \ell\leq 1 $ we have with Lemma \ref{lem:EstDynRadNoBulk}
	\begin{align*}
		|\partial_r\F(t,\tilde{R})|  |\partial_{\theta}\tilde{R} \, \partial_{a}\tilde{R}| &\lesssim \dfrac{K'}{1+\tilde{R}^3+t^3}\min\left( \dfrac{\tilde{R}}{1+t},1 \right) \jabr{\dfrac{p}{\tilde{R}}}^{1/2}\left( t\jabr{\dfrac{p}{\tilde{R}}}^{1/2}+ \dfrac{p}{a}\ln\jabr{\dfrac{\tilde{R}}{p}} \right)
		\\
		&\lesssim \dfrac{K'}{(1+t)^2}\min\left( \dfrac{\tilde{R}}{1+t},1 \right)\jabr{\dfrac{p}{\tilde{R}}}^{1/2}\left( \jabr{\dfrac{p}{\tilde{R}}}^{1/2}+\dfrac{1}{a}\dfrac{p}{1+\tilde{R}} \ln\jabr{\dfrac{\tilde{R}}{p}} \right)
		\\
		&\lesssim \ind_{p/\tilde{R}\leq 1} \dfrac{K'}{(1+t)^2}\jabr{\dfrac{1}{a}} + \ind_{p/\tilde{R}\geq 1}  \dfrac{K'}{(1+t)^3}\jabr{\dfrac{1}{a}}\jabr{p} \lesssim \dfrac{1}{\delta^4}\dfrac{K'}{(1+t)^2}.
	\end{align*}
	On the other hand, for $ \ell\geq 1 $ we have
	\begin{align*}
		|\partial_r\F(t,\tilde{R})|  |\partial_{\theta}\tilde{R} \, \partial_{a}\tilde{R}| &\lesssim \dfrac{K'}{1+\tilde{R}^3+t^3}  \jabr{\dfrac{p}{\tilde{R}}}^{1/2}\left( t\jabr{\dfrac{p}{\tilde{R}}}^{1/2}+ \dfrac{p}{a}\ln\jabr{\dfrac{\tilde{R}}{p}} \right) \lesssim \dfrac{1}{\delta}\dfrac{K'}{1+\tilde{R}^2+t^2} \left( 1+\dfrac{1}{a}\jabr{\dfrac{p}{\tilde{R}}}\right) 
		\\
		&\lesssim \dfrac{1}{\delta^3}\dfrac{K'}{(1+t)^{3/2}} \dfrac{1}{ar_0^{1/2}} \lesssim\dfrac{1}{\delta^3}\dfrac{K'}{(1+t)^{3/2}} \dfrac{\jabr{a\sqrt{\ell}}^{1/2}}{a\sqrt{\ell}} \lesssim \dfrac{1}{\delta^4}\dfrac{K'}{(1+t)^{3/2}}.
	\end{align*}
	
	\textit{Bound for $ \partial_{a}^2\tilde{\Psi} $.} 
    We have with Lemma \ref{lem:EstDynRadNoBulk} and $ \ell\leq 1 $
	\begin{align*}
		|\F(t,\tilde{R})||\partial_{a}^2\tilde{R}| &\lesssim \dfrac{K}{1+\tilde{R}^2+t^2} \, \min\left( \dfrac{\tilde{R}^2}{(1+t)^2},1 \right) \left( \dfrac{pt^2}{\tilde{R}^2} +\dfrac{tp}{a\tilde{R}} +\dfrac{p}{a^2}\ln\jabr{\dfrac{\tilde{R}}{p}}\right) 
		\\
		&\lesssim \dfrac{K}{1+t} \min\left( \dfrac{\tilde{R}^2}{(1+t)^2},1 \right) \left( \dfrac{pt}{\tilde{R}^2}+\dfrac{pt}{a\tilde{R}^2} +\dfrac{p}{a^2\tilde{R}}\ln\jabr{\dfrac{\tilde{R}}{p}}\right) 
		\\
		&\lesssim \dfrac{K}{1+t} \jabr{\dfrac{1}{a^2}}\min\left( \dfrac{\tilde{R}^2}{(1+t)^2},1 \right) \left( \dfrac{pt}{\tilde{R}^2} +\jabr{\dfrac{p}{\tilde{R}}}\right) \lesssim \dfrac{1}{\delta^4}\dfrac{K}{1+t}.
	\end{align*}
    When $ \ell\geq 1 $ we get
	\begin{align*}
		|\F(t,\tilde{R})||\partial_{a}^2\tilde{R}| &\lesssim \dfrac{K}{1+\tilde{R}^2+t^2} \, \min\left( \dfrac{\tilde{R}^2}{(1+t)^2},1 \right) \left( \dfrac{pt^2}{\tilde{R}^2} +\dfrac{tp}{a\tilde{R}} +\dfrac{p}{a^2}\ln\jabr{\dfrac{\tilde{R}}{p}}\right) 
		\\
		&\lesssim \dfrac{K}{1+t}\left(p + \dfrac{1}{\delta}\dfrac{\sqrt{p}}{a} + \dfrac{1}{\delta^2}\dfrac{1}{a^2} \right) \lesssim \dfrac{1}{\delta^2}\dfrac{K}{1+t}\left( p+\dfrac{1}{a^2}\right) \lesssim \dfrac{1}{\delta^2}\dfrac{K}{1+t}\dfrac{\jabr{a\sqrt{\ell}}}{a^2}
		\\
		&\lesssim \dfrac{1}{\delta^2}\dfrac{K}{1+t}\left( \dfrac{1}{a^2}+\dfrac{\sqrt{\ell}}{a}\right) \lesssim \dfrac{1}{\delta^3}\dfrac{\sqrt{\ell}}{a}\dfrac{K}{1+t},
	\end{align*}
	and in case $ \ell\leq 1 $ we have via Lemma \ref{lem:EstDynRadNoBulk} that
	\begin{align*}
		|\partial_r\F(t,\tilde{R})| \left( \partial_{a}\tilde{R}\right)^2 &\lesssim \dfrac{K'}{1+\tilde{R}^3+t^3} \, \min\left( \dfrac{\tilde{R}}{1+t},1 \right) \left( t^2\jabr{\dfrac{p}{\tilde{R}}}+ \dfrac{p^2}{a^2}\ln^2\jabr{\dfrac{\tilde{R}}{p}} \right) 
		\\
		&\lesssim \dfrac{K'}{1+\tilde{R}^2+t^2} \, \min\left( \dfrac{\tilde{R}}{1+t},1 \right) \left( t\jabr{\dfrac{p}{\tilde{R}}}+ \dfrac{p^2}{a^2(1+\tilde{R})}\ln^2\jabr{\dfrac{\tilde{R}}{p}} \right)
		\\
		&\lesssim \ind_{p/\tilde{R}\leq 1} \dfrac{K'}{1+t}\jabr{\dfrac{1}{a^2}} + \ind_{p/\tilde{R}\geq 1} \dfrac{K'}{1+t}\jabr{\dfrac{1}{a^2}}\jabr{p}^2 \lesssim \dfrac{1}{\delta^6}\dfrac{K'}{1+t}.
	\end{align*}
	Here, we also used $ x\ln^2\jabr{1/x}\lesssim\jabr{x} $ for $ x=p/\tilde{R} $. For $ \ell\geq 1 $ we estimate 
	\begin{align*}
		|\partial_r\F(t,\tilde{R})| \left( \partial_{a}\tilde{R}\right)^2 &\lesssim \dfrac{K'}{1+\tilde{R}^3+t^3} \, \min\left( \dfrac{\tilde{R}}{1+t},1 \right) \left( t^2\jabr{\dfrac{p}{\tilde{R}}}+ \dfrac{p^2}{a^2}\ln^2\jabr{\dfrac{\tilde{R}}{p}} \right) 
		\\
		&\lesssim \dfrac{K'}{1+t}\left( p+\dfrac{1}{\delta^2}\dfrac{1}{a^2} \right) \lesssim \dfrac{1}{\delta^3}\dfrac{\sqrt{\ell}}{a}\dfrac{K'}{1+t}.
	\end{align*}
\end{proof}


\section{Nonlinear dynamics}\label{sec:nonlinear}
In this section we turn to the study of the full nonlinear problem
\begin{align}\label{eq:Sec4:VPeq}
	\begin{split}
		\partial_t \gamma + \lambda \{\gamma,\tilde{\Psi}\} =0,\quad \tilde{\Psi}(t,\theta,a) = \Psi(t,\tilde{R}),
		\quad \Psi(t,r) = - \int_{\R}\int_0^\infty\int_0^\infty  \,  \dfrac{\gamma(t,\theta,a,\ell)^2}{\max(R(\theta+at,a),r)} \, d\theta da d\ell.
	\end{split}
\end{align}
In order to demonstrate the equivalence of \eqref{eq:Sec4:VPeq} with \eqref{eq:VPD} (under radial symmetry), we begin by demonstrating the potential self-consistency of \eqref{eq:Sec4:VPeq}: under suitable smallness assumptions on the gravitational field $\F(t,r) = -\partial_r\Psi(t,r)$, in Lemma \ref{lem:NonlinCharSys} we give quantitative bounds which guarantee that the nonlinear characteristics of \eqref{eq:Sec4:VPeq} remain in a region of hyperbolic trajectories for the linearized problem, if initially this is the case. Subsequently, we establish the global well-posedness of \eqref{eq:Sec4:VPeq} for sufficiently small initial data: we do this first for bounded Lagrangian solutions with suitable moments (see Definition \ref{def:Lagrangian_sol} and Theorem \ref{thm:WellPosedLagrangianSol}), and subsequently show that regularity can be propagated (Theorem \ref{thm:WellPosednessRegularSol}), yielding the global well-posedness of strong solutions in $\SobH^1$.

Our main result then establishes the asymptotic behavior, a modified scattering dynamic which arises naturally from an asymptotic shear equation. While this can be readily seen for the global strong solutions of Theorem \ref{thm:WellPosednessRegularSol}, a key novelty of our arguments is that it can be identified also for the merely Lagrangian solutions of Theorem \ref{thm:WellPosedLagrangianSol} -- see Theorems \ref{thm:ModifiedScattering} respectively \ref{thm:ModifiedScatteringLagrangianSol}.

\medskip
\paragraph{\bf The characteristic system of \eqref{eq:Sec4:VPeq}.} In order to justify the study of the nonlinear problem, we first consider the characteristic system of \eqref{eq:Sec4:VPeq}, i.e.\ for fixed $\ell>0$
\begin{align}\label{eq:Sec4:NonlinCharSys}
	\begin{cases}
		\dot{\theta} = \lambda\partial_{a}\tilde{\Psi}(\theta,a,\ell)=-\lambda\F(t,\tilde{R})\partial_{a}\tilde{R},
		\\
		\dot{a} = -\lambda\partial_{\theta}\tilde{\Psi}(\theta,a,\ell) = \lambda\F(t,\tilde{R})\partial_{\theta}\tilde{R}.
	\end{cases}
\end{align}
where as usual $\F(t,r) = -\partial_r\Psi(t,r)$.

\begin{lem}\label{lem:NonlinCharSys}
Assume $ \F(t,r) $ is a given continuous function for $ t\in [0,T] $, $ r>0 $, satisfying for some $C>0$ that
	\begin{align*}
		|\F(t,r)| \leq \dfrac{C\varepsilon^2}{1+r^2+t^2}.
	\end{align*} 
There exists a sufficiently small $c>0 $ such that for any $ \delta\in (0,1) $ and $ \nu\in [0,3/4] $ the following holds provided that $ \eps \leq c\delta^{2} $. For fixed $\ell>0$, consider a solution  $(a(t),\theta(t))$  to \eqref{eq:Sec4:NonlinCharSys} with initial data satisfying $ a(0)\jabr{\ell}^{\nu}\geq \delta $. Then for all $ t\in [0,T] $ we have
\begin{equation}\label{eq:Sec4:ActionLowerBound}
    a(t)\jabr{\ell}^{\nu}\geq \frac{\delta}{2},
\end{equation}
and the further bounds
\begin{equation}\label{eq:nonlin_char_bds}
\begin{aligned}
			|a(t)-a(0)| &\leq C\dfrac{\varepsilon^2}{\delta^2},
			\\
			|\theta(t)-\theta(0)| &\leq  C\dfrac{\varepsilon^2}{\delta^3} \left( 1+\jabr{\ell}^{3\nu-1}\right) \ln(1+t).
\end{aligned}
\end{equation}
\end{lem}
For future reference, we record the following simple consequence of Lemma \ref{lem:NonlinCharSys}.
\begin{cor}\label{cor:nonlin_chars}
 For $\delta>0$, let
\begin{align*}
    \D(\delta)=\left\lbrace (\theta,a,\ell) \, : \, a\jabr{\ell}^{\frac12}\geq \delta\right\rbrace.
\end{align*}
Under the assumptions of Lemma \ref{lem:NonlinCharSys} it follows that if $(\theta(0),a(0),\ell)\subset \D(\delta)$, then the corresponding solution satisfies $(\theta(t),a(t),\ell)\subset \D(\frac{\delta}{2})$ on its interval of definition.   
\end{cor}

\begin{proof}[Proof of Lemma \ref{lem:NonlinCharSys}]
	We first prove \eqref{eq:Sec4:ActionLowerBound} and then turn to the proof of the estimates \eqref{eq:nonlin_char_bds}. We split the proof into several steps.
	
	\textit{Step 1.} We first consider $ \ell\leq 1 $ and therefore suppress the term $\jabr{\ell}^{\nu}\approx 1$. We can assume that $ a(t)\leq1 $ for all $ t\in [0,T] $, else by continuity of $a(t)$ we use the following arguments on the time intervals on which $ a(t)\leq1 $. Furthermore, we make a bootstrap argument and show that as long as $ a(t)\geq \delta/4 $ we have in fact $ a(t) \geq \delta/2 $. Since $ a(0)\geq \delta $ it follows $ a(t)\geq \delta/2 $ for all $ t\in [0,T] $ by a continuation argument. 
	
	We use Lemma \ref{lem:EstimatesRadFunct} \eqref{it:lem:EstimatesRadFunct1} to get with $ p\approx a^{-2}\jabr{a\sqrt{\ell}} \lesssim a^{-2}+\ell $ and $ a,\, \ell\leq 1 $
	\begin{align*}
		\jabr{\dfrac{p}{\tilde{R}}}^{1/2}\approx \left( 1+ \dfrac{p}{\tilde{R}}\right)^{1/2} \lesssim 1+ \dfrac{1}{a\sqrt{\widetilde{R}}}.
	\end{align*}
	Lemma \ref{lem:EstDynRadNoBulk} and the assumptions on $ \F $ yield
	\begin{align}\label{eq:Sec4:ProofInvSet}
		|\F(t,\tilde{R}(\theta,a))|  |\partial_\theta \tilde{R}(\theta,a)| \leq \dfrac{C\varepsilon^2}{1+t^2} \left( 1+\dfrac{1}{a\sqrt{\widetilde{R}}}\right)  \leq\dfrac{C\varepsilon^2}{1+t^2} \left( 1+\dfrac{1}{a\sqrt{|\theta+at|}}\right) .
	\end{align}
	In the last estimate we used Lemma \ref{lem:EstimatesRadFunct} \eqref{it:lem:EstimatesRadFunct2}. Let us define $ \psi(t) := \theta(t)+a(t)t $. We observe that due to \eqref{eq:Sec4:NonlinCharSys}
	\begin{align*}
		\dot{\psi} = a - \lambda\partial_aR(\psi,a) \F(t,R(\psi,a)).
	\end{align*}
	We then have with Lemma \ref{lem:EstimatesRadFunct} \eqref{it:lem:EstimatesRadFunct1} and Lemma \ref{lem:EstimatesDerivRadFunct}, the constant $C>0$ may change from line to line,
	\begin{align*}
		|\partial_aR(\psi,a)| |\F(t,R(\psi,a))|&\leq \dfrac{p}{a} \ln \jabr{\dfrac{\tilde{R}}{p}} \dfrac{C\varepsilon^2}{1+t^2+\tilde{R}^2} \leq \dfrac{1}{a^3} \jabr{a\sqrt{\ell}}\ln \jabr{\dfrac{\tilde{R}}{p}} \dfrac{C\varepsilon^2}{1+t^2+\tilde{R}^2} 
		\\
		&\leq \dfrac{1}{a^3} \ln \jabr{\tilde{R}} \dfrac{C\varepsilon^2}{1+t^2+\tilde{R}^2}
		\leq  \dfrac{C\varepsilon^2}{\delta^3}\dfrac{1}{1+t^{3/2}}.
	\end{align*}
	Here, we also used the assumption that $ a,\ell \leq 1 $ and hence $ 1/p\lesssim 1 $. In the last estimate we used the assumption $ a(t)\geq \delta/4 $. We then obtain
	\begin{align*}
		\psi'(t) \geq \dfrac{\delta}{4} - \dfrac{C\varepsilon^2}{\delta^3}\dfrac{1}{1+t^{3/2}} \geq \dfrac{\delta}{8}>0
	\end{align*}
	as long as $ \varepsilon\leq c\delta^2 $ with $c>0$ sufficiently small. Thus, the function $ t\mapsto \psi(t) $ is one-to-one. We then have the estimate with \eqref{eq:Sec4:ProofInvSet}
	\begin{align*}
		a(t) &\geq a(0) - \int_0^t \dfrac{C\varepsilon^2}{1+s^2}\left( 1+ \dfrac{1}{a(s) \sqrt{|\psi(s)|}}\right) \, ds 
		\\
		&\geq \delta - \dfrac{C\varepsilon^2}{\delta}\int_0^t \dfrac{1}{1+s^2}\left( 1+\dfrac{1}{\sqrt{|\psi(s)|}}\right)  \, ds
		\\
		&\geq\delta - \dfrac{C\varepsilon^2}{\delta} \left[\arctan(t)+\int_0^t \dfrac{1}{1+s^2}\ind_{|\psi(s)|\geq1} \, ds + \int_0^t\ind_{|\psi(s)|\leq 1} \dfrac{1}{\sqrt{|\psi(s)|}} \, ds \right] 
		\\
		&\geq \delta - \dfrac{C\varepsilon^2}{\delta} - \dfrac{C\varepsilon^2}{\delta} \int_{-1}^{1} \dfrac{1}{\sqrt{|\xi|}}\dfrac{1}{\psi'(\psi^{-1}(\xi))} \, d\xi
		\\
		&\geq \delta - \dfrac{C\varepsilon^2}{\delta} -\dfrac{C\varepsilon^2}{\delta^2} \geq \dfrac{\delta}{2}.
	\end{align*}
	The last estimate holds as long as $ \varepsilon\leq c\delta^2 $ with $c>0$ sufficiently small. As mentioned above this allows to use a continuation argument yielding $ a(t)\geq \delta/2 $ for all $ t\in [0,T] $.
	
	\textit{Step 2.} We now consider the case $ \ell\geq 1 $ and $ \nu\in [0,1/2] $. We write therefore $ \ell^\nu $ instead of $ \jabr{\ell}^{\nu} $. We can assume that $ a(t)\leq1/\ell^\nu $ for all $ t\in [0,T] $ otherwise we use the following arguments on the time intervals on which $ a(t)\leq1/\ell^\nu $. Note that $ a\sqrt{\ell}\leq \ell^{1/2-\nu} $ and hence $ r_0\approx \ell/\jabr{a\sqrt{\ell}}\geq \ell^{1/2+\nu} $. We then have due to the assumptions on $ \F $ and Lemma \ref{lem:EstDynRadNoBulk}
	\begin{align*}
		|\F(t,\tilde{R}(\theta,a))|  |\partial_\theta \tilde{R}(\theta,a)| \leq \dfrac{C\varepsilon^2}{r_0^2+t^2}\jabr{\dfrac{p}{\tilde{R}}}^{1/2} \leq \dfrac{C\varepsilon^2}{\ell^{1+2\nu}+t^2}\left( 1+\dfrac{1}{a\sqrt{\ell}}\right) 
	\end{align*}
	Here, we used also Lemma \ref{lem:EstimatesRadFunct}. We then have
	\begin{align*}
		a(t) &\geq a(0) - \int_0^t \dfrac{C\varepsilon^2}{\ell^{1+2\nu}+s^2}\left( 1+\dfrac{1}{a(s)\sqrt{\ell}}\right)\, ds 
		\\
		&\geq a(0)-C\varepsilon^2\ell^{-\nu-1/2}-\int_0^t \dfrac{C\varepsilon^2}{\ell^{1+2\nu}+s^2}\dfrac{1}{a(s)\sqrt{\ell}}\, ds
	\end{align*}
	We now make a comparison argument with the function $ \alpha(t) $ satisfying
	\begin{align*}
		\alpha(t) = \alpha(0)-C\varepsilon^2\int_0^t \dfrac{1}{1+\ell^{1+2\nu}+s^2}\dfrac{1}{\alpha(s)\sqrt{\ell}}\, ds
	\end{align*}
	and hence
	\begin{align*}
		\alpha(t)= \left( \alpha(0)^2-2C\varepsilon^2\ell^{-1-\nu}\arctan\left( \dfrac{t}{\ell^{1/2+\nu}} \right)  \right)^{1/2}.
	\end{align*}
	This gives
	\begin{align*}
		a(t) &\geq \left( \left( a(0)-C\varepsilon^2\ell^{-\nu-1/2}\right)^2-\pi C\varepsilon^2\ell^{-1-\nu}  \right)^{1/2} \geq \ell^{-\nu} \left( \left( \delta-C\varepsilon^2\ell^{-1/2}\right)^2-\pi C\varepsilon^2\ell^{\nu-1} \right)^{1/2}
		\\
		&\geq \ell^{-\nu}\delta/2.
	\end{align*}
	The last estimate holds when $ \varepsilon\leq c\delta^2 $ and $c>0$ is chosen small enough.
	
	\textit{Step 3.} Finally, we look at the case $ \ell\geq 1 $ and $ \nu\in [1/2,3/4] $. We can again assume that $ a(t)\leq1/\ell^\nu $ for all $ t\in [0,T] $ otherwise we use the following arguments on the time intervals on which $ a(t)\leq1/\ell^\nu $. We now have $ r_0\gtrsim \ell $ and hence
	\begin{align*}
		a(t) &\geq a(0) - \int_0^t \dfrac{C\varepsilon^2}{\ell^2+s^2}\left( 1+\dfrac{1}{a(s)\sqrt{\ell}}\right)\, ds 
		\\
		&\geq a(0)-C\varepsilon^2\ell^{-1}-\int_0^t \dfrac{C\varepsilon^2}{\ell^{2}+s^2}\dfrac{1}{a(s)\sqrt{\ell}}\, ds
	\end{align*}
	As above we get via a comparison argument
	\begin{align*}
		a(t) \geq \left( \left( a(0)- C\varepsilon^2\ell^{-1}\right)^2-C\varepsilon^2\ell^{-3/2}  \right)^{1/2} \geq \ell^{\nu} \left( \left( \delta- C\varepsilon^2\ell^{-1+\nu}\right)^2-C\varepsilon^2\ell^{-3/2+2\nu}  \right)^{1/2}.
	\end{align*}
	For $ 2\nu\leq 3/2 $, i.e.\ $ \nu\leq 3/4 $, we can choose $ \varepsilon\leq c\delta^2 $ with $c>0$ small enough to obtain $ a(t)\geq \ell^{\nu}\delta/2 $.
	
	\textit{Step 4.} We now turn to the proof of the estimates \eqref{eq:nonlin_char_bds}. First of all, we prove the estimate for the actions. When $ \ell \leq 1 $ we have with $ p\approx a^{-2}\jabr{a\sqrt{\ell}}\lesssim a^{-2}+\ell $
	\begin{align*}
		|\F(t,\tilde{R}(\theta,a))|  |\partial_\theta \tilde{R}(\theta,a)| \leq \dfrac{C\varepsilon^2}{1+t^2} \jabr{\dfrac{p}{\tilde{R}}}^{1/2}\leq \dfrac{C\varepsilon^2}{1+t^2} \left( 1+\dfrac{1}{a} \right) \left( 1+\dfrac{1}{\sqrt{|\theta+at|}} \right).
	\end{align*}
	With $ a(t)\gtrsim \delta $ we hence obtain as in Step 1
	\begin{align*}
		|a(t)-a(0)| \leq C\dfrac{\varepsilon^2}{\delta}\int_0^t \dfrac{\varepsilon^2}{1+s^2} \left( 1+\dfrac{1}{\sqrt{|\psi(s)|}} \right)\, ds \leq C\dfrac{\varepsilon^2}{\delta^2}.
	\end{align*}
	On the other hand, for $ \ell\geq 1 $ we have
	\begin{align*}
		|\F(t,\tilde{R}(\theta,a))|  |\partial_\theta \tilde{R}(\theta,a)| \leq \dfrac{C\varepsilon^2}{1+r_0^2+t^2}\left( 1+\dfrac{1}{a\sqrt{\ell}} \right)
	\end{align*}
	and hence
	\begin{align*}
		|a(t)-a(0)|\leq C\varepsilon^2\int_0^t\dfrac{1}{1+r_0^2+s^2}
		\left( 1+\dfrac{1}{a(s)\sqrt{\ell}} \right) \, ds.
	\end{align*}
	We split the integral into the contributions $ a\sqrt{\ell}\leq 1 $ and $ a\sqrt{\ell}\geq 1 $ yielding with $ r_0\approx \ell/\jabr{a\sqrt{\ell}} $ as well as $ a(t)\gtrsim \ell^{-\nu}\delta $
	\begin{align*}
		|a(t)-a(0)|\leq C\varepsilon^2\left( \dfrac{\ell^{\nu-3/2}}{\delta}+1 \right) \leq C\dfrac{\varepsilon^2}{\delta}.
	\end{align*}
	In the last step we used $ \nu\leq 3/4 $.
	
	Now we prove the bound for the angles. For $ \ell\leq 1 $ we have with Lemma \ref{lem:EstDynRadNoBulk}
	\begin{align*}
		|\F(t,\tilde{R}(\theta,a))|  |\partial_a\tilde{R}| &\leq \dfrac{C\varepsilon^2}{1+\tilde{R}^2+t^2} \left(t\jabr{\dfrac{p}{\tilde{R}}}^{1/2} + \dfrac{p}{a}\ln\jabr{\dfrac{\tilde{R}}{p}}\right)  
		\\
		&\leq \dfrac{C\varepsilon^2}{1+t} \left( 1+\dfrac{1}{a} \right)\left( 1+\dfrac{1}{\sqrt{\tilde{R}}} \right) + \dfrac{C\varepsilon^2}{1+t^{3/2}}\dfrac{p}{a}.
	\end{align*}
	In the last step we split into the cases $ \tilde{R}/p\geq1 $ and $ \tilde{R}/p\leq 1 $. We hence obtain with $ a(t)\gtrsim \delta $ and $ \tilde{R}\gtrsim |\theta+at|=|\psi(t)| $
	\begin{align*}
		|\theta(t)-\theta(0)|\leq C\dfrac{\varepsilon^2}{\delta}\int_0^t \dfrac{1}{1+s}\dfrac{1}{\sqrt{\psi(s)}}\, ds + C\dfrac{\varepsilon^2}{\delta^3}
		\leq C\dfrac{\varepsilon^2}{\delta^3}\ln(1+t).
	\end{align*}
	Here, also used $ p\lesssim \ell+1/a^2 $ and $ \ell\leq 1 $ as well as arguments from Step 1 above. 
	
	On the other hand, for $ \ell\geq 1 $ we have with Lemma \ref{lem:EstDynRadNoBulk}
	\begin{align*}
		|\F(t,\tilde{R}(\theta,a))|  |\partial_a\tilde{R}| &\leq \dfrac{C\varepsilon^2}{1+\tilde{R}^2+t^2} \left(t\jabr{\dfrac{p}{\tilde{R}}}^{1/2} + \dfrac{p}{a}\ln\jabr{\dfrac{\tilde{R}}{p}}\right)  
		\\
		&\leq \dfrac{C\varepsilon^2}{1+t}\jabr{\dfrac{1}{a\sqrt{\ell}}} + \dfrac{C\varepsilon^2}{1+t}\dfrac{1}{a}\jabr{\dfrac{p}{\tilde{R}}} \leq\dfrac{C\varepsilon^2}{1+t}\left( 1+\dfrac{1}{a}\right)\jabr{\dfrac{1}{a^2\ell}}.
	\end{align*}
	Thus, with $ a(t)\gtrsim \ell^{-\nu}\delta $ we get
	\begin{align*}
		|\theta(t)-\theta(0)|\leq C\dfrac{\varepsilon^2}{\delta^3} \left( 1+\ell^{3\nu-1}\right)\ln(1+t).
	\end{align*}
	This concludes the proof.
\end{proof}

\subsection{Global well-posedness for small initial data}\label{subsec:WellPosedness}

By estimates from Section \ref{subsec:EstGravFieldLInf} one sees that the force $ \nabla_{(\theta,a)}\tilde{\Psi} $ in action-angle variables is locally Lipschitz for distributions in $\gamma\in L^\infty $ with appropriate moments -- see the summary in Lemma \ref{lem:BootstrapDerivGravFieldAALInf} below. As a consequence, (transport along) the flow of the characteristic system \eqref{eq:Sec4:NonlinCharSys} is well-defined. This provides a natural notion of weak solutions, which are the push-forward of their induced characteristic flow. We refer to these as Lagrangian solutions, according to the following definition:

\begin{defi}[Lagrangian solutions]\label{def:Lagrangian_sol}
	We say $ \gamma=\gamma(t,\theta,a,\ell)\in C([0,\infty),L^\infty_{\theta,a,\ell})$ is a Lagrangian solution to \eqref{eq:Sec4:VPeq} with initial condition $ \gamma_0 $, if the following are satisfied:
	\begin{enumerate}[(i)]
		\item The characteristic flow $ \Phi_t(\theta,a,\ell)=(\hat{\Theta}_t(\theta,a,\ell),\hat{A}_t(\theta,a,\ell),\ell) $ to \eqref{eq:Sec4:NonlinCharSys} is well-defined and solves the characteristic system, i.e.\ for all $ (\theta,a,\ell)\in\R\times\R_+\times\R_+ $ and all $ t\geq 0 $ we have
		\begin{equation}\label{eq:chars_def}
			(\hat{\Theta}_t(\theta,a,\ell),\hat{A}_t(\theta,a,\ell)) = (\theta,a) + \lambda\int_0^t (-\partial_a\tilde{\Psi}(s,\Phi_s(\theta,a,\ell)),\partial_\theta\tilde{\Psi}(s,\Phi_s(\theta,a,\ell)))\, ds,
		\end{equation}
		where $ \tilde{\Psi}(t,\theta,a,\ell) $ is defined in \eqref{eq:Sec4:VPeq} via $ \gamma(t) $.
		
		\item The function $ \gamma(t) $ is given by the push-forward $ (\Phi_t)_\#\gamma_0 $, i.e.\ $\gamma(t,\theta,a,\ell)=\gamma_0(\Phi_t^{-1}(\theta,a,\ell))$.
	\end{enumerate}
\end{defi}

We now prove global well-posedness of Lagrangian solutions via moment propagation.
\begin{thm}\label{thm:WellPosedLagrangianSol}
	There exists $c_0>0$ such that for $ \delta\in (0,1) $ and $ \eps_0=c_0\delta^2>0 $ the following holds. For $ \gamma_0\in L^2\cap L^\infty $ with $ \supp (\gamma_0)\subset \D(\delta) $ satisfying
    \begin{align}\label{eq:mom_bd_id}
        \norm[L^\infty]{\jabr{\ell}^{32} \ell^{-1} (a+a^{-1})^{32} \gamma_0}+\norm[L^\infty]{\jabr{\ell}^{16} \ell^{-1} \jabr{\theta}^{32} \gamma_0} \leq \eps\leq\eps_0,
	\end{align}
	the equation \eqref{eq:Sec4:VPeq} has a unique, global in time Lagrangian solution $ \gamma\in C([0,\infty),L^2\cap L^\infty) $ with initial datum $ \gamma_0 $. Moreover, $ \supp\gamma(t)\subset \D(\delta/2) $ and $\gamma(t)$ satisfies for all $t\geq 0$ that
	\begin{align}
		\norm[L^2]{(a+a^{-1})^{4}\gamma(t)} +\norm[L^\infty]{\jabr{\ell}^{32} \ell^{-1} (a+a^{-1})^{32} \gamma(t)}&\lesssim \eps, \label{eq:mom_bd_a}\\
        \norm[L^\infty]{\jabr{\ell}^{16} \ell^{-1} \jabr{\theta}^{k} \gamma(t)}&\lesssim \eps\ln^k\jabr{t},\quad 1\leq k\leq 32.\label{eq:mom_bd_theta}
	\end{align}
\end{thm}

Since moments in $\ell$ commute with the nonlinear dynamic, these can be propagated trivially. Moments in $a$ (and any $L^p$) are also uniformly bounded in time (thanks to the time integrable decay rate of $\partial_\theta\tilde{\Psi}(t)$, see Lemma \ref{lem:BootstrapDerivGravFieldAALInf} and \eqref{eq:mom-eqs}). The role of the $L^2$ bound in \eqref{eq:mom_bd_a} is to ensure the sharp $t^{-1}$ decay for $\partial_a\tilde{\Psi}(t)$, from which in turn the logarithmic bounds on moments in $\theta$ in \eqref{eq:mom_bd_theta} follow. The core difficulty of the proof is to obtain uniqueness in the relatively weak topology of Theorem \ref{thm:WellPosedLagrangianSol}, which in particular does not require any derivative control. 

As a preparation for the proof of Theorem \ref{thm:WellPosedLagrangianSol}, we collect some basic bounds regarding the regularity of the gravitational potential $ \tilde{\Psi} $ in action-angle variables. For sharp rates, the order of vanishing of $\gamma$ as $\ell\to 0$ (quantified via the parameter $q\geq 2$) plays an important role.

\begin{lem}\label{lem:BootstrapDerivGravFieldAALInf}
	Let $ \delta\in(0,1) $ and $ q\geq 2 $. Then, for $ a\jabr{\ell}^{\frac12}\gtrsim \delta$ we have
	\begin{align*}
		|\partial_{\theta}\tilde{\Psi}| &\lesssim \dfrac{\min\{a,1\}}{\delta^2}\dfrac{M_q(\gamma)}{(1+t)^{3/2}}, \quad |\partial_{a}\tilde{\Psi}| \lesssim \dfrac{\jabr{\ell}^{\frac12}}{\delta^3}\dfrac{M_q(\gamma)}{1+t}, \quad
		|\partial_{\theta}^2\tilde{\Psi}| \lesssim \dfrac{1}{\delta^3}\min\left( \dfrac{a}{\jabr{\ell}^{\frac12}}, 1 \right)\dfrac{M_q(\gamma)+M_q'(\gamma)}{(1+t)^2},
		\\
		|\partial_{\theta a}^2\tilde{\Psi}| &\lesssim \dfrac{1}{\delta^4}\dfrac{M_q(\gamma)+M_q'(\gamma)}{(1+t)^{3/2}},
		\quad
		|\partial_{a}^2\tilde{\Psi}| \lesssim \dfrac{1}{\delta^6}\dfrac{\jabr{\ell}^{\frac12}}{a}\dfrac{M_q(\gamma)+M_q'(\gamma)}{1+t},
	\end{align*}
	where $ M_q(\gamma) $ and $ M_q'(\gamma) $ are defined in Lemmas \ref{lem:EstimatesGravFieldLInf} and \ref{lem:EstimatesDerivGravFieldLInf}, respectively.
\end{lem}

\begin{proof}[Proof of Lemma \ref{lem:BootstrapDerivGravFieldAALInf}.]
	The results follow from Lemmas \ref{lem:EstimatesGravFieldAA} and \ref{lem:EstimatesDerivGravFieldAA} by using the estimates in Lemmas \ref{lem:EstimatesGravFieldLInf} and \ref{lem:EstimatesDerivGravFieldLInf}. In this case we have $ K=M_q(\gamma) $ and $ K'=M_q'(\gamma) $. Note that we have for $ q\geq 2 $
	\begin{align*}
		|\partial_r\F(t,r)| \lesssim \dfrac{M_q'(\gamma)}{1+r^2+t^2}\dfrac{1}{r}\min\left( \dfrac{r^q}{(1+t)^q},1 \right) \lesssim \dfrac{M_q'(\gamma)}{1+r^3+t^3}\min\left( \dfrac{r^{q-1}}{(1+t)^{q-1}},1 \right)
	\end{align*}
	as needed for Lemma \ref{lem:EstimatesDerivGravFieldAA}.
\end{proof}

\begin{proof}[Proof of Theorem \ref{thm:WellPosedLagrangianSol}]
In the first part of the proof we show how to bootstrap the relevant a priori estimates for general weak solutions. Standard arguments allow to construct the latter, and in particular also a Lagrangian weak solution satisfying the claimed support assumption (see also Corollary \ref{cor:nonlin_chars}). The second part establishes the uniqueness of Lagrangian solutions.
	
\textit{Moment estimates.} Assume for the sake of a bootstrap that for some $T>0$ we are given a weak solution $\gamma\in C([0,T],L^2\cap L^\infty)$ to \eqref{eq:Sec4:VPeq} with $\gamma(0)=\gamma_0$, $\supp(\gamma(t))\subset\D(\frac{\delta}{2})$ and satisfying for $0\leq t\leq T$ the bounds
\begin{equation}
    \norm[L^2]{(a+a^{-1})^{4}\gamma(t)} +\norm[L^\infty]{\jabr{\ell}^{32} \ell^{-1} (a+a^{-1})^{32} \gamma_0}\leq 100\eps, \qquad 
        \norm[L^\infty]{\jabr{\ell}^{16} \ell^{-1} \jabr{\theta}^{32} \gamma(t)}\leq \eps\jabr{t}^{\frac12}.
\end{equation}
It then follows that $M_2(\gamma)\lesssim \eps^2$, and thus by Lemma \ref{lem:BootstrapDerivGravFieldAALInf} we have
\begin{equation}
    |\partial_{\theta}\tilde{\Psi}| \lesssim \dfrac{\eps^2}{\delta^2}\dfrac{\min\{a,1\}}{(1+t)^{3/2}}, \quad |\partial_{a}\tilde{\Psi}| \lesssim \dfrac{\eps^2}{\delta^3}\dfrac{\jabr{\ell}^{\frac12}}{1+t}.
\end{equation}
Since for a scalar function $f(\theta,a,\ell)$ one has
\begin{equation}\label{eq:mom-eqs}
	\partial_t (f\gamma) + \lambda \{f\gamma,\tilde{\Psi}\} = \lambda \gamma\{f,\tilde{\Psi}\},
\end{equation}
where the left hand side conserves $L^p$ norms for all $1\leq p\leq \infty$.
It follows that for any $k_j\in\Z$, $1\leq j\leq 3$,
\begin{equation}
\begin{aligned}
   \partial_t (\jabr{\ell}^{k_1}\ell^{k_2}a^{k_3}\gamma) + \lambda \{\jabr{\ell}^{k_1}\ell^{k_2}a^{k_3}\gamma,\tilde{\Psi}\} = -\lambda k_3\jabr{\ell}^{k_1}\ell^{k_2}a^{k_3}\gamma \,a^{-1}\partial_\theta\tilde{\Psi},
\end{aligned}    
\end{equation}
which yields \eqref{eq:mom_bd_a} since $\eps\leq\eps_0 \leq c_0\delta$ with $c_0>0$ sufficiently small. Similarly, we have that
\begin{equation}
\begin{aligned}
   \partial_t (\jabr{\ell}^{k_1}\ell^{k_2}\theta^{k_3}\gamma) + \lambda \{\jabr{\ell}^{k_1}\ell^{k_2}\theta^{k_3}\gamma,\tilde{\Psi}\} = \lambda k_3\jabr{\ell}^{k_1}\ell^{k_2}\theta^{k_3-1}\gamma \partial_a\tilde{\Psi},
\end{aligned}    
\end{equation}
and thus
\begin{equation}
  \frac{d}{dt}\norm[L^\infty]{\jabr{\ell}^{k_1}\ell^{k_2}\theta^{k_3}\gamma}\lesssim \eps^2\delta^{-3}\frac{1}{1+t}\norm[L^\infty]{\jabr{\ell}^{k_1+\frac12}\ell^{k_2}\theta^{k_3-1}\gamma}.  
\end{equation}
Since $\eps\leq \eps_0\leq c_0 \delta^{\frac32}$ and $c_0>0$ sufficiently small, this implies that
\begin{equation}
   \norm[L^\infty]{\jabr{\ell}^{32-\frac{k}{2}} \ell^{-1} \jabr{\theta}^{k} \gamma(t)}\lesssim \eps\ln^k\jabr{t},\qquad 1\leq k\leq 32,
\end{equation}
from which \eqref{eq:mom_bd_theta} follows. The property $\supp(\gamma(t))\subset\D(\frac{\delta}{2})$ is then a consequence of Lemma \ref{lem:NonlinCharSys} provided $\eps \leq \eps_0 \leq c_0\delta^2$ with $c_0>0$ sufficiently small.
	
\textit{Uniqueness.} In the following steps we prove uniqueness. To this end, let $ \gamma_1, \, \gamma_2 $ be two Lagrangian solutions with initial data $\gamma_0$ satisfying the stated assumptions. We will denote by $ \F^1,\,  \F^2 $ the associated gravitational fields and by $ \Psi^1,\, \Psi^2 $ the corresponding gravitational potentials. To prove uniqueness it suffices to show that the induced flows $ \Phi_t^1=(\hat{\Theta}_t^1,\hat{A}_t^1,\ell) $,  $ \Phi_t^2=(\hat{\Theta}_t^2,\hat{A}_t^2,\ell) $ satisfying the characteristic system \eqref{eq:Sec4:NonlinCharSys} coincide. (Since $\ell$ is not a dynamic variable, we will sometimes abuse notation by suppressing the explicit dependence on $\ell$ and treat $\Phi_t^j$ as $(\hat{A}_t^j,\hat{\Theta}_t^j)$, $j=1,2$.) To achieve this, we will invoke a Gr\"onwall estimate for the function
\begin{equation}
	\bar{\zeta}(t) := \sup_{(\theta,a,\ell)\in \D(\delta)} \zeta(t,\theta,a,\ell),\qquad 	\zeta(t,\theta,a,\ell) := \left( \omega^2 \left| \hat{\Theta}_t^1-\hat{\Theta}_t^2 \right|^2+\omega^{-2}\left| \hat{A}_t^1-\hat{A}_t^2 \right|^2  \right)^{1/2},
\end{equation}
where
\begin{equation}
    \omega:=\left(\frac{a}{a+\jabr{l}^{\frac12}}\right)^{\frac12}.
\end{equation}
To motivate this choice, note that from \eqref{eq:chars_def} we have
that
\begin{equation}\label{eq:chars_diff}
\begin{aligned}
	\dfrac{d}{dt}\left( \Phi_t^1-\Phi_t^2\right) &= \lambda J\nabla_{(\theta,a)}\tilde{\Psi}^1\circ\Phi_t^1-\lambda J\nabla_{(\theta,a)}\tilde{\Psi}^2\circ\Phi_t^2
	\\
	&=\lambda \left[ J\nabla_{(\theta,a)}\tilde{\Psi}^1\circ\Phi_t^1-  J\nabla_{(\theta,a)}\tilde{\Psi}^1\circ\Phi_t^2 \right] + \lambda \left[  J\nabla_{(\theta,a)}\tilde{\Psi}^1\circ\Phi_t^2- J\nabla_{(\theta,a)}\tilde{\Psi}^2\circ\Phi_t^2 \right] 
    \\
    &=:\lambda I^1+\lambda I^2, \quad J := \begin{pmatrix}
			0 & -1 \\ 1 & 0
		\end{pmatrix}.
\end{aligned}    
\end{equation}    
To control $ I^1 $ we will use Lipschitz bounds for $ \nabla_{(\theta,a)}\tilde{\Psi}^1 $: as one sees from Lemma \ref{lem:BootstrapDerivGravFieldAALInf}, these hold locally in $\D(\delta)$, but can be globally rebalanced using the weight $\omega$.\footnote{More precisely, we observe that $|D^2_{(\theta,a)}\tilde{\Psi}^1|\lesssim \eps_0^2\delta^{-6}\begin{pmatrix}
			\omega^{2} & 1 \\ 1 & \omega^{-2}
\end{pmatrix}$, since $\omega^2\sim\min\{1,a\jabr{\ell}^{-\frac12}\}$}.
This is important for $ I^2 $, where we use that $ \F^1-\F^2 $ can be controlled (nonlocally) through $\bar\zeta$ via set size estimates (see \eqref{eq:Fdiff_bd}).
	
	\textit{Step 1. Comparability of actions.} We begin by observing that there exists $C>0$ such that
    \begin{equation}\label{eq:a-compare}
		\exp\left( -2C\lambda\dfrac{\varepsilon_0^2}{\delta^2} \right)  \leq \dfrac{\hat{A}_t^1}{\hat{A}_t^2} \leq \exp\left( 2C\lambda\dfrac{\varepsilon_0^2}{\delta^2} \right).
	\end{equation}
    Indeed, from \eqref{eq:chars_def} we obtain
	\begin{align*}
		\dfrac{d}{dt} \dfrac{\hat{A}^j_t}{a} = \lambda\dfrac{\hat{A}^j_t}{a} \dfrac{1}{\hat{A}^j_t}\partial_{\theta}\tilde{\Psi}(\hat{\Theta}^j_t,\hat{A}^j_t),\quad j=1,2,
	\end{align*}
	and with Lemma \ref{lem:BootstrapDerivGravFieldAALInf} a Gr\"onwall argument yields that
	\begin{equation*}
		\exp\left( -C\lambda\dfrac{\varepsilon_0^2}{\delta^2} \right)  \leq \dfrac{\hat{A}}{a} \leq \exp\left( C\lambda\dfrac{\varepsilon_0^2}{\delta^2} \right).
	\end{equation*}
	
	\textit{Step 2. Estimation of $ I^1 $.} We bound each component of $ I^1=(I^1_\theta,I^2_a) $ separately. With \eqref{eq:a-compare} we obtain
	\begin{align*}
		|I_\theta^1| &\leq \int_0^1\left( |\partial_{a}^2\tilde{\Psi}^1|(t,s\Phi_t^1+(1-s)\Phi_t^2) \, |\hat{A}_t^1-\hat{A}_t^2| + |\partial_{\theta a}^2\tilde{\Psi}^1|(t,s\Phi_t^1+(1-s)\Phi_t^2) \, |\hat{\Theta}_t^1-\hat{\Theta}_t^2| \right) \, ds
		\\
		&\lesssim \dfrac{\varepsilon^2}{\delta^6}\left( \omega^{-2}|\hat{A}_t^1-\hat{A}_t^2| + |\hat{\Theta}_t^1-\hat{\Theta}_t^2| \right),
	\end{align*}
	and similarly
	\begin{align*}
		|I_a^1| &\leq \int_0^1\left( |\partial_{\theta a}^2\tilde{\Psi}^1|(t,s\Phi_t^1+(1-s)\Phi_t^2) \, |\hat{A}_t^1-\hat{A}_t^2| + |\partial_{\theta}^2\tilde{\Psi}^1|(t,s\Phi_t^1+(1-s)\Phi_t^2) \, |\hat{\Theta}_t^1-\hat{\Theta}_t^2| \right) \, ds
		\\
		&\lesssim \dfrac{\varepsilon_0^2}{\delta^4}\left( |\hat{A}_t^1-\hat{A}_t^2| + \omega^2 |\hat{\Theta}_t^1-\hat{\Theta}_t^2| \right).
	\end{align*}
	
	\textit{Step 3. Auxiliary estimate on gravitational fields.} In order to bound $ I^2 $ we will use the following bound, valid for any $r>0$:
	\begin{equation}\label{eq:Fdiff_bd}
		|\F^1-\F^2|(t,r)\lesssim \dfrac{\varepsilon^2(1+t) \, \min\{\sqrt{r},1\}}{1+r^2}\bar{\zeta}(t).
	\end{equation}
	To see this, we write
	\begin{align*}
		|\F^1-\F^2|(t,r)= \dfrac{1}{r^2}\left| \int_\R\int_0^\infty\int_0^\infty \left[ \ind_{\{\tilde{R}\leq r\}}\gamma_1^2(t)-\ind_{\{\tilde{R}\leq r\}}\gamma_2^2(t)\right] \, d\theta da d\ell\right|  \lesssim \dfrac{1}{r^2}\int_\R\int_0^\infty\int_0^\infty \ind_{D(t)}\gamma_0^2\, d\theta da d\ell,
	\end{align*}
	where we defined
	\begin{align*}
		D(t) :=& D_1(t)\cup D_2(t),
		\\
		D_1(t) :=& \left\lbrace\tilde{R}\circ\Phi_t^1\leq r\leq \tilde{R}\circ\Phi_t^2 \right\rbrace \cap\D(\delta),
		\\
		D_2(t) := &\left\lbrace\tilde{R}\circ\Phi_t^2\leq r\leq \tilde{R}\circ\Phi_t^1 \right\rbrace \cap \D(\delta).
	\end{align*}
	By symmetry it suffices to consider $ D_1(t) $, which we further decompose as $ D_1(t)\subset D_{11}(t)\cup D_{12}(t) $, where
	\begin{align*}
		D_{11}(t) &=  \left\lbrace R\left( \hat{\Theta}_t^1+t\hat{A}_t^1,\hat{A}_t^1\right) \leq r\leq R\left( \hat{\Theta}_t^2+t\hat{A}_t^2,\hat{A}_t^1\right) \right\rbrace \cap\D(\delta),
		\\
		D_{12}(t) &=  \left\lbrace R\left( \hat{\Theta}_t^2+t\hat{A}_t^2,\hat{A}_t^1\right) \leq r\leq R\left( \hat{\Theta}_t^2+t\hat{A}_t^2,\hat{A}_t^2\right) \right\rbrace \cap\D(\delta).
	\end{align*}
	Accordingly, let us define the corresponding integrals
	\begin{align*}
		J_{11}= \dfrac{1}{r^2}\int_\R\int_0^\infty\int_0^\infty \ind_{D_{11}(t)}\gamma_0^2\, d\theta da d\ell, \quad J_{12}= \dfrac{1}{r^2}\int_\R\int_0^\infty\int_0^\infty \ind_{D_{12}(t)}\gamma_0^2\, d\theta da d\ell.
	\end{align*}
	For $ J_{11} $, by definition of $R$ (see \eqref{eq:Sec2:PropAAVar1}) we have
	\begin{align*}
		D_{11}(t) =& \left\lbrace |\hat{\Theta}_t^1+t\hat{A}_t^1|\leq\xi_*(r,\hat{A}_t^1,\ell)\leq |\hat{\Theta}_t^2+t\hat{A}_t^2| \right\rbrace\cap \D(\delta)
		\\
		=& \left\lbrace 0\leq\xi_*(r,\hat{A}_t^1,\ell)-|\hat{\Theta}_t^1+t\hat{A}_t^1|\leq |\hat{\Theta}_t^2+t\hat{A}_t^2|-|\hat{\Theta}_t^1+t\hat{A}_t^1| \right\rbrace \cap \D(\delta),
	\end{align*}	
	where
	\begin{align*}
		\xi_*(r,\hat{A}_t^1,\ell) := p(\hat{A}_t^1,\ell)G_{\kappa(\hat{A}_t^1,\ell)}\left(\dfrac{r}{p(\hat{A}_t^1,\ell)}+\kappa(\hat{A}_t^1,\ell)\right).
	\end{align*}
	Noting that by \eqref{eq:a-compare}
	\begin{align*}
		|\hat{\Theta}_t^2+t\hat{A}_t^2|-|\hat{\Theta}_t^1+t\hat{A}_t^1| \lesssim (1+t)\zeta(t,\theta,a,\ell) \left( 1+\dfrac{\hat{A}_t^1}{\jabr{\sqrt{\ell}}}+\dfrac{\jabr{\sqrt{\ell}}}{\hat{A}_t^1} \right) \lesssim (1+t)\bar{\zeta}(t) \left( 1+\dfrac{\hat{A}_t^1}{\jabr{\sqrt{\ell}}}+\dfrac{\jabr{\sqrt{\ell}}}{\hat{A}_t^1} \right),
	\end{align*}
	applying the inverse $ (\Phi_t^1)^{-1} $ yields
	\begin{align*}
		D_{11}(t) \subset (\Phi_t^1)^{-1}\left( \left\lbrace (\theta, a,\ell) \, : \,  \xi_* -|\theta+at| \lesssim (1+t)\jabr{a+a^{-1}}\jabr{\ell}\bar{\zeta}(t) \right\rbrace \cap\D(\delta) \right),
	\end{align*}
	so that
	\begin{align*}
		J_{11}\lesssim \dfrac{1}{r^2}\int_\R\int_0^\infty\int_0^\infty \ind_{\{r_0\leq r\}} \ind_{\{0\leq \xi_*(r,a,\ell)-|\theta+at|\lesssim(1+t)\jabr{a+a^{-1}}\jabr{\ell}\bar{\zeta}(t) \}}(\gamma_1(t))^2\, d\theta da d\ell.
	\end{align*}
	Changing variables $ \theta\mapsto\xi=\theta+at $ and invoking the moment bounds \eqref{eq:mom_bd_a} for $\gamma_1$ gives
	\begin{align*}
		J_{11} &\lesssim \dfrac{\varepsilon^2}{r^2}\int_\R\int_0^\infty\int_0^\infty \ind_{\{r_0\leq r\}}\ind_{\{0\leq \xi_*(r,a,\ell)-|\xi|\lesssim(1+t)\jabr{a+a^{-1}}\jabr{\ell}\bar{\zeta}(t) \}}\, \left(a+a^{-1}\right)^{-8}\jabr{\ell}^{-8}\ell^2\,  d\xi da d\ell\\
        &\lesssim (1+t)\bar{\zeta}(t)\dfrac{\varepsilon^2}{r^2} \int_0^\infty\int_0^\infty \ind_{\{r_0\leq r\}}\, \left(a+a^{-1}\right)^{-7}\jabr{\ell}^{-7}\ell^2\, da d\ell,
	\end{align*}
    since the integral in $ \xi $ is restricted to a set with size of order $ \bar{\zeta}(t) $. Using the fact that $ r_0(a,\ell)\gtrsim \min\{a^{-2}, \ell\} $ one obtains by distinguishing the cases $ r\leq 1 $ and $ r\geq 1 $ that
	 \begin{equation}\label{eq:J11_bd}
	 	J_{11} \lesssim \dfrac{\varepsilon^2(1+t) \, \min\{\sqrt{r}, 1\}}{1+r^2}\bar{\zeta}(t).
	 \end{equation}
Next we treat the integral $ J_{12} $. To this end, we observe that by Lemma \ref{lem:BoundActionDerivR}, for fixed $\xi\in\R$, $\ell>0$ the function $ a\mapsto R(\xi,a,\ell) $ is strictly decreasing. Hence there exists $ \tilde{A}=\tilde{A}(r,\hat{\Theta}_t^2+t\hat{A}_t^2,\ell) \in [\hat{A}_t^2,\hat{A}_t^1]$ such that $r=R\left(  \hat{\Theta}_t^2+t\hat{A}_t^2,\tilde{A}\right)$, and thus
 	\begin{align*}
 		D_{12}(t)  =  \left\lbrace R\left( \hat{\Theta}_t^2+t\hat{A}_t^2,\hat{A}_t^1\right) \leq R\left(  \hat{\Theta}_t^2+t\hat{A}_t^2 , \tilde{A}(r,t,\Phi_t^2) \right)  \leq R\left( \hat{\Theta}_t^2+t\hat{A}_t^2,\hat{A}_t^2\right) \right\rbrace \cap\D(\delta).
 	\end{align*}
 Moreover, by \eqref{eq:a-compare} there exists $C_0\geq 1$ such that
 	\begin{align*}
 		\dfrac{\hat{A}_t^1}{a},\,  \dfrac{\hat{A}_t^2}{a},\, \dfrac{\tilde{A}}{a} \in \left(\dfrac{1}{C_0},C_0\right),
 	\end{align*}
 and we infer that on the set $ D_{12}(t) $
 	\begin{align*}
 		-\left| \hat{A}_t^2-\hat{A}_t^1 \right|\,  \max_{[a/C_0,C_0a]}|\partial_aR(\hat{\Theta}_t^2+t\hat{A}_t^2,\cdot,\ell)| &\leq R\left(  \hat{\Theta}_t^2+t\hat{A}_t^2 ,\hat{A}_t^1 \right)  - R\left( \hat{\Theta}_t^2+t\hat{A}_t^2,\hat{A}_t^2\right) 
 		\\
 		& \leq R\left( \hat{\Theta}_t^2+t\hat{A}_t^2,\tilde{A}\right) -R\left( \hat{\Theta}_t^2+t\hat{A}_t^2,\hat{A}_t^2\right)
 		\\
 		& \leq -\left|\tilde{A}-\hat{A}_t^2\right|\,  \min_{[a/C_0,C_0a]}|\partial_aR(\hat{\Theta}_t^2+t\hat{A}_t^2,\cdot,\ell)|.
 	\end{align*}
 	As a consequence of Lemma \ref{lem:BoundActionDerivR} we get
 	\begin{align*}
 		\left|\tilde{A}(r,\hat{\Theta}_t^2+t\hat{A}_t^2,\ell)-\hat{A}_t^2\right| \lesssim \left| \hat{A}_t^2-\hat{A}_t^1 \right| \leq \bar{\zeta}(t),
 	\end{align*}
 	and thus
 	\begin{align*}
 		D_{12}(t) \subset \left\lbrace \left| \tilde{A}(r,\hat{\Theta}_t^2+t\hat{A}_t^2,\ell)-\hat{A}_t^2 \right| \lesssim \bar{\zeta}(t) \right\rbrace \cap\D(\delta).
 	\end{align*}
 	This implies after application of $ (\Phi_t^2)^{-1} $ and using the change of variables $ \theta\mapsto \xi=\theta+at $ that
 	\begin{align*}
 		J_{12} &\lesssim \dfrac{\varepsilon^2}{r^2} \int_\R\int_0^\infty\int_0^\infty \ind_{\{r_0\leq r\}}\ind_{\{| \tilde{A}(r,\xi,\ell)-a |\lesssim \bar{\zeta}(t) \}}\, \left(a+a^{-1}\right)^{-8}\jabr{\ell}^{-8}\ell^2\,  d\xi da d\ell\\
        &\lesssim \varepsilon^2\dfrac{\min\{\sqrt{r},1\}}{1+r^2}\int_\R\int_0^\infty\int_0^\infty \ind_{\{| \tilde{A}(r,\xi,\ell)-a |\lesssim \bar{\zeta}(t) \}}\, \left(a+a^{-1}\right)^{-3}\jabr{\ell}^{-8}\ell^{-1/2}\,  d\xi da d\ell
 		\\
 		&\lesssim \dfrac{\varepsilon_0^2(1+t) \, \min\{\sqrt{r},1\}}{1+r^2}\bar{\zeta}(t),
 	\end{align*}
    having used again that $ r_0\gtrsim \min(a^{-2},\ell) $. Together with \eqref{eq:J11_bd}, this implies the assertion \eqref{eq:Fdiff_bd}.

	\textit{Step 4. Estimation of $ I^2 $.} In order to account for the weights in $\zeta$ (see Step 5 for the details), we will establish the following bounds:
	\begin{equation}\label{eq:I2_bds}
		\omega |I^2_\theta| \lesssim \varepsilon^2\delta^{-3}(1+t)^2\bar{\zeta}(t),
		\quad
		\omega^{-1} |I^2_a| \lesssim\varepsilon^2\delta^{-2}(1+t)\bar{\zeta}(t).
	\end{equation}
	To see the first inequality, note that by $\omega\lesssim1$ and \eqref{eq:Fdiff_bd}
	\begin{align*}
		\omega |I^2_\theta| \lesssim |\partial_a\tilde{\Psi}^1-\partial_a\tilde{\Psi}^2|\circ\Phi_t^2 \lesssim \varepsilon_0^2(1+t)\bar{\zeta}(t)\left( \dfrac{\min\{1,\tilde{R}^{1/2}\}}{1+\tilde{R}^2} |\partial_{a}\tilde{R}| \right) \circ \Phi_t^2.
	\end{align*}
    Here we have by Lemma \ref{lem:EstDynRadNoBulk} and \eqref{eq:a-compare} that
	\begin{align*}
		\dfrac{\min\{1,\tilde{R}^{1/2}\}}{1+\tilde{R}^2} |\partial_{a}\tilde{R}| \lesssim \dfrac{\min\{1,\tilde{R}^{1/2}\}}{1+\tilde{R}^2} \left( t\jabr{\dfrac{p}{\tilde{R}}}^{1/2}+\dfrac{p}{a}\ln\jabr{\dfrac{\tilde{R}}{p}} \right).
	\end{align*}
	When $ \ell\leq 1 $ we have $ p\lesssim 1+1/a^2 $ as well as $ a\gtrsim \delta $, and we can bound
	\begin{align*}
		\dfrac{\min\{1,\tilde{R}^{1/2}\}}{1+\tilde{R}^2} |\partial_{a}\tilde{R}| \lesssim(1+t)\dfrac{\min\{1,\tilde{R}^{1/2}\}}{1+\tilde{R}^{3/2}} \jabr{\dfrac{1}{a}} \jabr{\dfrac{p}{\tilde{R}}}^{1/2}\jabr{p}^{1/2} \lesssim (1+t)\jabr{\dfrac{1}{a}}\jabr{p} \lesssim (1+t)\dfrac{1}{\delta^3},
	\end{align*}
	having distinguished the cases $ \tilde{R}\leq 1 $ and $ \tilde{R}\geq1 $. On the other hand, for $ \ell\geq 1 $ we have $ p/\tilde{R}\lesssim 1/\delta^2 $ due to $ a\gtrsim \delta/\sqrt{\ell} $, and hence
	\begin{align*}
		\dfrac{\min\{1,\tilde{R}^{1/2}\}}{1+\tilde{R}^2} |\partial_{a}\tilde{R}| \lesssim (1+t)\dfrac{1}{1+\tilde{R}} \jabr{\dfrac{1}{a}}\jabr{\dfrac{p}{\tilde{R}}} \lesssim (1+t)\dfrac{1}{\delta^2} \dfrac{1}{r_0} \jabr{\dfrac{1}{a}} \lesssim (1+t)\dfrac{1}{\delta^2} \left( 1+ \dfrac{\jabr{a\sqrt{\ell}}}{a\ell}\right) \lesssim (1+t)\dfrac{1}{\delta^3}.
	\end{align*}
	For $ I^2_a $ we have with \eqref{eq:a-compare} that
	\begin{align*}
		\omega^{-1}|I^2_a| \lesssim \dfrac{\jabr{\ell}^{\frac12}}{\hat{A}_t^2}|\partial_\theta\tilde{\Psi}^1-\partial_\theta\tilde{\Psi}^2|\circ\Phi_t^2 \lesssim \varepsilon_0^2(1+t)\bar{\zeta}(t)\left( \dfrac{\jabr{\ell}^{\frac12}}{a}\dfrac{\min\{1,\tilde{R}^{1/2}\}}{1+\tilde{R}^2} |\partial_{\theta}\tilde{R}| \right) \circ \Phi_t^2.
	\end{align*}
	We have with Lemma \ref{lem:EstDynRadNoBulk} that
	\begin{align*}
		\dfrac{\jabr{\ell}^{\frac12}}{a}\dfrac{\min\{1,\tilde{R}^{1/2}\}}{1+\tilde{R}^2} |\partial_{\theta}\tilde{R}|  \lesssim 	\dfrac{\jabr{\ell}^{\frac12}}{a}\dfrac{\min\{1,\tilde{R}^{1/2}\}}{1+\tilde{R}^2} \jabr{\dfrac{p}{\tilde{R}}}^{1/2}.
	\end{align*}
	For $ \ell\leq 1 $ we argue as before, distinguishing the cases $ \tilde{R}\leq1 $ and $ \tilde{R}\geq 1 $, and thereby obtain that
	\begin{align*}
		\dfrac{\jabr{\ell}^{\frac12}}{a}\dfrac{\min\{1,\tilde{R}^{1/2}\}}{1+\tilde{R}^2} \jabr{\dfrac{p}{\tilde{R}}}^{1/2} \lesssim \dfrac{1}{\delta^2},
	\end{align*}
	while for $ \ell\geq 1 $ we have
	\begin{align*}
			\dfrac{\jabr{\ell}^{\frac12}}{a}\dfrac{\min\{1,\tilde{R}^{1/2}\}}{1+\tilde{R}^2} \jabr{\dfrac{p}{\tilde{R}}}^{1/2} \lesssim \dfrac{1}{\delta} \dfrac{1}{r_0}\dfrac{\jabr{\ell}^{\frac12}}{a} \lesssim \dfrac{1}{\delta}\dfrac{\jabr{a\sqrt{\ell}}}{a\ell} \jabr{\ell}^{\frac12} \lesssim \dfrac{1}{\delta^2}.
	\end{align*}
	This yields \eqref{eq:I2_bds}.
	
	\textit{Step 5. Proof of uniqueness.} From \eqref{eq:chars_diff} and the bounds in Step 2 we infer that
	\begin{align*}
		\dfrac{d}{dt} \left( \hat{\Theta}_t^1-\hat{\Theta}_t^2 \right)^2 &\lesssim  \dfrac{\varepsilon^2}{\delta^6}\left( \omega^{-2}|\hat{A}_t^1-\hat{A}_t^2||\hat{\Theta}_t^1-\hat{\Theta}_t^2| + |\hat{\Theta}_t^1-\hat{\Theta}_t^2|^2 \right) +  |\hat{\Theta}_t^1-\hat{\Theta}_t^2||I^2_\theta|,
		\\
		\dfrac{d}{dt} \left( \hat{A}_t^1-\hat{A}_t^2 \right)^2 &\lesssim \dfrac{\varepsilon^2}{\delta^4}\left( |\hat{A}_t^1-\hat{A}_t^2|^2 + \omega^2 |\hat{\Theta}_t^1-\hat{\Theta}_t^2||\hat{A}_t^1-\hat{A}_t^2| \right) + |\hat{A}_t^1-\hat{A}_t^2||I^2_a|,
	\end{align*}
	and thus
	\begin{align*}
		\dfrac{d}{dt} \left[ \omega^2\left( \hat{\Theta}_t^1-\hat{\Theta}_t^2 \right)^2\right]  &\lesssim  \dfrac{\varepsilon^2}{\delta^6}\left( \omega^{-1}|\hat{A}_t^1-\hat{A}_t^2|\, \omega|\hat{\Theta}_t^1-\hat{\Theta}_t^2| + \omega^2|\hat{\Theta}_t^1-\hat{\Theta}_t^2|^2 \right) +  \omega|\hat{\Theta}_t^1-\hat{\Theta}_t^2|\omega|I^2_\theta|,
		\\
		\dfrac{d}{dt} \left[ \omega^{-2}\left( \hat{A}_t^1-\hat{A}_t^2 \right)^2\right]  &\lesssim \dfrac{\varepsilon^2}{\delta^4}\left( \omega^{-2}|\hat{A}_t^1-\hat{A}_t^2|^2 + \omega |\hat{\Theta}_t^1-\hat{\Theta}_t^2|\, \omega^{-1}|\hat{A}_t^1-\hat{A}_t^2| \right) + \omega^{-1}|\hat{A}_t^1-\hat{A}_t^2|\omega^{-1}|I^2_a|.
	\end{align*}
	Summing these inequalities and using \eqref{eq:I2_bds} gives
	\begin{align*}
		\dfrac{d}{dt}\zeta(t,\theta,a,\ell)^2 \lesssim \dfrac{\varepsilon^2}{\delta^6}\zeta(t,\theta,a,\ell)^2 + \dfrac{\varepsilon^4}{\delta^6}(1+t)^4\bar{\zeta}(t)^2.
	\end{align*}
	Integrating this and taking the supremum over $ (\theta,a,\ell)\in\D(\delta) $ yields with Gr\"onwall's lemma that $ \bar{\zeta}(t)\leq \bar{\zeta}(0)=0 $. As a consequence, $ \Phi_t^1(t)=\Phi_t^2(t) $ on $ \D(\delta) $, and thus $\gamma_1(t) = \Phi_t^1(t)_\# \gamma_0 = \Phi_t^2(t)_\# \gamma_0 = \gamma_2(t)$. This concludes the proof. 
\end{proof}

Next we turn to strong solutions and show that (suitably weighted) derivatives can be propagated. For simplicity and in order not to be too demanding in terms of the topology, we have chosen to propagate derivatives in an $L^2$-based framework. (As one can easily see from the proof, other integrability orders and higher order derivatives could be propagated as well.) This yields the global existence of unique, strong solutions.

\begin{thm}\label{thm:WellPosednessRegularSol}
There exists $c_1>0$ such that the following holds. Let $\gamma\in C([0,\infty),L^2\cap L^\infty)$ be a global solution to \eqref{eq:Sec4:VPeq} as given by Theorem \ref{thm:WellPosedLagrangianSol}.  Assume also that $ \gamma_0\in \SobH^1$ with
	\begin{align*}
		\norm[L^2]{\omega^{-1}\chi\partial_\theta\gamma_0} + \norm[L^2]{\omega\chi\partial_a\gamma_0}+
		\norm[L^2]{\omega^{-1}\jabr{\theta}^{k_5}\partial_\theta\gamma_0} + \norm[L^2]{\omega\jabr{\theta}^{k_5}\partial_a\gamma_0}\leq \eps\leq c_1\delta^3,
	\end{align*}
where 
	\begin{align*}
		\omega=\left(\frac{a}{a+\jabr{l}^{\frac12}}\right)^{\frac12}, \quad \chi = \jabr{\ell}^{k_1} \jabr{\ell^{-1}}^{k_2} \jabr{a}^{k_3} \jabr{a^{-1}}^{k_4},
	\end{align*}
for $k_i\in\mathbb{N}_0$, $1\leq i\leq 5$, with $k_1\geq k_5/2$.
Then the solution is strong $ \gamma\in C^1((0,\infty);L^2)\cap C([0,\infty);\SobH^1) $ and satisfies for all $t\geq0$ the bounds
	\begin{align}
		\norm[L^2]{\omega^{-1}\chi\partial_\theta\gamma(t)}&\lesssim \varepsilon, \quad \norm[L^2]{\omega\chi\partial_a\gamma(t)}\lesssim \varepsilon \ln^{33}\jabr{t},\label{eq:deriv_bd1}
		\\
		\norm[L^2]{\omega^{-1}\jabr{\ell}^{\frac{k_5-k}{2}}\jabr{\theta}^{k}\partial_\theta\gamma(t)} &\lesssim \eps\ln^k\jabr{t}, \quad \norm[L^2]{\omega\jabr{\ell}^{\frac{k_5-k}{2}}\jabr{\theta}^{k}\partial_a\gamma(t)}\lesssim \eps\ln^{k+33}\jabr{t},\quad 1\leq k\leq k_5.\label{eq:deriv_bd2}
	\end{align}
\end{thm}

\begin{rem}
    The condition $k_1\geq k_5/2$ is imposed since each moment in $\theta$ requires an extra moment $\jabr{\ell}^{1/2}$ -- see the bound of $\partial_a\tilde{\Psi}$ in Lemma \ref{lem:BootstrapDerivGravFieldAALInf} resp.\ \eqref{eq:mom-bracket-daPsi} -- leading also to the combination of moments $\jabr{\ell}^{\frac{k_5-k}{2}}\jabr{\theta}^{k}$ in \eqref{eq:deriv_bd2}. We note also that the weight $\omega$ has already appeared in the proof of uniqueness in Theorem \ref{thm:WellPosedLagrangianSol}, with the same purpose: to rebalance the bounds for $D^2_{(\theta,a)}\tilde{\Psi}$.
\end{rem}

\begin{proof}[Proof of Theorem \ref{thm:WellPosednessRegularSol}]
It suffices to give the relevant a priori estimate. Differentiating the equation for $\gamma$ yields
\begin{equation*}
	\partial_t\partial_\nu\gamma + \lambda \{\partial_\nu\gamma,\tilde{\Psi}\} = -\lambda\{\gamma,\partial_\nu\tilde{\Psi}\},\quad \nu\in\{a,\theta\},
\end{equation*}
so that for any scalar function $f(\theta,a,\ell)$ one has
\begin{equation}\label{eq:mom-deriv-eqs}
	\partial_t (f\partial_\nu\gamma) + \lambda \{f\partial_\nu\gamma,\tilde{\Psi}\} = \lambda \partial_\nu\gamma\{f,\tilde{\Psi}\} -\lambda f\{\gamma,\partial_\nu\tilde{\Psi}\}.
\end{equation}
Moreover, the moment bounds \eqref{eq:mom_bd_a} and \eqref{eq:mom_bd_theta} in Theorem \ref{thm:WellPosedLagrangianSol} imply that there exists $C>0$ such that
\begin{equation}\label{eq:Mq_bds_derivs}
   M_2(\gamma)\leq C\varepsilon^2, \quad M'_2(\gamma)\leq C\varepsilon^2\ln^{32}\jabr{t}.
\end{equation}
Invoking these with Lemma \ref{lem:BootstrapDerivGravFieldAALInf} and observing that
\begin{align*}
	\partial_\theta\chi=\partial_\theta\omega=0,\quad |\partial_{a}\chi| \lesssim \jabr{\dfrac{1}{a}}\chi,\quad |\partial_{a}\omega| \lesssim \jabr{\dfrac{1}{a}}\omega,\quad |\partial_{a}\omega^{-1}| \lesssim \jabr{\dfrac{1}{a}}\omega^{-1},
\end{align*}
and hence
\begin{align*}
	\left| \left\lbrace \omega^\sigma\chi,\tilde{\Psi} \right\rbrace \right| \lesssim \jabr{\dfrac{1}{a}}\omega^\sigma\chi \left| \partial_{\theta}\tilde{\Psi} \right| \lesssim \dfrac{\eps^2}{\delta^2}\dfrac{\omega^\sigma\chi}{(1+t)^{3/2}},\quad\sigma\in\{-1,+1\},
\end{align*}
we obtain from \eqref{eq:mom-deriv-eqs} (with $f=\omega^{-1}\chi$ resp.\ $f=\omega\chi$) that
\begin{align*}
	\partial_t\left(  (\omega^{-1}\chi\partial_{\theta}\gamma)^2 \right) + \left\lbrace  (\omega^{-1}\chi\partial_{\theta}\gamma)^2,\tilde{\Psi}\right\rbrace  &\lesssim \dfrac{\eps^2}{\delta^4}\left[ \dfrac{\ln^{32}\jabr{t}  }{(1+t)^{3/2}}(\omega^{-1}\chi\partial_{\theta}\gamma)^2+ \dfrac{\ln^{32}\jabr{t}}{(1+t)^{2}}\chi |\partial_{\theta}\gamma|\chi|\partial_{a}\gamma|\right],
	\\
	\partial_t\left( (\omega\chi\partial_{a}\gamma)^2 \right) + \left\lbrace (\omega\chi\partial_{a}\gamma)^2,\tilde{\Psi}\right\rbrace  &\lesssim \dfrac{\eps^2}{\delta^6}\left[ \dfrac{\ln^{32}\jabr{t}}{(1+t)}\chi |\partial_{\theta}\gamma|\chi|\partial_{a}\gamma|+\dfrac{\ln^{32}\jabr{t} }{(1+t)^{3/2}}(\omega\chi\partial_{a}\gamma)^2 \right].
\end{align*}
Abbreviating $ \Gamma_\theta=\omega^{-1}\chi\partial_{\theta}\gamma $ and $ \Gamma_a=\omega\chi\partial_{a}\gamma $ and integrating we obtain that
\begin{equation}
\begin{aligned}
 \frac{d}{dt}\norm[L^2]{\Gamma_{\theta}(t)}&\lesssim \eps^2\delta^{-4} \left[ \dfrac{\ln^{32}\jabr{t}  }{(1+t)^{3/2}}\norm[L^2]{\Gamma_{\theta}(t)}+ \dfrac{\ln^{32}\jabr{t}}{(1+t)^{2}}\norm[L^2]{\Gamma_{a}(t)}\right],\\
 \frac{d}{dt}\norm[L^2]{\Gamma_{a}(t)}&\lesssim \eps^2\delta^{-6}\left[ \dfrac{\ln^{32}\jabr{t}}{(1+t)}\norm[L^2]{\Gamma_{\theta}(t)}+\dfrac{\ln^{32}\jabr{t} }{(1+t)^{3/2}}\norm[L^2]{\Gamma_{a}(t)} \right],
\end{aligned} 
\end{equation}
which gives the bounds \eqref{eq:deriv_bd1} upon integration in time and bootstrapping for $\eps\leq c_1\delta^3$ with $c_1>0$ sufficiently small.

For additional $\theta$-moments we argue similarly, using the bounds for $\tilde\Psi$ from \eqref{eq:Mq_bds_derivs} and Lemma \ref{lem:BootstrapDerivGravFieldAALInf}. Since for $k\in\N$
\begin{equation}\label{eq:mom-bracket-daPsi}
 \left| \left\lbrace \omega^\sigma\theta^k,\tilde{\Psi} \right\rbrace \right| \lesssim \jabr{\dfrac{1}{a}}\omega^\sigma|\theta|^k \left| \partial_{\theta}\tilde{\Psi} \right| +k\omega^\sigma|\theta|^{k-1}\left|\partial_a\tilde{\Psi} \right|\lesssim \dfrac{\eps^2}{\delta^2}\dfrac{\omega^\sigma|\theta|^k}{(1+t)^{3/2}}+k\dfrac{\eps^2}{\delta^3}\dfrac{\jabr{\ell}^{\frac12}\omega^\sigma|\theta|^{k-1}}{1+t},\quad\sigma\in\{-1,+1\}, 
\end{equation}
from \eqref{eq:mom-deriv-eqs} with $f=\omega^{-1}\theta^k$ resp.\ $f=\omega\theta^k$ for some $k\in\N$, we obtain
\begin{align*}
	\partial_t\left(  (\omega^{-1}\theta^k\partial_{\theta}\gamma)^2 \right) + \left\lbrace  (\omega^{-1}\theta^k\partial_{\theta}\gamma)^2,\tilde{\Psi}\right\rbrace  &\lesssim \dfrac{\eps^2}{\delta^4}\left[ \dfrac{\ln^{32}\jabr{t}  }{(1+t)^{3/2}}(\omega^{-1}\theta^k\partial_{\theta}\gamma)^2+ \dfrac{\ln^{32}\jabr{t}}{(1+t)^{2}}\theta^k |\partial_{\theta}\gamma|\theta^k|\partial_{a}\gamma|\right]\\
    &\qquad +k\dfrac{\eps^2}{\delta^3}\dfrac{1}{1+t}\jabr{\ell}^{\frac12}(\omega^{-1}|\theta|^{k-1}\partial_\theta\gamma) (\omega^{-1}\theta^k\partial_{\theta}\gamma),
    \\
    \partial_t\left( (\omega\theta^k\partial_{a}\gamma)^2 \right) + \left\lbrace (\omega\theta^k\partial_{a}\gamma)^2,\tilde{\Psi}\right\rbrace  &\lesssim \dfrac{\eps^2}{\delta^6}\left[ \dfrac{\ln^{32}\jabr{t}}{(1+t)}\theta^k |\partial_{\theta}\gamma|\theta^k|\partial_{a}\gamma|+\dfrac{\ln^{32}\jabr{t} }{(1+t)^{3/2}}(\omega\theta^k\partial_{a}\gamma)^2 \right]\\
    &\qquad +k\dfrac{\eps^2}{\delta^3}\dfrac{1}{1+t}\jabr{\ell}^{\frac12}(\omega|\theta|^{k-1}\partial_a\gamma) (\omega\theta^k\partial_{a}\gamma).
\end{align*}    
Using additionally the commutation with $\ell$ and abbreviating $ \Sigma^{j,k}_\theta=\omega^{-1}\jabr{\ell}^j\theta^k\partial_{\theta}\gamma $, $ \Sigma^{j,k}_a=\omega\jabr{\ell}^j\theta^k\partial_{a}\gamma $, we obtain upon integration that 
\begin{equation*}
\begin{aligned}
 \frac{d}{dt}\norm[L^2]{\Sigma^{j,k}_{\theta}(t)}&\lesssim \eps^2\delta^{-4} \left[ \dfrac{\ln^{32}\jabr{t}  }{(1+t)^{3/2}}\norm[L^2]{\Sigma^{j,k}_{\theta}(t)}+ \dfrac{\ln^{32}\jabr{t}}{(1+t)^{2}}\norm[L^2]{\Sigma^{j,k}_{a}(t)}\right]+k\eps^2\delta^{-3}\frac{1}{1+t}\norm[L^2]{\Sigma^{j+\frac12,k-1}_{\theta}(t)},\\
 \frac{d}{dt}\norm[L^2]{\Sigma^{j,k}_{a}(t)}&\lesssim \eps^2\delta^{-6}\left[ \dfrac{\ln^{32}\jabr{t}}{(1+t)}\norm[L^2]{\Sigma^{j,k}_{\theta}(t)}+\dfrac{\ln^{32}\jabr{t} }{(1+t)^{3/2}}\norm[L^2]{\Sigma^{j,k}_{a}(t)} \right]+k\eps^2\delta^{-3}\frac{1}{1+t}\norm[L^2]{\Sigma^{j+\frac12,k-1}_{a}(t)}.
\end{aligned} 
\end{equation*}
Since for $0\leq j\leq 2k_1$ we have that
\begin{equation}
 \norm[L^2]{\Sigma^{j,0}_{\theta}(t)}\lesssim \norm[L^2]{\Gamma_{\theta}(t)}\lesssim \eps,\qquad \norm[L^2]{\Sigma^{j,0}_{a}(t)}\lesssim \norm[L^2]{\Gamma_{a}(t)}\lesssim \eps\ln^{33}\jabr{t},
\end{equation}
we obtain from this hierarchy the claimed bounds \eqref{eq:deriv_bd2} via a bootstrap for $\eps\leq c_1\delta^3$ with $c_1>0$ sufficiently small.
\end{proof}

We end this section with a short remark. 
\begin{rem}[Existence of weak solutions]
 As a matter of fact in order to ensure that the characteristic flow remains in the domain $ \D(\delta/2) $ for all times $ t\geq0 $, so that the action-angle variables can be used, the induced gravitational field needs to be merely bounded, see Lemma \ref{lem:NonlinCharSys}. This bound can be proved using weaker assumptions in terms of moments in the space $ L^2\cap L^\infty $, see Lemma \ref{lem:EstimatesGravFieldLInf}. This observation allows to construct global in time weak (distributional) solutions $ \gamma\in C([0,\infty);\mathcal{D}') $, provided the initial data $ \gamma_0 $ satisfy a suitable smallness assumption for the moments
 \begin{align*}
	\norm[L^\infty]{\left(a+a^{-1} \right)^{4} \jabr{a}a^{-4}\jabr{\ell}^{5/2}\jabr{\theta}^{3}\gamma_0}.
 \end{align*}
 We highlight that no vanishing condition for $\ell\to 0$ is thus imposed. The proof uses standard arguments, e.g.\ by approximating with solutions as constructed in Theorem \ref{thm:WellPosedLagrangianSol}, but does not provide any uniqueness.
\end{rem}

\subsection{Asymptotic behavior: modified scattering}\label{subsec:LongtimeBehavior}
In this section we prove that solutions constructed in Theorems \ref{thm:WellPosedLagrangianSol} resp.\ \ref{thm:WellPosednessRegularSol} exhibit a modified scattering dynamic for large times. More precisely, we identify a logarithmic correction to the linear characteristics, along which the nonlinear solutions converge as $t\to\infty$. For the Lagrangian solutions of Theorem \ref{thm:WellPosedLagrangianSol} this convergence is weak in the sense of distributions (see Theorem \ref{thm:ModifiedScatteringLagrangianSol}), whereas for the strong solutions of Theorem \ref{thm:WellPosednessRegularSol} the convergence is strong in $L^2$ (see Theorem \ref{thm:ModifiedScattering}).

\medskip
\paragraph{\bf Long-time behavior: Lagrangian solutions.} 
To understand the long-time behavior of Lagrangian solutions, we begin by studying the asymptotics of the associated characteristic flow.

\begin{pro}\label{pro:ConvergenceCharacFlow}
	Let $ \gamma $ be a solution as given in Theorem \ref{thm:WellPosedLagrangianSol}, satisfying in particular \eqref{eq:mom_bd_a}--\eqref{eq:mom_bd_theta}. Let $(\hat{\Theta}_t,\hat{A}_t)$ denote the associated characteristic flow, as given by \eqref{eq:chars_def}.
    Then, if $ \varepsilon\leq \tilde{c}_0\delta^3 $ for $\tilde{c}_0>0$ sufficiently small, the following statements hold.
	\begin{enumerate}[(i)]
		\item\label{it:Aconv} There is a continuous map $ \hat{A}_\infty=\hat{A}_\infty(\theta,a,\ell):\D(\delta)\to [\delta/2,\infty) $ such that
		\begin{equation}\label{eq:Aconv}
			\norm[L^\infty(\D(\delta))]{\hat{A}_t-\hat{A}_\infty} \lesssim \dfrac{\varepsilon^2}{\delta^2}\dfrac{1}{(1+t)^{1/2}}.
		\end{equation}
		\item\label{it:Finfty} There is a locally Lipschitz continuous map $ \F_{\infty}=\F_{\infty}(a):\R_+\to (-\infty,0] $ such that for all $ a>0 $
		\begin{equation}\label{eq:Finfty_conv}
			|t^2\F_{\eff}(t,at)-\F_\infty(a)| \lesssim \dfrac{\eps^2}{\delta^2}\dfrac{1}{a^2}\dfrac{1}{(1+t)^{1/2}}.
		\end{equation}
		 Furthermore, we have for all $a>0$
		\begin{equation}\label{eq:Finfty_bd}
			|\F_\infty(a)|\lesssim \dfrac{\varepsilon^2}{a^2}, \quad |\F'_\infty(a)|\lesssim \dfrac{\varepsilon^2}{a^3}+\dfrac{\varepsilon^2}{a^2}.
		\end{equation}
	\end{enumerate}
\end{pro}
\begin{proof}
    For statement \eqref{it:Aconv} we observe that by Lemma \ref{lem:BootstrapDerivGravFieldAALInf} and \eqref{eq:chars_def} we have
	\begin{equation}\label{eq:dtAt_bd}
		\left|  \dfrac{d}{dt}\hat{A}_t \right|  \lesssim |\partial_{\theta}\tilde{\Psi}\circ\Phi_t| \lesssim \dfrac{\varepsilon^2}{\delta^2}\dfrac{1}{(1+t)^{3/2}}.
	\end{equation}
	Here, we used that $ \Phi_t $ maps $ \D(\delta) $ into $ \D(\delta/2) $ by Corollary \ref{cor:nonlin_chars}. Hence, $ \hat{A}_t $ converges locally uniformly to a continuous function $ \hat{A}_\infty $ on $ \D(\delta) $. In particular, the above estimate yields the inequality in \eqref{it:Aconv}.
	
    In order to prove \eqref{it:Finfty}, we first show that for all $ (\theta,a,\ell)\in \D(\delta) $ we have
	\begin{align}\label{eq:Sec4:ProofConvergActionFlow}
		\dfrac{1}{2}\leq \partial_{a}\hat{A}_\infty(\theta,a,\ell)\leq \dfrac{3}{2}.
	\end{align}
	In order to prove this inequality, note that by Lemma \ref{lem:BootstrapDerivGravFieldAALInf} the flow $ \Phi_t(\theta,a,\ell) $ is differentiable with respect to the initial condition $ (\theta,a,\ell) $. Differentiating the characteristic system \eqref{eq:chars_def} then yields
	\begin{align*}
		\dfrac{d}{dt} \partial_{a}\hat{\Theta}_t &= \lambda\partial_{a \theta}^2\tilde{\Psi}\circ \Phi_t \, \partial_{a}\hat{\Theta}_t + \lambda\partial_{a}^2\tilde{\Psi}\circ \Phi_t \, \partial_{a}\hat{A}_t,
		\\
		\dfrac{d}{dt} \partial_{a}\hat{A}_t &= -\lambda\partial_{\theta}^2\tilde{\Psi}\circ \Phi_t \, \partial_{a}\hat{\Theta}_t - \lambda\partial_{a\theta}^2\tilde{\Psi}\circ \Phi_t \, \partial_{a}\hat{A}_t.
	\end{align*}
	Using the estimates in Lemma \ref{lem:BootstrapDerivGravFieldAALInf} together with the bound $M'_2(\gamma)\lesssim\varepsilon^2 \ln^{32}\jabr{t} $, which follows from \eqref{eq:mom_bd_a} and \eqref{eq:mom_bd_theta}, we obtain that
	\begin{align*}
		\dfrac{d}{dt} \left( \partial_{a}\hat{\Theta}_t \right)^2  &\lesssim C_0\dfrac{\varepsilon^2}{\delta^6} \left( |\partial_{a}\hat{\Theta}_t|^2 \dfrac{\ln^{32}\jabr{t}}{(1+t)^{3/2}} + \omega^{-2}|\partial_{a}\hat{\Theta}_t||\partial_{a}\hat{A}_t|\dfrac{\ln^{32}\jabr{t}}{(1+t)} \right),
		\\
		\dfrac{d}{dt} \left( \partial_{a}\hat{A}_t \right)^2  &\lesssim C_0\dfrac{\varepsilon^2}{\delta^4} \left(  \omega^2|\partial_{a}\hat{\Theta}_t||\partial_{a}\hat{A}_t|\dfrac{\ln^{32}\jabr{t}}{(1+t)^{2}} + |\partial_{a}\hat{A}_t|^2\dfrac{\ln^{32}\jabr{t}}{(1+t)^{3/2}} \right),
	\end{align*}
    where $\omega=a^{\frac12}\left(a+\jabr{l}^{\frac12}\right)^{-\frac12}$ as in the proof of Theorem \ref{thm:WellPosedLagrangianSol} and Theorem \ref{thm:WellPosednessRegularSol}. As in the proof of Theorem \ref{thm:WellPosednessRegularSol}, if $\eps\leq \tilde{c}_0\delta^3$ and $\tilde{c}_0>0$ sufficiently small, a bootstrap argument yields the bounds
	\begin{align*}
		|\omega\partial_a\hat{\Theta}_t|^2\leq 2C_0\varepsilon^2\delta^{-6}\ln^{33}\jabr{t}, \quad |\omega^{-1}\partial_a\hat{A}_t|^2\leq 1+2C_0\varepsilon^2\delta^{-4}.
	\end{align*}
	In particular, we obtain
	\begin{align*}
		|\partial_a\hat{A}_t|\lesssim \omega\left( 1+ C\eps\delta^{-2}\right) \lesssim 1,
	\end{align*}
	provided $ \eps\leq\tilde{c}_0 \delta^3$ is small enough. Writing $ A_t(\theta,a,\ell)=a+\Omega_t(\theta,a,\ell) $, by \eqref{eq:Aconv} we know that $ \Omega_t $ converges uniformly to some $ \Omega_\infty $, and $ \hat{A}_\infty=a+\Omega_\infty $. Moreover, we have that $ |\partial_{a}\Omega_\infty|\lesssim \varepsilon^2\delta^{-4} $, which yields the claim \eqref{eq:Sec4:ProofConvergActionFlow} by choosing $ \tilde{c}_0$ small enough: since
    \begin{align*}
		\dfrac{d}{dt}\partial_{a}\Omega_t= \dfrac{d}{dt} \partial_{a}\hat{A}_t = -\lambda\partial_{\theta}^2\tilde{\Psi}\circ \Phi_t \, \partial_{a}\hat{\Theta}_t - \lambda\partial_{a\theta}^2\tilde{\Psi}\circ \Phi_t \, \partial_{a}\hat{A}_t,
	\end{align*}
    with the above estimates we have
    \begin{align*}
		\left| \dfrac{d}{dt}\partial_{a}\Omega_t\right| &\lesssim \dfrac{\varepsilon^2}{\delta^4}\left( \omega|\partial_{a}\hat{\Theta}_t|\dfrac{\ln^{32}\jabr{t}}{(1+t)^{2}} + |\partial_{a}\hat{A}_t|\dfrac{\ln^{32}\jabr{t}}{(1+t)^{3/2}} \right) \lesssim \dfrac{\varepsilon^2}{\delta^4}\dfrac{\ln^{32}\jabr{t}}{(1+t)^{3/2}},
	\end{align*}	
    so it suffices to use that $\partial_a\Omega_0=\partial_a A_0-1=0$.
    
    \textit{Definition of $\F_\infty.$} We now define $ \F_\infty $ as
	\begin{equation}\label{eq:def_Finfty}
		\F_\infty(a) := -\dfrac{1}{a^2}\int_{\R}\int_{0}^\infty\int_{0}^\infty\ind_{\{\hat{A}_\infty(\theta,\alpha,\ell)\leq a\}} \, \gamma_0^2(\theta,\alpha, \ell) \, d\theta d\alpha d\ell.
	\end{equation}
	As a consequence, by \eqref{eq:Aconv} we have
	\begin{align*}
		\left| t^2\F_{\eff}(t,at) - \F_{\infty}(a)\right| &\leq \dfrac{1}{a^2}\int_{\R}\int_{0}^\infty\int_{0}^\infty\ind_{\{\hat{A}_\infty\leq a\leq \hat{A}_t\}\cup\{\hat{A}_t\leq a\leq \hat{A}_\infty\}} \, \gamma_0^2(\theta,\alpha, \ell) \, d\theta d\alpha d\ell
		\\
		&\lesssim \dfrac{\eps^2}{\delta^2}\dfrac{1}{a^2}\dfrac{1}{(1+t)^{1/2}},
	\end{align*}
    since for any fixed $(\theta,\ell)$ we have
    \begin{equation*}
      \left|\{\alpha:\hat{A}_\infty(\theta,\alpha,\ell)\leq a\leq \hat{A}_t(\theta,\alpha,\ell)\}\right|+\left|\{\alpha:\hat{A}_t(\theta,\alpha,\ell)\leq a\leq \hat{A}_\infty(\theta,\alpha,\ell)\}\right| \leq 2 \norm[L^\infty(\D(\delta))]{\hat{A}_t-\hat{A}_\infty}.
    \end{equation*}
    This shows \eqref{eq:Finfty_conv}.
	
	\textit{Bounds on $\F_\infty$.} From the definition of $ \F_\infty $ it is clear that
	\begin{align*}
		|\F_\infty(a)|\leq\dfrac{1}{a^2}\norm[L^2]{\gamma_0}^2\lesssim \frac{\eps^2}{a^2}.
	\end{align*}
	On the other hand, we have
	\begin{align*}
		\F_\infty'(a) &= \dfrac{2}{a^3}\int_{\R}\int_{0}^\infty\int_{0}^\infty\ind_{\{\hat{A}_\infty(\theta,\alpha,\ell)\leq a\}} \, \gamma_0^2 \, d\theta d\alpha d\ell - \dfrac{1}{a^2}\dfrac{d}{da} \int_{\R}\int_{0}^\infty\int_{0}^\infty\ind_{\{\hat{A}_\infty(\theta,\alpha,\ell)\leq a\}} \, \gamma_0^2\, d\theta d\alpha d\ell
		\\
		&=:I_1+I_2.
	\end{align*}
	Clearly, we have $ |I_1|\lesssim \varepsilon^2a^{-3} $. Using the change of variables $ \alpha\mapsto \hat{A}_\infty(\theta,\alpha,\ell) $, which is well-defined by \eqref{eq:Sec4:ProofConvergActionFlow}, we obtain 
	\begin{align*}
		|I_2|\lesssim \dfrac{1}{a^2}\int_{\R}\int_{0}^\infty \, \gamma_0^2\left( \theta,\hat{A}_\infty^{-1}(\theta,a,\ell) ,\ell\right) \, d\theta d\ell \lesssim \dfrac{\varepsilon^2}{a^2}.
	\end{align*}
	This concludes the proof.
\end{proof}
We can now prove the modified scattering dynamics for Lagrangian solutions.

\begin{thm}[Modified scattering for Lagrangian solutions]\label{thm:ModifiedScatteringLagrangianSol}
    There exists $c_2>0$ such that if $ \gamma $ is a Lagrangian solution as in Theorem \ref{thm:WellPosedLagrangianSol} with $\eps\leq c_2\delta^3$, and $ \F_\infty $ is as in Proposition \ref{pro:ConvergenceCharacFlow},
    then there is a continuous map $ \Phi_\infty:\D(\delta)\to \D(\delta/2) $ such that the characteristic flow $\Phi_t=(\hat{\Theta}_t,\hat{A}_t,\ell)$ converges as
    \begin{equation}\label{eq:char_modscat}
		\Phi_t + (\lambda\ln(1+t)\F_{\infty}(\hat{A}_t),0,0) \to \Phi_\infty,\quad t\to\infty.
	\end{equation}
    Furthermore, there exists $ \gamma_\infty\in L^2\cap L^\infty $ such that
	\begin{equation}\label{eq:gamma_Lagrange_modscat}
		\lim_{t\to \infty} \gamma\left( t, \theta-\lambda\ln(1+t)\F_\infty(a),a,\ell \right) = \gamma_\infty(\theta,a,\ell)
	\end{equation}
	in the sense of distributions.
\end{thm}
We note that thanks to the convergence \eqref{eq:char_modscat} of the characteristics and the definition of $\F_\infty$ in \eqref{eq:def_Finfty}, we can express $\F_\infty$ in terms of $\gamma_\infty$ as
\begin{equation}
    \F_\infty(a) = -\dfrac{1}{a^2}\int_{\R}\int_{0}^\infty\int_{0}^\infty\ind_{\{\alpha\leq a\}} \, \gamma_\infty^2(\theta,\alpha, \ell) \, d\theta d\alpha d\ell.
\end{equation}

\begin{rem}
  One can quantify the above distributional convergence of $\gamma$: the proof of Theorem \ref{thm:ModifiedScatteringLagrangianSol} in fact shows that
  \begin{align*}
		\mathcal{W}_1\left( \hat{\gamma}(t)^2,\gamma_\infty^2 \right) \lesssim \dfrac{\eps^4}{\delta^4}\dfrac{\ln^{16}\jabr{t}}{(1+t)^{1/2}}, \quad \hat{\gamma}(t,\theta,a,\ell):=\gamma\left( t, \theta-\lambda\ln(1+t)\F_\infty(a),a,\ell \right),
	\end{align*}
	with respect to the $ 1 $-Wasserstein distance $ \mathcal{W}_1 $, for example.
\end{rem}

\begin{proof}[Proof of Theorem \ref{thm:ModifiedScatteringLagrangianSol}]
By the convergence $\hat{A}_t\to\hat{A}_\infty$ in Proposition \ref{pro:ConvergenceCharacFlow} \eqref{it:Aconv}, it suffices to prove that
	\begin{align*}
		\Xi(t) := \hat{\Theta}_t+\lambda\ln(1+t)\F_{\infty}(\hat{A}_t) \to \hat{\Theta}_\infty, \quad t\to\infty,
	\end{align*}
locally uniformly for some $ \hat{\Theta}_\infty $, and then let $ \Phi_\infty = (\hat{\Theta}_\infty,\hat{A}_\infty,\ell) $. To this end, we observe that by \eqref{eq:Sec4:NonlinCharSys}
	\begin{align*}
		\dfrac{d}{dt}\Xi(t) = -\lambda\left[ \F(t,\tilde{R}) \partial_{a}\tilde{R} - \dfrac{1}{1+t}\F_{\infty}(a) \right] \circ\Phi_t + \lambda\ln(1+t) \F_{\infty}'(\hat{A}_t)\dfrac{d}{dt}\hat{A}_t =: \mathcal{R}_1+\mathcal{R}_2.
	\end{align*}
From \eqref{eq:dtAt_bd} in the proof of Proposition \ref{pro:ConvergenceCharacFlow} and \eqref{eq:Finfty_bd} we obtain that for all $ (\theta,a,\ell)\in \D(\delta) $ 
	\begin{align*}
		|\mathcal{R}_2(t)| \lesssim \dfrac{\jabr{\ell}^{\frac32}\eps^4}{\delta^5}\dfrac{\ln(1+t)}{(1+t)^{3/2}},
	\end{align*}
having used also that $ \hat{A}_t\jabr{\ell}^{\frac12}\gtrsim \delta $. On the other hand, we have
	\begin{align*}
		\mathcal{R}_1 &= -\lambda\left[ \F(t,\tilde{R})\left( \partial_{a}\tilde{R} - t \right)\right]\circ\Phi_t  + \lambda t\left[ -\F(t,\tilde{R})+\F_{\eff}(t,\tilde{R}) \right]\circ\Phi_t
		\\
		&\quad +\lambda t\left[ -\F_{\eff}(t,\tilde{R})+\F_{\eff}(t,at) \right]\circ\Phi_t+\dfrac{\lambda}{1+t} \left[ -(1+t)t\F_{\eff}(t,at)+\F_\infty(a)\right]\circ\Phi_t
		\\
		&=: T_1+T_2+T_3+T_4.
	\end{align*}   
	We estimate now term by term. We use Lemmas \ref{lem:EstDynRadBulk}, \ref{lem:EstDynRadNoBulk} and Lemma \ref{lem:EstimatesGravFieldLInf} to get
	\begin{align*}
		|T_1\circ\Phi_t^{-1}| &\lesssim \ind_{\mathcal{B}} \dfrac{\varepsilon^2}{(1+t)^2}\jabr{\dfrac{p}{a}}\ln\jabr{t} + \ind_{\mathcal{B}^c} \dfrac{1}{a}\dfrac{\varepsilon^2}{1+t}\jabr{\dfrac{p}{\tilde{R}}} \lesssim \dfrac{1}{\delta^3}\dfrac{\varepsilon^2\ln\jabr{t}}{(1+t)^2}\jabr{\ell}^{\frac32} + \dfrac{1}{\delta^4}\dfrac{\varepsilon^2}{(1+t)^2}\jabr{\dfrac{1}{\ell}}\jabr{\ell}^{\frac12}\jabr{\theta}
		\\
		&\lesssim \dfrac{1}{\delta^4}\dfrac{\varepsilon^2\ln\jabr{t}}{(1+t)^2}\jabr{\dfrac{1}{\ell}} \jabr{\ell}^{\frac32}\jabr{\theta}.
	\end{align*}
	Here we used a decomposition into the contribution inside and outside the bulk $ \mathcal{B} $, see \eqref{eq:Sec2:DefBulk}, and in particular that outside the bulk we have $ 1\leq2|\theta|/ta $. Provided that $\eps_0^2\delta^3\lesssim 1$, we thus get with \eqref{eq:nonlin_char_bds} from Lemma \ref{lem:NonlinCharSys} (for $ \nu=1/2 $) that for all $ (\theta,a,\ell)\in \D(\delta) $
	\begin{align*}
		|T_1| \lesssim \dfrac{\eps^2}{\delta^4}\dfrac{\ln^2\jabr{t}}{(1+t)^2}\jabr{\dfrac{1}{\ell}} \jabr{\ell}^{2}\jabr{\theta}.
	\end{align*}
	By Lemma \ref{lem:GravFieldvsEffField} and the bounds \eqref{eq:mom_bd_a}, \eqref{eq:mom_bd_theta} we have for $ t\geq 1 $ and $ (\theta,a,\ell)\in \D(\delta) $
	\begin{align*}
		|T_2|\lesssim \dfrac{t\varepsilon^2}{\tilde{R}^2+t^2}\dfrac{\ln^{16}\jabr{t}}{(1+t)^{1/2}}\leq \dfrac{\varepsilon^2\ln^{16}\jabr{t}}{(1+t)^{3/2}}.
	\end{align*}
	Furthermore, for $ T_3 $ we use Lemma \ref{lem:EstimatesDerivGravEffField} in conjunction with Lemma \ref{lem:EffectiveFieldAA} and the estimates \eqref{eq:nonlin_char_bds} in Lemma \ref{lem:NonlinCharSys}, to obtain that for all $ (\theta,a,\ell)\in \D(\delta) $
	\begin{align*}
		|T_3| \lesssim \dfrac{\varepsilon^2t\ln^3\jabr{t}}{(1+t)^{3}}\left[\jabr{\ell}^{1/2}\jabr{a}\jabr{\dfrac{1}{a^2}}\jabr{\theta}\ln\jabr{t}\right]\circ \Phi_t \lesssim  \dfrac{\eps^2}{\delta^2}\dfrac{\ln^{5}\jabr{t}}{(1+t)^{2}}\jabr{\ell}^{2}\jabr{a}\jabr{\theta}.
	\end{align*}
	Finally, with Proposition \ref{pro:ConvergenceCharacFlow} \eqref{it:Finfty} we have for all $ (\theta,a,\ell)\in \D(\delta) $
	\begin{align*}
		|T_4| \lesssim \dfrac{\eps^2}{\delta^4}\dfrac{1}{(1+t)^{3/2}}\jabr{\ell}.
	\end{align*}
	Altogether, we conclude that
	\begin{align*}
		\left| \dfrac{d}{dt}\Xi(t) \right| \lesssim \dfrac{\eps^2}{\delta^4}\dfrac{\ln^{16}\jabr{t}}{(1+t)^{3/2}}\jabr{\theta}\jabr{a}\jabr{\dfrac{1}{\ell}}\jabr{\ell}^{5/2},
	\end{align*}
	whence $ \Xi(t;a,\theta,a,\ell) $ converges locally uniformly to a continuous map $ \hat{\Theta}_\infty(\theta,a,\ell) $ on $ \D(\delta) $, and we have for all $ (\theta,a,\ell)\in \D(\delta) $ that
	\begin{align*}
		\left| \Xi(t)-\hat{\Theta}_\infty \right| = \left| \hat{\Theta}_t+\lambda\ln(1+t)\F_{\infty}(\hat{A}_t) -\hat{\Theta}_\infty \right|  \lesssim \dfrac{\eps^2}{\delta^4}\dfrac{\ln^{16}\jabr{t}}{(1+t)^{1/2}}\jabr{\theta}\jabr{a}\jabr{\dfrac{1}{\ell}}\jabr{\ell}^{5/2}.
	\end{align*}
	Consequently, with Proposition \ref{pro:ConvergenceCharacFlow} \eqref{it:Aconv} we get
	\begin{align*}
		\left| \Phi_t + (\lambda\ln(1+t)\F_{\infty}(\hat{A}_t),0,0) - \Phi_\infty \right| \lesssim \dfrac{\eps^2}{\delta^4}\dfrac{\ln^{16}\jabr{t}}{(1+t)^{1/2}}\jabr{\theta}\jabr{a}\jabr{\dfrac{1}{\ell}}\jabr{\ell}^{5/2}.
	\end{align*}
	Defining the function $ \gamma_\infty:= (\Phi_\infty)_\# \gamma_0$ yields then for all test functions $ \varphi\in \D $
	\begin{align*}
		\skp{\varphi}{\hat{\gamma}(t)} = \skp{\varphi\circ(\Xi(t),\hat{A}_t,\ell)}{\gamma_0} \to \skp{\varphi\circ\Phi_\infty}{\gamma_0} = \skp{\varphi}{\gamma_\infty}, \quad t\to\infty.
	\end{align*}
\end{proof}

\medskip
\paragraph{\bf Long-time behavior: strong solutions.} Here we prove that the weak convergence in Theorem \ref{thm:ModifiedScatteringLagrangianSol} can in fact be strengthened to strong convergence for $ \SobH^1 $-regular solutions as in Theorem \ref{thm:WellPosednessRegularSol}. We give an independent proof of this, leaning on the methods of \cite{PW2020}, to illustrate the ease with which this result follows for strong solutions. As before, an essential (but here rather simple) result towards the proof is that macroscopic quantities only depending on the actions $ a $ converge as $ t\to \infty $.
\begin{pro}\label{pro:ConvergenceMacroQuant}
	Let $ \gamma $ be as in Theorem \ref{thm:WellPosednessRegularSol}. 
    Then for any $ \psi=\psi(a) \in L^\infty(0,\infty) $ the limit
	\begin{align*}
		\lim_{t\to\infty}\skp{\psi}{\gamma^2(t)}_{L^2} =:\jabr{\psi}_\infty
	\end{align*}
	exists and
	\begin{equation}\label{eq:a_av_conv}
		\left| \skp{\psi}{\gamma^2(t)}_{L^2} - \jabr{\psi}_\infty\right|\lesssim \dfrac{\varepsilon^4}{\delta^4}\dfrac{\ln^{32}\jabr{t}}{(1+t)^{1/2}}\norm[L^\infty]{\psi}.
	\end{equation}
\end{pro}
\begin{proof}
	Using the equation for $ \gamma $, we have
	\begin{align*}
		\dfrac{d}{dt}\skp{\psi}{\gamma^2} &= 2\lambda\skp{\psi\gamma}{\{\tilde{\Psi},\gamma\}} = \lambda\skp{\psi}{\{\tilde{\Psi},\gamma^2\}} = - \lambda\skp{\{\tilde{\Psi},\psi\}}{\gamma^2} = -\lambda\skp{\partial_{\theta}\tilde{\Psi}\, \partial_{a}\psi}{\gamma^2}
		\\
		&= \lambda\skp{\psi}{\partial_{a\theta}^2\tilde{\Psi}\gamma^2+2\gamma \partial_{\theta}\tilde{\Psi} \partial_{a}\gamma}.
	\end{align*}
	With the bounds in Lemma \ref{lem:BootstrapDerivGravFieldAALInf} and Theorem \ref{thm:WellPosednessRegularSol} we then have
	\begin{align*}
		\left| \dfrac{d}{dt}\skp{\psi}{\gamma^2}\right| \lesssim \dfrac{\varepsilon^4}{\delta^4}\dfrac{\ln^{32}\jabr{t}}{(1+t)^{3/2}}\norm[L^\infty]{\psi} + \dfrac{\varepsilon^4}{\delta^2}\dfrac{\ln^{33}\jabr{t}}{(1+t)^{2}}\norm[L^\infty]{\psi},
	\end{align*}
	which yields the claim upon integrating in time.
\end{proof}
As an immediate corollary of the previous result we have the following.
\begin{cor}\label{cor:ConvEffField}
	Let $ \gamma $ be as in Theorem \ref{thm:WellPosednessRegularSol}. Then the following pointwise limit exists
	\begin{align*}
		\lim_{t\to\infty} t^2\F_{\eff}(t,at) =: \F_\infty(a),
	\end{align*}
	and we have for all $ a>0 $ that
	\begin{align*}
		\left| t^2\F_{\eff}(t,at)-\F_\infty(a) \right| \lesssim \dfrac{\varepsilon^4}{\delta^4}\dfrac{1}{a^2}\dfrac{\ln^{32}\jabr{t}}{(1+t)^{1/2}}.
	\end{align*}
\end{cor}
\begin{proof}
	We apply Proposition \ref{pro:ConvergenceMacroQuant} to the function $ \psi_a(\alpha)=a^{-2}\ind_{\{\alpha\leq a\}} $ for fixed $ a>0 $, thereby getting
	\begin{align*}
		\lim_{t\to \infty} t^2\F_{\eff}(t,at) = \lim_{t\to \infty} -\dfrac{1}{a^2} \int_{\R} \int_{0}^\infty \int_{0}^\infty\ind_{\{\alpha\leq a\}} \gamma^2(t,\theta,\alpha,\ell)\, d\theta d\alpha d\ell =:\F_\infty(a).
	\end{align*}
	For the asserted estimate we use $ \norm[L^\infty]{\psi_a}\leq 1/a^2 $ and the bound \eqref{eq:a_av_conv} from Proposition \ref{pro:ConvergenceMacroQuant}.
\end{proof}

We can now prove our main result concerning the long-time behavior.
\begin{thm}[Modified scattering for strong solutions]\label{thm:ModifiedScattering}
	 Let $\gamma$ be as in Theorem \ref{thm:WellPosednessRegularSol}, with $k_1\geq 4$, $k_2\geq 2$, $k_3\geq 2$, $k_4\geq 0$ and $k_5\geq \frac{k_1}{2}$, and $ \F_\infty $ be given via Corollary \ref{cor:ConvEffField}.
    Then there exists $ \gamma_\infty\in L^2 $ such that
	\begin{align*}
		\lim_{t\to \infty} \gamma\left( t, \theta-\lambda\ln(1+t)\F_\infty(a),a,\ell \right) = \gamma_\infty(\theta,a,\ell)
	\end{align*}
	with respect to the $ L^2 $-convergence. More precisely, we have
	\begin{align*}
		\norm[L^2]{\hat{\gamma}(t)-\gamma_\infty} \lesssim \dfrac{\varepsilon^3}{\delta^4}\dfrac{\ln^{35}\jabr{t}}{(1+t)^{1/2}}, \quad \hat{\gamma}(t,\theta,a,\ell):=\gamma\left( t, \theta-\lambda\ln(1+t)\F_\infty(a),a,\ell \right).
	\end{align*}
\end{thm}
\begin{proof}
	We write $ \S_t(\theta,a,\ell)=(t, \theta-\lambda\ln(1+t)\F_\infty(a),a,\ell) $ for brevity, and use the equation for $ \gamma $ to get
	\begin{align*}
		\partial_t\hat{\gamma} &= \partial_t\gamma\circ\S_t -\dfrac{\lambda}{1+t} \F_\infty(a) \partial_{\theta}\gamma\circ\S_t = -\lambda\{\gamma,\tilde{\Psi}\}\circ\S_t -\dfrac{\lambda}{1+t} \F_\infty(a) \partial_{\theta}\gamma\circ\S_t
		\\
		&=-\lambda\left[ \partial_{a}\tilde{\Psi}\circ\S_t+\dfrac{1}{1+t} \F_\infty(a) \right]\partial_{\theta}\gamma\circ\S_t+\lambda\partial_{\theta}\tilde{\Psi}\circ\S_t \, \partial_{a}\gamma\circ\S_t
		\\
		&=-\lambda (\mathcal{R}_1\, \partial_{\theta}\gamma)\circ\S_t+\lambda(\mathcal{R}_2 \, \partial_{a}\gamma)\circ\S_t,
	\end{align*}
	where
	\begin{align*}
		\mathcal{R}_1 := \partial_{a}\tilde{\Psi}+\dfrac{1}{1+t} \F_\infty(a)=-\F(t,\tilde{R})\partial_a \tilde{R}+\dfrac{1}{1+t} \F_\infty(a), \quad \mathcal{R}_2:=\partial_{\theta}\tilde{\Psi}.
	\end{align*}
	This yields for $ s\leq t $ that
	\begin{align*}
		\norm[L^2]{\hat{\gamma}(t)-\hat{\gamma}(s)} &\leq \lambda\int_{s}^t\norm[L^2]{(\mathcal{R}_1\, \partial_{\theta}\gamma)\circ\S_r+(\mathcal{R}_2\S_r\, \partial_{a}\gamma)\circ\S_r}\, dr
		\\
		&= \lambda\int_{s}^t\norm[L^2]{\mathcal{R}_1(r)\,  \partial_{\theta}\gamma(r)+\mathcal{R}_2(r) \, \partial_{a}\gamma(r)}\, dr.
	\end{align*}
We decompose
	\begin{align*}
		\mathcal{R}_1 &= -\F(t,\tilde{R})\partial_{a}\tilde{R}+\dfrac{1}{1+t}\F_\infty(a) = -\F(t,\tilde{R})\left[ \partial_{a}\tilde{R} -t \right] + t\left[ -\F(t,\tilde{R})+\F_{\eff}(t,\tilde{R}) \right]
		\\
		&\quad +t\left[ -\F_{\eff}(t,\tilde{R})+\F_{\eff}(t,at) \right]+\dfrac{1}{1+t} \left[ -(1+t)t\F_{\eff}(t,at)+\F_\infty(a)\right]
		\\
		&=: T_1+T_2+T_3+T_4,
	\end{align*}
	and estimate term by term. Firstly, we have with Lemmas \ref{lem:EstDynRadBulk}, \ref{lem:EstDynRadNoBulk} and \ref{lem:EstimatesGravFieldLInf}
	\begin{align*}
		|T_1| &\lesssim \dfrac{\eps^2}{1+t^2}|\partial_{a}\tilde{R} -t| \lesssim \ind_{\mathcal{B}} \dfrac{\eps^2}{(1+t)^2}\jabr{\dfrac{p}{a}}\ln\jabr{t} + \ind_{\mathcal{B}^c} \dfrac{1}{a}\dfrac{\eps^2}{1+t}\jabr{\dfrac{p}{\tilde{R}}} 
		\\
		&\lesssim \dfrac{\eps^2}{\delta^3}\dfrac{\ln\jabr{t}}{(1+t)^2}\jabr{\ell}^2 + \dfrac{\eps^2}{\delta^4}\dfrac{1}{(1+t)^2}\jabr{\dfrac{1}{\ell}}\jabr{\ell}\jabr{\theta} \lesssim \dfrac{\eps^2}{\delta^4}\dfrac{\ln\jabr{t}}{(1+t)^2}\jabr{\dfrac{1}{\ell}}\jabr{\ell}^2\jabr{\theta}.
	\end{align*}
	Here we used a decomposition into the contribution inside and outside the bulk $ \mathcal{B} $, see \eqref{eq:Sec2:DefBulk}, and the fact that outside the bulk there holds $ 1\leq2|\theta|/ta $. Furthermore, we have with Lemma \ref{lem:GravFieldvsEffField} and $ t\geq1 $
	\begin{align*}
		|T_2|\lesssim \dfrac{\eps^2}{\tilde{R}+t}\dfrac{\ln^{16}\jabr{t}}{(1+t)^{1/2}}\leq \dfrac{\eps^2}{(1+t)^{3/2}}.
	\end{align*}
	In addition, by Lemma \ref{lem:EstimatesDerivGravFieldLInf} in conjunction with Lemma \ref{lem:EffectiveFieldAA} we get
	\begin{align*}
		|T_3| &\lesssim \dfrac{\eps^2}{\delta^2}\dfrac{t\ln^{32}\jabr{t}}{(1+t)^{3}}\jabr{\ell}^{3/2}\jabr{a}\jabr{\theta}\ln\jabr{t}.
	\end{align*}
	Finally, we have with Corollary \ref{cor:ConvEffField} and the fact that $ a\jabr{\ell}^{\frac12}\gtrsim\delta $
	\begin{align*}
		|T_4|\lesssim \dfrac{\eps^4}{\delta^6}\dfrac{\jabr{\ell}\ln^{32}\jabr{t}}{(1+t)^{3/2}}.
	\end{align*}
	
	Using the derivative bounds \eqref{eq:deriv_bd1}, \eqref{eq:deriv_bd2} satisfied by $ \gamma $ according to Theorem \ref{thm:WellPosednessRegularSol}, it follows that 
	\begin{align*}
		\norm[L^2]{\mathcal{R}_1(t)\partial_{\theta}\gamma(t)} &\lesssim \dfrac{\eps^2}{\delta^4}\dfrac{\ln^{33}\jabr{t}}{(1+t)^{3/2}}\norm[L^2]{\jabr{\dfrac{1}{\ell}}\jabr{\ell}^2\jabr{a}\jabr{\theta} \partial_{\theta}\gamma(t)} \lesssim \dfrac{\eps^3}{\delta^4}\dfrac{\ln^{35}\jabr{t}}{(1+t)^{3/2}}.
	\end{align*}
	In addition, Lemma \ref{lem:BootstrapDerivGravFieldAALInf} gives that
	\begin{align*}
		\norm[L^2]{\mathcal{R}_2(t)\partial_{a}\gamma(t)} = \norm[L^2]{\partial_{\theta}\tilde{\Psi}\partial_{a}\gamma(t)}\lesssim \dfrac{\eps^2}{\delta^2}\dfrac{1}{(1+t)^{3/2}}\norm[L^2]{\omega\jabr{\ell}^{1/4}\partial_{a}\gamma(t)} \lesssim \dfrac{\eps^3}{\delta^2}\dfrac{\ln^{33}\jabr{t}}{(1+t)^{3/2}}.
	\end{align*}
	This yields that
	\begin{align*}
		\norm[L^2]{\hat{\gamma}(t)-\hat{\gamma}(s)} \lesssim \dfrac{\eps^3}{\delta^4}\dfrac{\ln^{35}\jabr{s}}{(1+s)^{1/2}}, \quad s\leq t.
	\end{align*}
	The existence of the limit $ \lim_{t\to \infty}\hat{\gamma}(t)=:\gamma_\infty $ follows, and we have the bound
	\begin{align*}
		\norm[L^2]{\hat{\gamma}(t)-\gamma_\infty} \lesssim \dfrac{\eps^3}{\delta^4}\dfrac{\ln^{35}\jabr{t}}{(1+t)^{1/2}}.
	\end{align*}
	This concludes the proof.
\end{proof}

\addtocontents{toc}{\protect\setcounter{tocdepth}{0}}
\section*{Acknowledgments}
The authors gratefully acknowledge support of the SNSF through grant PCEFP2\_203059 and the NCCR SwissMAP.
\addtocontents{toc}{\protect\setcounter{tocdepth}{1}}

\appendix
\section{Derivation of the equation for radially symmetric data}\label{sec:AppendixDerivRadSymEq}
In this section we give a derivation of equation \eqref{eq:Sec1:NonDynPhysVar} satisfied for radially symmetric data $ f(t,\bx,\bv) $. More precisely we look at equation \eqref{eq:VPD} with $ \cX=\cV=0 $
\begin{align*}
	\left(  \partial_t + \bv \cdot\nabla_\bx - \dfrac{m}{2} \dfrac{\bx}{|\bx|^3} \cdot \nabla_\bv \right) f - \lambda \nabla_\bx \phi_g \cdot \nabla_\bv f =0,\quad \Delta \phi_g = 4\pi\int f \, d\bv.
\end{align*}
Observe that the equation remains invariant under transformations $ (\bx,\bv) \mapsto (O\bx,O\bv) $ for all $ O\in \mathcal{O}(3) $. Thus, we look for distributions $ f(t,\bx,\bv) $ invariant under such transformations. This implies that $ f $ only depends on $ r=|\bx| $ (by moving $ \bx $ to e.g. $ r(1,0,0) $). In each tangent space at a point $ \bx $ the distribution $ f $ then only depends on the parallel component of $ \bv $ to $ \bx $ and the length of the normal component (one can use a rotation with axis of rotation along $ \bx=r(1,0,0) $), i.e.
\begin{align*}
	f(t,\bx,\bv) = \mu(t,r(\bx),v(\bx,\bv),\ell(\bx,\bv))^2, \quad r(\bx) =|\bx|,\quad v(\bx,\bv):=\bv\cdot \dfrac{\bx}{|\bx|},\quad  \ell(\bx,\bv)=|\bx \wedge \bv|^2,
\end{align*}
Here, we defined the normal part using the length of the angular momentum $ \ell $. For the local density we obtain by using cylindrical coordinates in $ \bv $ with axis $ \bx $, using the abbreviation $ \hat{\bx}=\bx/r $,
\begin{align*}
	4\pi\int_{\R^3} f(t,\bx,\bv) \, d\bv &= 4\pi\int_{\R}dv\int_{w\perp \hat{\bx}}dw \, f\left(t,\bx,v \,  \hat{\bx} + w\right) =\dfrac{4\pi}{r^2}\int_{\R}dv\int_{w\perp \hat{\bx}}dw \, f\left(t,\bx,v \,  \hat{\bx} + r^2w\right)
	\\
	&= \dfrac{4\pi^2}{r^2} \int_{\R}dv\int_0^\infty d\ell \, \mu(t,r,v,\ell)^2,
\end{align*}
having used polar coordinates in the variable $ w $ and the fact that $ r^2|w|^2=\ell $. Furthermore, using spherical coordinates in $ \bx $ we obtain
\begin{align*}
	f(t,\bx,\bv) \, d\bx d\bv = 4\pi^2\mu(t,r,v,\ell)^2 \, drdvd\ell
\end{align*}
In order to use our $ L^2 $-based functional framework we write this in the form
\begin{align*}
	f(t,\bx,\bv) \, d\bx d\bv = \mu(t,r,v,\ell)^2\, drdvd\ell, \quad 4\pi\int_{\R^3_p} f(t,\bx,\bv) \, d\bv = \dfrac{1}{r^2} \int_{\R}dv\int_0^\infty d\ell \,  \mu(t,r,v,\ell)^2.
\end{align*}

\medskip
\paragraph{\bf Gravitational potential.} Due to the invariance of $ f $ with respect to rotations the induced gravitational potential is also radially symmetric, in particular we have $ \phi_g(t,\bx)=\psi(t,r) $. The Poisson equation for the potential of the gas can now be written using spherical coordinates
\begin{align*}
	\dfrac{1}{r^2}\partial_r \left( r^2\partial_r \psi(t,r)\right) = \dfrac{1}{r^2} \int_{\R}dv\int_0^\infty d\ell \,  \mu(t,r,v,\ell)^2 =:\dfrac{1}{r^2}\rho(t,r)
\end{align*}
This yields
\begin{align*}
	\psi(t,r) = \int_0^r\dfrac{ds}{s^2}\int_0^s dz \, \rho(t,z) + C = \int_0^rdz \int_z^r \dfrac{ds}{s^2} \, \rho(t,z) + C = \int_0^rdz \left( \dfrac{1}{z}-\frac{1}{r} \right)  \, \rho(t,z) + C.
\end{align*}
We choose the constant $ C $ to ensure that the potential vanishes at infinity. This yields
\begin{align*}
	\psi(t,r) &= \int_0^rdz \left( \dfrac{1}{z}-\frac{1}{r} \right)  \, \rho(t,z) - \int_0^\infty\dfrac{dz}{z} \rho(t,z) = -\int_0^\infty\dfrac{ds}{\max(r,s)}\rho(t,s)
	\\
	&=-\int_0^\infty ds\int_{\R}dv\int_0^\infty d\ell \,  \dfrac{\mu(t,s,v,\ell)^2}{\max(r,s)}.
\end{align*}
Furthermore, for the force we obtain
\begin{align*}
	\F(t,r) = -\partial_r \psi(t,r) = -\dfrac{1}{r^2}\int_0^r ds\int_{\R}dv\int_0^\infty d\ell \, \mu(t,s,v,\ell)^2.
\end{align*}

\medskip
\paragraph{\bf Equation for distribution.} For the equation of $ \mu $ we make use of the following identities
\begin{align*}
	\bv \cdot \nabla_\bx r = v, \quad \bv\cdot \nabla_\bx v &= \bv^\top \left( \dfrac{I-\hat{\bx}\otimes \hat{\bx}}{|\bx|} \right)\bv = \dfrac{\ell}{r^3}, \quad \bv\cdot \nabla_\bx \ell = 2(\bx\times \bv)\cdot (\bv\times \bv) =0,
	\\
	\dfrac{\bx}{r}\cdot \nabla_\bv v &= 1, \quad \dfrac{\bx}{r}\cdot  \nabla_\bv \ell =  \dfrac{2}{r}(\bx\wedge \bv)\cdot (\bx\times \bx) =0,
\end{align*}
and thereby obtain the equation
\begin{align*}
	\left( \partial_t + v\partial_r+ \dfrac{\ell}{r^3}\partial_v - \dfrac{m}{2} \dfrac{1}{r^2}\partial_v \right)\mu- \lambda\partial_r\psi \, \partial_v\mu =0.
\end{align*}
This can be derived using the weak formulation by testing with functions only depending on $ (r,v,\ell) $. This yields the equation \eqref{eq:Sec1:NonDynPhysVar}.

\bibliographystyle{habbrv}
\bibliography{vp-lib}

\end{document}